\documentclass{amsart} 
\usepackage[utf8]{inputenc}
\usepackage{amssymb}
\usepackage{amsmath}
\usepackage{amsthm}
\usepackage{mathtools}
\usepackage{multirow}
\usepackage{thmtools}

\usepackage[colorlinks = true, 
			linkcolor=blue, 
			urlcolor=blue,
			citecolor=blue,  
			anchorcolor=blue,
			backref=page]{hyperref}
			
\usepackage[pdftex]{graphicx}

\newtheorem{lemma}{Lemma}[section]
\newtheorem{theorem}{Theorem}[section]
\newtheorem{corollary}{Corollary}[section]

\newtheorem{example}{Example}[section]
\theoremstyle{definition}
\newtheorem{definition}{Definition}[section]

\usepackage{tikz}
\usetikzlibrary{
  knots,
  hobby,
  decorations.pathreplacing,
  shapes.geometric,
  calc
}
\tikzset{
  knot diagram/every strand/.append style={
    thick,
    black
  },
  show curve controls/.style={
    postaction=decorate,
    decoration={show path construction,
      curveto code={
        \draw [blue, dashed]
        (\tikzinputsegmentfirst) -- (\tikzinputsegmentsupporta)
        node [at end, draw, solid, black, inner sep=1pt]{};
        \draw [blue, dashed]
        (\tikzinputsegmentsupportb) -- (\tikzinputsegmentlast)
        node [at start, draw, solid, black, inner sep=1pt]{}
        node [at end, fill, blue, ellipse, inner sep=1pt]{}
        ;
      }
    }
  },
  show curve endpoints/.style={
    postaction=decorate,
    decoration={show path construction,
      curveto code={
        \node [fill, blue, ellipse, inner sep=1pt] at (\tikzinputsegmentlast) {}
        ;
      }
    }
  }
}

\begin{document}

\title[Polynomials of complete spatial graphs]{Polynomials of complete spatial graphs \\ and Jones polynomial of related links}

\author[O.~Oshmarina]{Olga~Oshmarina}
\address{Department of Mechanics and Mathematics, 
Novosibirsk State University, Novosibirsk, 630090, Russia 
\& 
 Regional Scientific and Educational Mathematical Center, Tomsk State University, Tomsk, 634050, Russia}
\email{\href{mailto:o.oshmarina@g.nsu.ru}{o.oshmarina@g.nsu.ru}}

\author[A.~Vesnin]{Andrei~Vesnin} 
\address{Sobolev Institute of Mathematics, Russian Academy of Sciences, Novosibirsk, 630090, Russia  
\&
Department of  Mechanics and Mathematics, 
Novosibirsk State University, Novosibirsk, 630090, Russia  
\& Regional Scientific and Educational Mathematical Center, Tomsk State University, Tomsk, 634050, Russia}
\email{\href{mailto:vesnin@math.nsc.ru}{vesnin@math.nsc.ru}}

\begin{abstract}
Let $\mathbb K_n$ be a complete graph with $n$ vertices. An embedding of $\mathbb K_n$ in $S^3$ is called a spatial $\mathbb K_n$-graph. Knots in a spatial $\mathbb K_n$-graph corresponding to simple cycles of $\mathbb K_n$ are said to be constituent knots.  We consider the case $n=4$. The boundary of an oriented  band surface with zero Seifert form, constructed for a spatial $\mathbb K_4$, is a four-component associated link.  There are obtained relations between normalized Yamada and Jaeger polynomials of spatial graphs and Jones polynomials of constituent knots and the associated link. 
\end{abstract} 

\thanks{This work was supported by the Ministry of Science and Higher Education of Russia (agreement no. 075-02-2024-1437).} 


\renewcommand{\baselinestretch}{1.2}

\date{\today}

\maketitle

%\tableofcontents

\section{Introduction}

Throughout this paper we work in the piecewise-linear category.  We consider a graph as a topological space as well as a combinatorial object both. Let $G = (V, E)$ be a graph with finite set $V$ of vertices and finite set $E$ of edges. An embedding  $f : G \to S^3$, or, similarly, $f : G \to \mathbb R^3$, is called a \emph{spatial embedding of graph} $G$, and $\mathcal G = f(G)$ is called a \emph{spatial graph} or, more concrete,  a \emph{spatial $G$-graph}. If $\gamma$ is a simple cycle in $G$ then its spatial embedding $f(\gamma)$ is a knot in $S^3$. Moreover, if $\lambda = \alpha_1 \cup \ldots \cup \alpha_n$ is a collection of $n$ pairwise disjoint cycles in $G$, then its spatial embedding $f(\lambda)$ is an $n$-component link in $S^3$, where all cycles assumed to be \emph{simple}, that is without repeating of vertices (except for the beginning and ending vertices). Thus, the theory of spatial graphs is a natural generalization of the theory of knots and links. 

In~\cite{Ya} Yamada introduced a polynomial $Y(\mathcal G)$ for a spatial graph $\mathcal G$, which is referred to as the \emph{Yamada} polynomial now. It is a Laurent polynomial in the sense that negative degrees are allowed. The Yamada polynomial is the most common and studied invariant of spatial graphs in the last years. In~\cite{Mu} Murakami described a relation of this polynomial with knit algebras.  An upper bound of span of the Yamada polynomial via a number of crossings in a spatial graph diagram was obtained in~\cite{MOT}. In~\cite{LLLV18} it was described how the Yamada polynomial changes under replacing edges by subgraphs. It was demonstrated in~\cite{LLLV19} that roots of Yamada polynomials form a dense set on a complex plane. In~\cite{Yos} Yoshinaga introduced a polynomial invariant of spatial graphs which can be reduced to the Yamada polynomial~\cite{DV}. In~\cite{Yok} Yokota introduced invariants which correspond to Yamada polynomials for trivalent graphs.  A generalization of the Yamada polynomial for virtual spatial graphs was constructed in~\cite{DJK}.  An interesting application of the Yamada polynomial in engineering sciences was demonstrated in~\cite{Pe}. 

We will work with the Yamada polynomial $Y(D)$ and  the Jaeger polynomial $\mathfrak{J} (D)$ defined for a diagram $D$ of a spatial graph $\mathcal G$, as well as with their nomalizations $\widetilde{Y} (D)$ and $\widetilde{\mathfrak{J}} (D)$ defined for the case when $D$ is a diagram of a spatial $\mathbb K_4$-graph $\mathcal G$. 

Let us note that a combinatorically simplest case of a spatial graph is the case when $G$ consists of two vertices connected by three edges. Such a graph $G$ is called a \emph{$\theta$-graph} (a theta-graph) and any its embedding to $S^3$ (or $\mathbb R^3$) is called a spatial \emph{$\theta$-graph} (also called a theta-curve, see~\cite{KSWZ, Wo}). Spatial $\theta$-graphs were tabulated according to increasing of number of crossings in diagrams. The first ten spatial $\theta$-graphs were presented by Simon~\cite{Si}. Yamada polynomials of these graphs can be found  in~\cite{VD}. A table of spatial $\theta$-graphs admitting diagrams with at most seven crossings was obtained by Moriuchi~\cite{Mo}.  A relation between the normalized Jaeger polynomial of a spatial $\theta$-graph and Jones polynomial of related knots and links was obtained by Huh~\cite{H}.    
 
In the present paper we will demonstrate  relations between the normalized Yamada polynomial and the normalized Jaeger polynomial for spatial $\mathbb K_4$-graphs on the one side, and the Jones polynomials for knots and links related to the spatial graph on the other side.   
 
The structure of the paper is as follows. In Section~\ref{section:preliminaries} we describe generalized Reidemeister moves for diagrams of spatial graphs, some results on knots and links in spatial complete graphs, and the concept of an oriented band for a spatial graph diagram. In Section~\ref{section:polynomials} we recall definitions of the Jones polynomial and the Dubrovnik polynomial for knots and links, as well as definitions of the Yamada polynomial and the Jaeger polynomial for spatial graphs.  In Section~\ref{section:theta} we  discuss results on invariants of spatial $\theta$-graphs and related links, obtained in~\cite{H}. In Section~\ref{section:K4} there is obtained the formula for the Jones polynomial of the associated link of a spatial $\mathbb K_4$-graph, see formula~(\ref{eqn:Jonesmain}). In Section~\ref{section:Jaeger} we define a normalized Jaeger polynomial $\widetilde{\mathfrak{J}} (\Omega)$ for a spatial $\mathbb K_4$-graph $\Omega$ and prove in Theorem~\ref{theorem4.1} that it is an invariant of $\Omega$. In Lemma~\ref{lemma4.1} we prove a relation between normalized Jaeger polynomial of $\Omega$ and bracket polynomials of band diagrams of $\theta$-graphs and cycles contained in $\Omega$. The main results of the paper are presented in Section~\ref{section:results}. 
Theorem~\ref{theorem:K4} shows a relation between normalized Jaeger polynomials of a spatial $\mathbb K_4$-graph, spatial $\theta$-graphs corresponding to its subgraphs, its constituent knots and the Jones polynomial of the associated link. Corollary~\ref{cor6.1} presents the similar relation for normalized Yamada polynomials and Jones polynomial.  Corollary~\ref{cor6.2} presents the relation between the normalized Jaeger polynomial of a spatial $\mathbb K_4$-graph and Jones polynomials of related links. In Section~\ref{section:examples} there are presented two examples that illustrate discussed invariants and relations.  

\section{Preliminaries and notations} \label{section:preliminaries}

\subsection{Generalized Reidemeister moves}

Kauffman~\cite{K89} and Yamada~\cite{Ya} independently introduced the notion of a diagram of a spatial graph which generalizes the notion of a knot diagram. For two spatial graphs $\mathcal G$ and $\mathcal G'$, if there exists an isotopy $h_t : S^3 \to S^3$, $t \in [0,1]$, such that $h_0 = id$ and $h_1 (\mathcal G) = \mathcal G'$, then we say that  $\mathcal G$ and $\mathcal G'$  are \emph{ambient isotopic}. If for each vertex $v$ of $\mathcal G$. there exist a neighborhood $U_v$ of $v$ and a plane $P_v$ such that $\mathcal G \cup U_v \subset P_v$, then we say that $\mathcal G$ is a \emph{flat vertex graph}. For two flat vertex graphs $\mathcal G$ and $\mathcal G'$, if there exists an isotopy $h_t : S^3 \to S^3$, $t \in [0,1]$, such that $h_0 = id$ and $h_1 (\mathcal G) = \mathcal G'$, and $h_t (\mathcal G)$ are flat vertex graphs for each $t \in [0,1]$, then we say that $\mathcal G$ and $\mathcal G'$ are ambient isotopic as flat vertex graph, or shortly, \emph{flatly isotopic}.

%The first one is RV-isotopy for graphs with rigid vertices (also called flat vertices) and the second is ambient isotopy for graphs with pliable vertices. Two graphs are said to be \emph{equivalent} if they are RV-isotopic or ambient isotopic, respectively. Following~\cite{KSWZ}, the equivalence class of a spatial graph $\mathcal G$ is called the \emph{knot type} of $\mathcal G$.

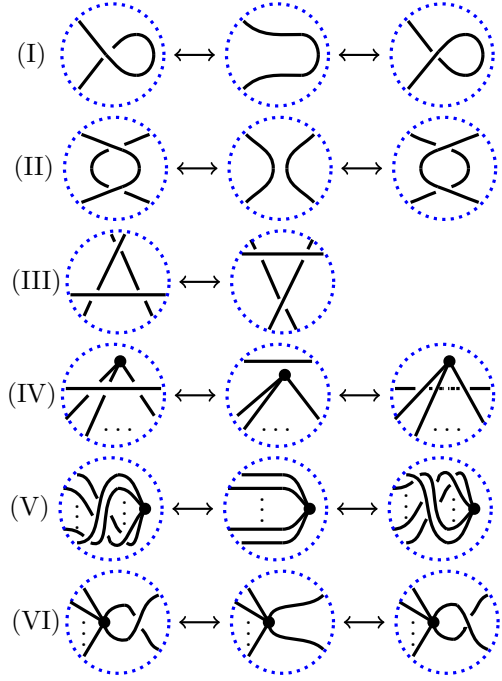
\begin{figure}[h]
\begin{center}
%(I)
\begin{tikzpicture} [scale=0.9]
\node at (-0.7,0.5) {(I)};
\draw[black, very thick] (0,1) -- (0.4,0.5); 
\draw[black, very thick] (0.4,0.5) to [out=315, in=180] (0.85,0.2);
\draw[black, very thick] (0.85,0.2) to [out=0, in=270] (1.1,0.5);
\draw[black, very thick] (1.1,0.5) to [out=90, in=0] (0.85,0.8);
\draw[black, very thick] (0.85,0.8) to [out=180, in=45] (0.5,0.6);
%\draw[black, very thick, rounded corners] (0.,1) -- (0.35,0.55) -- (0.45,0.45) -- (0.6,0.3) --(0.95,0.2) -- (1.1,0.5) -- (0.95,0.8) -- (0.6,0.7) -- (0.45, 0.55);
\draw[black, very thick] (0.35,0.45) -- (0,0);
\draw[blue, dotted, very thick] (0.5,0.5) circle (0.75);
\draw[thick, <->, black] (1.4,0.5) -- (2,0.5);
%\node at (1.6,0.5) {${\bf \longrightarrow}$};
\end{tikzpicture} 
\begin{tikzpicture} [scale=0.9]
\draw[black, very thick] (0,0) to [out=45, in=180] (0.85,0.2); 
\draw[black, very thick] (0.85,0.2) to [out=0, in=270] (1.1,0.5); 
\draw[black, very thick] (1.1,0.5) to [out=90, in=0] (0.85,0.8); 
\draw[black, very thick] (0.85,0.8) to [out=180, in=315] (0,1); 
\draw[blue, dotted, very thick] (0.5,0.5) circle (0.75);
\draw[thick, <->, black] (1.4,0.5) -- (2,0.5);
\end{tikzpicture}
\begin{tikzpicture} [scale=0.9]
    \draw[black, very thick] (0,1) -- (0.35,0.55); 
\draw[black, very thick] (0.45,0.45) to [out=315, in=180] (0.85,0.2);
\draw[black, very thick] (0.85,0.2) to [out=0, in=270] (1.1,0.5);
\draw[black, very thick] (1.1,0.5) to [out=90, in=0] (0.85,0.8);
\draw[black, very thick] (0.85,0.8) to [out=180, in=45] (0.45,0.55);
\draw[black, very thick] (0,0) -- (0.45,0.55); 
\draw[blue, dotted, very thick] (0.5,0.5) circle (0.75);
\end{tikzpicture}

\smallskip 
%(II)
\begin{tikzpicture}[scale=0.9]
\node at (-0.7,0.5) {(II)};
\draw[black, very thick] (0,0) to [out=20, in=200] (0.6,0.23); 
\draw[black, very thick] (0.6,0.23) to [out=20,in=270] (0.85,0.5);
\draw[black, very thick] (0.85,0.5) to [out=90, in=340] (0.6,0.77); 
\draw[black, very thick] (0.6,0.77) to [out=160,in=340] (0,1);
\draw[black, very thick] (1,1) to [out=200,in=20] (0.63,0.85);
\draw[black, very thick] (0.4,0.77) to [out=200,in=90] (0.15,0.5);
\draw[black, very thick] (0.15,0.5) to [out=270,in=160] (0.4,0.23);
\draw[black, very thick] (0.63,0.15) to [out=340,in=160] (1,0);
\draw[blue, dotted, very thick] (0.5,0.5) circle (0.75);
\draw[thick, <->, black] (1.4,0.5) -- (2,0.5);
\end{tikzpicture}
\begin{tikzpicture}[scale=0.9]
\draw[black, very thick] (0,0) to [out=45, in=270] (0.4,0.5);
\draw[black, very thick] (0.4,0.5) to [out=90, in=315] (0,1); 
\draw[black, very thick] (1,1) to [out=225,in=90] (0.6,0.5);
\draw[black, very thick] (0.6,0.5) to [out=270,in=135] (1,0);
\draw[blue, dotted, very thick] (0.5,0.5) circle (0.75);
\draw[thick, <->, black] (1.4,0.5) -- (2,0.5);
\end{tikzpicture}
\begin{tikzpicture}[scale=0.9]
\draw[black, very thick] (1,0) to [out=160, in=340] (0.4,0.23); 
\draw[black, very thick] (0.4,0.23) to [out=160,in=270] (0.15,0.5);
\draw[black, very thick] (0.15,0.5) to [out=90, in=200] (0.4,0.77); 
\draw[black, very thick] (0.4,0.77) to [out=20,in=200] (1,1);
\draw[black, very thick] (0,1) to [out=330,in=160] (0.37,0.85);
\draw[black, very thick] (0.6,0.77) to [out=340,in=90] (0.85,0.5);
\draw[black, very thick] (0.85,0.5) to [out=270,in=20] (0.6,0.23);
\draw[black, very thick] (0.37,0.15) to [out=200,in=20] (0,0);
\draw[blue, dotted, very thick] (0.5,0.5) circle (0.75);
\end{tikzpicture}
 
\smallskip 
%(III)
\begin{tikzpicture} [scale=0.9]
\node at (-0.7,0.5) {(III)};
\draw[black, very thick] (0,0) -- (0.1,0.22);
\draw[black, very thick] (0.2,0.38) -- (0.6,1.2);
\draw[black, very thick] (-0.2,0.31) -- (1.2,0.31);
\draw[black, very thick] (0.4,1.2) -- (0.45,1.09);
\draw[black, very thick] (0.55,0.95) -- (0.8,0.38);
\draw[black, very thick] (0.9,0.22) -- (1,0);
\draw[blue, dotted, very thick] (0.5,0.5) circle (0.75);
\draw[thick, <->, black] (1.4,0.5) -- (2,0.5);
\end{tikzpicture}
\begin{tikzpicture} [scale=0.9]
\draw[black, very thick] (0.3,-0.2) -- (0.8,0.82);
\draw[black, very thick] (0.87,0.98) -- (0.94,1.12);
\draw[black, very thick] (-0.1,0.91) -- (1.1,0.91);
\draw[black, very thick] (0.7,-0.2) -- (0.55,0.1);
\draw[black, very thick] (0.45,0.25) -- (0.2,0.82);
\draw[black, very thick] (0.1,0.98) -- (0.03,1.12);
\draw[blue, dotted, very thick] (0.5,0.5) circle (0.75);
\draw[thick, <->, white] (1.4,0.5) -- (2,0.5);
\end{tikzpicture}
\begin{tikzpicture} [scale=0.9]
\draw[white, dotted, very thick] (0.5,0.5) circle (0.75);
\end{tikzpicture}

\smallskip 
%(IV)
\begin{tikzpicture}[scale=0.9]
\node at (-0.7,0.5) {(IV)};
\draw[black, very thick] (-0.2,0.6) -- (1.2,0.6);
\filldraw [line width=2pt, black] (0.6,1) circle[radius=0.05cm];
\draw[black, very thick] (0.6,1) -- (0.3,0.7);
\draw[black, very thick] (0.13,0.5) -- (-0.2,0.15);
\draw[black, very thick] (0.6,1) -- (0.45,0.7);
\draw[black, very thick] (0.37,0.5) -- (0.1,-0.1);
\draw[black, very thick] (0.6,1) -- (0.8,0.7);
\draw[black, very thick] (0.9,0.5) -- (1.1,0.2);
\node at (0.6,0) {$\dots$};
\draw[blue, dotted, very thick] (0.5,0.5) circle (0.75);
\draw[thick, <->, black] (1.4,0.5) -- (2,0.5);
\end{tikzpicture}
\begin{tikzpicture}[scale=0.9]
\filldraw [line width=2pt, black] (0.6,0.8) circle[radius=0.05cm];
\draw[black, very thick] (0,1) -- (1,1);
\draw[black, very thick] (0.6,0.8) -- (0,0);
\draw[black, very thick] (0.6,0.8) -- (-0.1,0.2);
\draw[black, very thick] (0.6,0.8) -- (1.1,0.15);
\node at (0.5,0) {$\dots$};
\draw[blue, dotted, very thick] (0.5,0.5) circle (0.75);
\draw[thick, <->, black] (1.4,0.5) -- (2,0.5);
\end{tikzpicture}
\begin{tikzpicture}[scale=0.9]
\filldraw [line width=2pt, black] (0.6,1) circle[radius=0.05cm];
\draw[black, very thick] (0.6,1) -- (-0.2,0.15);
\draw[black, very thick] (0.6,1) -- (0.1,-0.1);
\draw[black, very thick] (0.6,1) -- (1.1,0.1);
\node at (0.6,0) {$\dots$};
\draw[black, very thick] (-0.2,0.6) -- (0.1,0.6);
\draw[black, very thick] (0.3,0.6) -- (0.35,0.6);
\draw[black, very thick] (0.45,0.6) -- (0.5,0.6);
\draw[black, very thick] (0.51,0.6) -- (0.52,0.6);
\draw[black, very thick] (0.58,0.6) -- (0.6,0.6);
\draw[black, very thick] (0.62,0.6) -- (0.68,0.6);
\draw[black, very thick] (0.7,0.6) -- (0.75,0.6);
\draw[black, very thick] (0.9,0.6) -- (1.2,0.6);
\draw[blue, dotted, very thick] (0.5,0.5) circle (0.75);
\end{tikzpicture}
\smallskip 
%\\ 

%(V)
\begin{tikzpicture}[scale=0.9]
\node at (-0.7,0.5) {(V)};
\filldraw [line width=2pt, black] (1,0.5) circle[radius=0.05cm];
\draw[black, very thick] (1,0.5) to [out=120, in=0] (0.6,1);
\draw[black, very thick] (0.6,1) to [out=200, in=0] (0,0);
\draw[black, very thick] (1,0.5) to [out=140, in=0] (0.6,0.8);
\draw[black, very thick] (0.6,0.8) to [out=210, in=340] (0.2,0);
\draw[black, very thick] (0.1,0.1) to [out=140, in=0] (-0.2,0.2);
\draw[black, very thick] (1,0.5) to [out=220, in=320] (0.5,0);
\draw[black, very thick] (0.5,0) to [out=140, in=320] (0.45,0.1);
\draw[black, very thick] (0.23,0.4) to [out=120, in=0] (-0.2,0.8);
\draw[black, very thick] (1,0.5) to [out=240, in=320] (0.73,0);
\draw[black, very thick] (0.65,0.15) to [out=140, in=300] (0.5,0.4);
\draw[black, very thick] (0.3,0.65) to [out=120, in=0] (0,1);
\node at (0,0.5) {$\vdots$};
\node at (0.7,0.55) {$\vdots$};
\draw[blue, dotted, very thick] (0.5,0.5) circle (0.75);
\draw[thick, <->, black] (1.4,0.5) -- (2,0.5);
\end{tikzpicture}
\begin{tikzpicture}[scale=0.9]
\filldraw [line width=2pt, black] (1,0.5) circle[radius=0.05cm];
\draw[black, very thick] (1,0.5) to [out=120, in=0] (0.6,1);
\draw[black, very thick] (0.6,1)--(0,1);
\draw[black, very thick] (1,0.5) to [out=140, in=0] (0.6,0.8);
\draw[black, very thick] (0.6,0.8)--(-0.2,0.8);
\draw[black, very thick] (1,0.5) to [out=220, in=0] (0.6,0.2);
\draw[black, very thick] (0.6,0.2)--(-0.2,0.2);
\draw[black, very thick] (1,0.5) to [out=240, in=0] (0.6,0);
\draw[black, very thick] (0.6,0)--(0,0);
\node at (0.3,0.6) {$\vdots$};
\draw[blue, dotted, very thick] (0.5,0.5) circle (0.75);
\draw[thick, <->, black] (1.4,0.5) -- (2,0.5);
\end{tikzpicture}
\begin{tikzpicture}[scale=0.9]
\filldraw [line width=2pt, black] (1,0.5) circle[radius=0.05cm];
\draw[black, very thick] (1,0.5) to [out=240, in=0] (0.6,0);
\draw[black, very thick] (0.6,0) to [out=160, in=0] (0,1);
\draw[black, very thick] (1,0.5) to [out=220, in=0] (0.6,0.2);
\draw[black, very thick] (0.6,0.2) to [out=150, in=20] (0.2,1);
\draw[black, very thick] (0.1,0.9) to [out=220, in=0] (-0.2,0.8);
\draw[black, very thick] (1,0.5) to [out=160, in=40] (0.5,1);
\draw[black, very thick] (0.5,1) to [out=220, in=40] (0.45,0.9);
\draw[black, very thick] (0.23,0.6) to [out=240, in=0] (-0.2,0.2);
\draw[black, very thick] (1,0.5) to [out=120, in=40] (0.73,1);
\draw[black, very thick] (0.65,0.85) to [out=220, in=60] (0.5,0.6);
\draw[black, very thick] (0.3,0.35) to [out=240, in=0] (0,0);
\node at (0,0.65) {$\vdots$};
\node at (0.7,0.6) {$\vdots$};
\draw[blue, dotted, very thick] (0.5,0.5) circle (0.75);
\end{tikzpicture}
\smallskip  

%(VI)
\begin{tikzpicture}[scale=0.9]
\node at (-0.7,0.5) {(VI)};
\filldraw [line width=2pt, black] (0.3,0.5) circle[radius=0.05cm];
\draw[black, very thick] (0.3,0.5) to [out=270, in=200] (0.65,0.25);
\draw[black, very thick] (0.65,0.25) to [out=45, in=200] (1.1,0.9);
\draw[black, very thick] (0.3,0.5) to [out=90, in=180] (0.6,0.75);
\draw[black, very thick] (0.6,0.75) to [out=0, in=135] (0.75,0.6);
\draw[black, very thick] (0.85,0.4) to [out=315, in=160] (1.1,0.1);
\draw[blue, dotted, very thick] (0.5,0.5) circle (0.75);
\draw[black, very thick] (0.3,0.5) -- (0,1);
\draw[black, very thick] (0.3,0.5) -- (-0.2,0.8);
\draw[black, very thick] (0.3,0.5) -- (0,0);
\draw[thick, <->, black] (1.4,0.5) -- (2,0.5);
\node at (0,0.45) {$\vdots$};
\end{tikzpicture}
\begin{tikzpicture}[scale=0.9]
\filldraw [line width=2pt, black] (0.3,0.5) circle[radius=0.05cm];
\draw[black, very thick] (0.3,0.5) to [out=90, in=225] (1.1,0.9);
\draw[black, very thick] (0.3,0.5) to [out=270, in=135] (1.1,0.1);
\draw[black, very thick] (0.3,0.5) -- (0,1);
\draw[black, very thick] (0.3,0.5) -- (-0.2,0.8);
\draw[black, very thick] (0.3,0.5) -- (0,0);
\node at (0,0.45) {$\vdots$};
\draw[blue, dotted, very thick] (0.5,0.5) circle (0.75);
\draw[thick, <->, black] (1.4,0.5) -- (2,0.5);
\end{tikzpicture}
\begin{tikzpicture}[scale=0.9]
\filldraw [line width=2pt, black] (0.3,0.5) circle[radius=0.05cm];
\draw[black, very thick] (0.3,0.5) to [out=90, in=160] (0.65,0.75);
\draw[black, very thick] (0.65,0.75) to [out=315, in=160] (1.1,0.1);
\draw[black, very thick] (0.3,0.5) to [out=270, in=180] (0.6,0.25);
\draw[black, very thick] (0.6,0.25) to [out=0, in=45] (0.75,0.4);
\draw[black, very thick] (0.85,0.6) to [out=45, in=200] (1.1,0.9);
\draw[black, very thick] (0.3,0.5) -- (0,1);
\draw[black, very thick] (0.3,0.5) -- (-0.2,0.8);
\draw[black, very thick] (0.3,0.5) -- (0,0);
\node at (0,0.45) {$\vdots$};
\draw[blue, dotted, very thick] (0.5,0.5) circle (0.75);
\end{tikzpicture}
\end{center} \caption{Generalized Reidemeister moves for diagrams of spatial graphs.} \label{fig-R}
\end{figure}

Let us consider local moves of diagrams presented in Fig.~\ref{fig-R}. These moves are referred to as \emph{generalized Reidemeister moves} for diagrams of spatial graphs since the first three of them are classical  Reidemeister moves for diagrams of knots and links. Generalized Reidemeister moves play an important role in the theory of spatial graphs because of the following result. 

\begin{theorem} \cite{K89, Ya} \label{theorem:Reidemeister}
Two spatial graphs  $\mathcal G$ and $\mathcal G'$ are isotopic (resp. flatly isotopic) if and only if a diagram of $\mathcal G$ can be transformed to a diagram of $\mathcal G'$ by a finite sequence of moves among (I)--(VI) (resp. (I)--(V)).
\end{theorem}

%\section{Spatial $\text{K}_4$-graphs}

\subsection{Spatial complete graphs}

Recall that a graph without loops and multiply edges is said to be \emph{complete} if any two vertices are connected by an edge. 
Let $\mathbb K_n$, $n \geq 3$, denotes a complete graph with $n$ vertices, thus $\mathbb K_n$ has $\frac{n(n-1)}{2}$ edges. Let $\mathbb K_{m,n}$ denotes a complete bipartite graph with $m$-part and $n$-part. The following properties of spatial embeddings of complete graphs are known. Conway and Gordon~\cite{CG} proved that each spatial $\mathbb K_7$-graph contains a cycle which forms a non-trivial knot. In~\cite{Sh} Shimabara proved that each spatial $\mathbb K_{5,5}$-graph contains a cycle which forms a non-trivial knot. Conway and Gordon~\cite{CG} and Sachs~\cite{Sa} proved that each spatial $\mathbb K_6$-graph contains a pair of cycles which form an unsplittable 2-component link. Sachs~\cite{Sa} proved that  each spatial $\mathbb K_{4,4}$-graph contains a pair of cycles which form an unsplittable 2-component link. Flapan, Naimi and Pommersheim~\cite{FNP} proved that each spatial $\mathbb K_{10}$-graph contains an unsplittable link of three components and also exhibited an embedding of $\mathbb K_9$ with no such link of three components. O'Donnol~\cite{O} proved that for $n > 1$, each embedding of $\mathbb K_{[\frac{7}{2} n]}$ contains an unsplittible link of $n$-components. Drummond-Cole and O'Donnol~\cite{DO} proved that for $n > 1$ each embedding of $\mathbb K_{2n+1, 2n+1}$ contains an unsplittible link of $n$ components.  

In the present paper we study spatial embeddings of a complete graph $\mathbb K_4$, i.e., spatial $\mathbb K_4$-graphs. As well as diagrams of  knots and links, diagrams of spatial graphs also can be tabulated in order of increasing of number of crossings. The diagrams of spatial $\mathbb  K_4$-graphs with at most four crossings from~\cite{Si} are presented in Fig.~\ref{fig:K4}. These diagrams correspond to ten spatial $\mathbb K_4$-graphs which we denote by $\Omega_1, \ldots, \Omega_{10}$. 
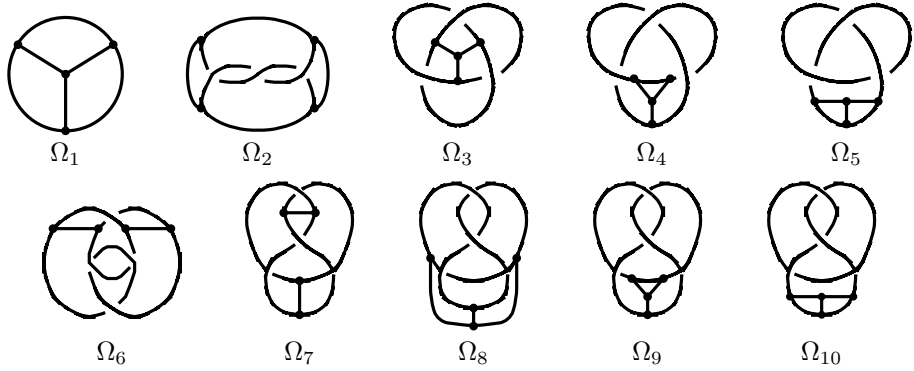
\begin{figure}[h] 
\begin{center}
%1%
\begin{tikzpicture}[scale=0.03] 
%\draw [very thick, black, fill] (20,30) circle[radius=1.2cm];
%\draw [very thick, black, fill] (80,30) circle[radius=1.2cm];
\draw [very thick, black, fill] (50,50) circle[radius=1.2cm];
\draw[black, very thick] (50,50) -- (50,25);
\draw [very thick, black, fill] (50,25) circle[radius=1.2cm];
\draw[black, very thick] (50,50) -- (29,63);
\draw [very thick, black, fill] (29,63) circle[radius=1.2cm];
\draw[black, very thick] (50,50) -- (71,63);
\draw [very thick, black, fill] (71,63) circle[radius=1.2cm];
%\draw [very thick, black, fill] (50,80) circle[radius=1.2cm];
%\draw[black, very thick] (20,30) -- (80,30); 
%\draw[black, very thick] (20,30) -- (50,50); 
%\draw[black, very thick] (20,30) -- (50,80); 
%\draw[black, very thick] (80,30) -- (50,80); 
%\draw[black, very thick] (80,30) -- (50,50); 
%\draw[black, very thick] (50,80) -- (50,50); 
\draw[black, very thick] (50,50) circle (25);
\node at (50,15) {$\Omega_1$};
\end{tikzpicture}
\qquad 
% 2 %
\begin{tikzpicture}[scale=0.03] 
\draw[black, very thick, rounded corners]  (53,47) -- (60,46) -- (67,47) -- (70,50);
\draw[black, very thick, rounded corners] (50,50) -- (47,47) -- (40,46) -- (33,47);
\draw[black, very thick, rounded corners] (47,53) -- (40,54) -- (33,53) -- (30,50);
\draw[black, very thick, rounded corners] (50,50) -- (53,53) -- (60,54) -- (67,53);
\draw[black, very thick, rounded corners]  (30,50) -- (28,46) -- (25,40) -- (25,35) -- (30,30) -- (40,26)  -- (50,25) -- (60,26) -- (70,30) -- (75,35) -- (75,40) -- (72,48);
\draw[black, very thick, rounded corners] (70,50) -- (72,52) -- (75,60) -- (75,65) -- (70,70) -- (60,74) -- (50,75) -- (40,74) -- (30,70) -- (25,65) -- (25,60) -- (28,52); 
 \draw [very thick, black]  (75,35) arc (-90:90: 6 and 15);
  \draw [very thick, black]  (25,35) arc (270:90: 6 and 15);
\draw [very thick, black, fill] (75,35) circle[radius=1.2cm];
\draw [very thick, black, fill] (25,35) circle[radius=1.2cm];
\draw [very thick, black, fill] (75,65) circle[radius=1.2cm];
\draw [very thick, black, fill] (25,65) circle[radius=1.2cm]; 
\node at (50,15) {$\Omega_2$};
\end{tikzpicture}
\qquad 
% 3 % 
\begin{tikzpicture}[scale=0.03] 
\draw[black, very thick, rounded corners] (60.62,48.40) -- (55.16,46.94) -- (50.20,46.34) -- (45.73,46.61) -- (40.40,47.38) -- (35.57,49.01) -- (30.74,50.64) -- (26.91,54.01) -- (23.58,58.24) --  (22.48,62.34) -- (21.52,68.67) -- (23.15,73.50) -- (26.52,77.33) -- (32.48,79.66) -- (36.58,80.76) -- (41.91,79.99) -- (46.74,78.36);
\draw[black, very thick, rounded corners] (36.38,53.60) -- (37.84,59.06) -- (39.80,63.66) -- (42.27,67.39) -- (45.60,71.62) -- (49.43,74.99) -- (53.26,78.36) -- (58.09,79.99) -- (63.42,80.76) -- (67.51,79.66) -- (73.48,77.33) -- (76.84,73.50) -- (78.48,68.67) -- (77.52,62.34) -- (76.42,58.24) -- (73.09,54.01) -- (69.26,50.65);
\draw[black, very thick, rounded corners]  (53,72) -- (57,68) -- (60,64) -- (62,60) -- (64,55) -- (65,50) -- (66,45) -- (65,40) -- (63,35) -- (60,32) -- (55,28) -- (50,27) -- (45,28) -- (40,32) -- (37,35) -- (35,40) -- (34,45); 
\draw [very thick, black, fill] (50,58) circle[radius=1.2cm];
\draw [very thick, black, fill] (60,64) circle[radius=1.2cm];
\draw [very thick, black, fill] (40,64) circle[radius=1.2cm];
\draw [very thick, black, fill] (50,47) circle[radius=1.2cm];
\draw[black, very thick, rounded corners] (50,58) -- (60,64); 
\draw[black, very thick, rounded corners] (50,58) -- (40,64); 
\draw[black, very thick, rounded corners] (50,58) -- (50,47); 
\node at (50,15) {$\Omega_3$};
\end{tikzpicture}
\qquad 
% 4 % 
\begin{tikzpicture}[scale=0.03] 
\draw[black, very thick, rounded corners] (60.62,48.40) -- (55.16,46.94) -- (50.20,46.34) -- (45.73,46.61) -- (40.40,47.38) -- (35.57,49.01) -- (30.74,50.64) -- (26.91,54.01) -- (23.58,58.24) --  (22.48,62.34) -- (21.52,68.67) -- (23.15,73.50) -- (26.52,77.33) -- (32.48,79.66) -- (36.58,80.76) -- (41.91,79.99) -- (46.74,78.36);
\draw[black, very thick, rounded corners] (36.38,53.60) -- (37.84,59.06) -- (39.80,63.66) -- (42.27,67.39) -- (45.60,71.62) -- (49.43,74.99) -- (53.26,78.36) -- (58.09,79.99) -- (63.42,80.76) -- (67.51,79.66) -- (73.48,77.33) -- (76.84,73.50) -- (78.48,68.67) -- (77.52,62.34) -- (76.42,58.24) -- (73.09,54.01) -- (69.26,50.65);
\draw[black, very thick, rounded corners]  (53,72) -- (57,68) -- (60,64) -- (62,60) -- (64,55) -- (65,50) -- (66,45) -- (65,40) -- (63,35) -- (60,32) -- (55,28) -- (50,27) -- (45,28) -- (40,32) -- (37,35) -- (35,40) -- (34,45); 
\draw [very thick, black, fill] (50,38) circle[radius=1.2cm];
\draw [very thick, black, fill] (58,48) circle[radius=1.2cm];
\draw [very thick, black, fill] (42,48) circle[radius=1.2cm];
\draw [very thick, black, fill] (50,28) circle[radius=1.2cm];
\draw[black, very thick, rounded corners] (50,38) -- (58,48); 
\draw[black, very thick, rounded corners] (50,38) -- (42,48); 
\draw[black, very thick, rounded corners] (50,38) -- (50,28); 
\node at (50,15) {$\Omega_4$};
\end{tikzpicture}
\qquad 
% 5 % 
\begin{tikzpicture}[scale=0.03] 
\draw[black, very thick, rounded corners] (60.62,48.40) -- (55.16,46.94) -- (50.20,46.34) -- (45.73,46.61) -- (40.40,47.38) -- (35.57,49.01) -- (30.74,50.64) -- (26.91,54.01) -- (23.58,58.24) --  (22.48,62.34) -- (21.52,68.67) -- (23.15,73.50) -- (26.52,77.33) -- (32.48,79.66) -- (36.58,80.76) -- (41.91,79.99) -- (46.74,78.36);
\draw[black, very thick, rounded corners] (36.38,53.60) -- (37.84,59.06) -- (39.80,63.66) -- (42.27,67.39) -- (45.60,71.62) -- (49.43,74.99) -- (53.26,78.36) -- (58.09,79.99) -- (63.42,80.76) -- (67.51,79.66) -- (73.48,77.33) -- (76.84,73.50) -- (78.48,68.67) -- (77.52,62.34) -- (76.42,58.24) -- (73.09,54.01) -- (69.26,50.65);
\draw[black, very thick, rounded corners]  (53,72) -- (57,68) -- (60,64) -- (62,60) -- (64,55) -- (65,50) -- (66,45) -- (65,40) -- (63,35) -- (60,32) -- (55,28) -- (50,27) -- (45,28) -- (40,32) -- (37,35) -- (35,40) -- (34,45); 
\draw [very thick, black, fill] (50,38) circle[radius=1.2cm];
\draw [very thick, black, fill] (64,38) circle[radius=1.2cm];
\draw [very thick, black, fill] (36,38) circle[radius=1.2cm];
\draw [very thick, black, fill] (50,28) circle[radius=1.2cm];
\draw[black, very thick, rounded corners] (50,38) -- (36,38); 
\draw[black, very thick, rounded corners] (50,38) -- (64,38); 
\draw[black, very thick, rounded corners] (50,38) -- (50,28); 
\node at (50,15) {$\Omega_5$};
\end{tikzpicture}
\\
\smallskip 
% 6 % 
\begin{tikzpicture}[scale=0.03] 
\draw[black, very thick] (42,52) -- (46,56) -- (48,57) --(50,57) -- (52,57) -- (54,56) -- (58,52) -- (60,50) -- (60,46) -- (58,40) -- (56,35) -- (53,32); 
\draw[black, very thick] (58,48) -- (54,44) -- (52,43) --(50,43) -- (48,43) -- (46,44) -- (42,48) -- (40,50) -- (40,54) -- (42,60) -- (44,65) -- (47,68); 
\draw[black, very thick, rounded corners] (40,46) -- (42,40) -- (44,35) --(47,32) -- (50,30) -- (53,28) -- (58,26) -- (62,26) -- (67,28) -- (73,32) -- (76,35) -- (78,40) -- (80,46) -- (80,50) -- (80,54) -- (78,60) -- (76,65) -- (73,68) -- (67,72) -- (62,74) -- (58,74) -- (53,72); 
\draw[black, very thick, rounded corners] (60,54) -- (58,60) -- (56,65) --(53,68) -- (50,70) -- (47,72) -- (42,74) -- (38,74) -- (33,72) -- (27,68) -- (24,65) -- (22,60) -- (20,54) -- (20,50) -- (20,46) -- (22,40) -- (24,35) -- (27,32) -- (33,28) -- (38,26) -- (42,26) -- (47,28); 
\draw [very thick, black, fill] (44,65) circle[radius=1.2cm];
\draw [very thick, black, fill] (56,65) circle[radius=1.2cm];
\draw [very thick, black, fill] (24,65) circle[radius=1.2cm];
\draw [very thick, black, fill] (76,65) circle[radius=1.2cm];
\draw[black, very thick, rounded corners] (24,65) -- (44,65); 
\draw[black, very thick, rounded corners] (56,65) -- (76,65); 
\node at (50,10) {$\Omega_6$};
\end{tikzpicture}
\qquad
% 7 %  
\begin{tikzpicture}[scale=0.03] 
\draw[black, very thick, rounded corners] (52,80) -- (57,75) -- (57,70) --(52,62); 
\draw[black, very thick, rounded corners] (48,58) -- (44,55) -- (40,53) -- (37,50) -- (36,48) -- (35,44); 
\draw[black, very thick, rounded corners] (48,80) -- (43,75) -- (43,70) -- (48,62) -- (50,60) -- (52,58) -- (56,55) -- (60,53) -- (63,50) -- (64,48); 
\draw[black, very thick, rounded corners] (52,80) -- (48,83); 
\draw[black, very thick, rounded corners] (52,83) -- (55,85) --(60,85) --(67,83) --(70,79) -- (72,76) -- (74,70) -- (72,60) -- (70,55) -- (68,50) -- (66,47) -- (62,45) -- (57,43) --(53,42) -- (50,42) -- (47,42) -- (43,43) -- (38,45); 
\draw[black, very thick, rounded corners] (48,83) -- (45,85) -- (40,85) -- (33,83) -- (30,79) -- (28,76) -- (26,70) -- (28,60) -- (30,55) -- (32, 50) -- (34,48); 
\draw[black, very thick, rounded corners] (65,44) -- (65,40) -- (64,35) -- (60,29) -- (53,27) -- (50,27) -- (47,27) --(40,29) -- (36,35) -- (35,40) --(35,44); 
\draw[black, very thick, rounded corners] (43,72) -- (57,72); 
\draw[black, very thick, rounded corners] (50,27) -- (50,42); 
\draw [very thick, black, fill] (43,72) circle[radius=1.2cm];
\draw [very thick, black, fill] (57,72) circle[radius=1.2cm];
\draw [very thick, black, fill] (50,27) circle[radius=1.2cm];
\draw [very thick, black, fill] (50,42) circle[radius=1.22cm];
\node at (50,10) {$\Omega_7$};
\end{tikzpicture}
\qquad
% 8 %  
\begin{tikzpicture}[scale=0.03] 
\draw[black, very thick, rounded corners] (52,80) -- (57,75) -- (57,70) --(52,62); 
\draw[black, very thick, rounded corners] (48,58) -- (44,55) -- (40,53) -- (37,50) -- (36,48) -- (35,44); 
\draw[black, very thick, rounded corners] (48,80) -- (43,75) -- (43,70) -- (48,62) -- (50,60) -- (52,58) -- (56,55) -- (60,53) -- (63,50) -- (64,48); 
\draw[black, very thick, rounded corners] (52,80) -- (48,83); 
\draw[black, very thick, rounded corners] (52,83) -- (55,85) --(60,85) --(67,83) --(70,79) -- (72,76) -- (74,70) -- (72,60) -- (70,55) -- (68,50) -- (66,47) -- (62,45) -- (57,43) --(53,42) -- (50,42) -- (47,42) -- (43,43) -- (38,45); 
\draw[black, very thick, rounded corners] (48,83) -- (45,85) -- (40,85) -- (33,83) -- (30,79) -- (28,76) -- (26,70) -- (28,60) -- (30,55) -- (32, 50) -- (34,48); 
\draw[black, very thick, rounded corners] (65,44) -- (65,40) -- (64,35) -- (60,32) -- (53,30) -- (50,30) -- (47,30) --(40,32) -- (36,35) -- (35,40) --(35,44); 
\draw[black, very thick, rounded corners] (50,22) -- (50,30); 
\draw[black, very thick, rounded corners] (69,52) -- (69,34) -- (66,24)   -- (50,22)  -- (34,24) -- (31,34) --(31,52);
\draw [very thick, black, fill] (31,52) circle[radius=1.2cm];
\draw [very thick, black, fill] (69,52) circle[radius=1.2cm];
\draw [very thick, black, fill] (50,30) circle[radius=1.2cm];
\draw [very thick, black, fill] (50,22) circle[radius=1.2cm];
\node at (50,10) {$\Omega_8$};
\end{tikzpicture}
\qquad
% 9 %  
\begin{tikzpicture}[scale=0.03] 
\draw[black, very thick, rounded corners] (52,80) -- (57,75) -- (57,70) --(52,62); 
\draw[black, very thick, rounded corners] (48,58) -- (44,55) -- (40,53) -- (37,50) -- (36,48) -- (35,44); 
\draw[black, very thick, rounded corners] (48,80) -- (43,75) -- (43,70) -- (48,62) -- (50,60) -- (52,58) -- (56,55) -- (60,53) -- (63,50) -- (64,48); 
\draw[black, very thick, rounded corners] (52,80) -- (48,83); 
\draw[black, very thick, rounded corners] (52,83) -- (55,85) --(60,85) --(67,83) --(70,79) -- (72,76) -- (74,70) -- (72,60) -- (70,55) -- (68,50) -- (66,47) -- (62,45) -- (57,43) --(53,42) -- (50,42) -- (47,42) -- (43,43) -- (38,45); 
\draw[black, very thick, rounded corners] (48,83) -- (45,85) -- (40,85) -- (33,83) -- (30,79) -- (28,76) -- (26,70) -- (28,60) -- (30,55) -- (32, 50) -- (34,48); 
\draw[black, very thick, rounded corners] (65,44) -- (65,40) -- (64,35) -- (60,29) -- (53,27) -- (50,27) -- (47,27) --(40,29) -- (36,35) -- (35,40) --(35,44); 
\draw[black, very thick, rounded corners] (50,35) -- (50,27); 
\draw[black, very thick, rounded corners] (50,35) -- (43,43); 
\draw[black, very thick, rounded corners] (50,35) -- (57,43); 
\draw [very thick, black, fill] (43,43) circle[radius=1.2cm];
\draw [very thick, black, fill] (57,43) circle[radius=1.2cm];
\draw [very thick, black, fill] (50,27) circle[radius=1.2cm];
\draw [very thick, black, fill] (50,35) circle[radius=1.2cm];
\node at (50,10) {$\Omega_9$};
\end{tikzpicture}
\qquad
% 10 % 
\begin{tikzpicture}[scale=0.03] 
\draw[black, very thick, rounded corners] (52,80) -- (57,75) -- (57,70) --(52,62); 
\draw[black, very thick, rounded corners] (48,58) -- (44,55) -- (40,53) -- (37,50) -- (36,48) -- (35,44); 
\draw[black, very thick, rounded corners] (48,80) -- (43,75) -- (43,70) -- (48,62) -- (50,60) -- (52,58) -- (56,55) -- (60,53) -- (63,50) -- (64,48); 
\draw[black, very thick, rounded corners] (52,80) -- (48,83); 
\draw[black, very thick, rounded corners] (52,83) -- (55,85) --(60,85) --(67,83) --(70,79) -- (72,76) -- (74,70) -- (72,60) -- (70,55) -- (68,50) -- (66,47) -- (62,45) -- (57,43) --(53,42) -- (50,42) -- (47,42) -- (43,43) -- (38,45); 
\draw[black, very thick, rounded corners] (48,83) -- (45,85) -- (40,85) -- (33,83) -- (30,79) -- (28,76) -- (26,70) -- (28,60) -- (30,55) -- (32, 50) -- (34,48); 
\draw[black, very thick, rounded corners] (65,44) -- (65,40) -- (64,35) -- (60,29) -- (53,27) -- (50,27) -- (47,27) --(40,29) -- (36,35) -- (35,40) --(35,44); 
\draw[black, very thick] (36,35) -- (64,35); 
\draw[black, very thick] (50,27) -- (50,35); 
\draw [very thick, black, fill] (36,35) circle[radius=1.0cm];
\draw [very thick, black, fill] (64,35) circle[radius=1.0cm];
\draw [very thick, black, fill] (50,27) circle[radius=1.0cm];
\draw [very thick, black, fill] (50,35) circle[radius=1.0cm];
\node at (50,10) {$\Omega_{10}$};
\end{tikzpicture}
\end{center}
\caption{Spatial $\mathbb{K}_4$-graphs with at most four crossings.} \label{fig:K4}
\end{figure}

A graph $\mathbb K_4$ has four vertices and six edges which we denote by $a_1, a_2, \ldots, a_6$ as presented in Fig.~\ref{fig:K4cycles}. It is easy to see that $\mathbb K_4$ has 
four simple cycles of length three, $\{ a_1, a_2, a_6\}$, $\{ a_1, a_3, a_5\}$, $\{ a_2, a_3, a_4\}$, $\{ a_4, a_5, a_6\}$, and three simple cycles of length four $\{ a_1, a_2, a_4, a_5\}$, $\{ a_1, a_3, a_4, a_6\}$, $\{ a_2, a_3, a_5, a_6\}$.  These seven cycles give seven knots in a spatial embedding of $\mathbb K_4$. A knot corresponding to a simple cycle in a spatial graph $\mathcal G$ is called a \emph{constituent knot}  of $\mathcal G$. Thus, any spatial $\mathbb K_4$-graph has seven constituent knots. The following result was obtained by Yamamoto in~\cite{Yam}.

\begin{theorem}\cite{Yam} \label{theorem:Yamamoto}
Let $c_1, \ldots, c_7$ be the seven cycles in $\mathbb K_4$. For any ordered $7$-tuple $(k_1, \ldots, k_7)$ of knots, there is a spatial embedding of $\mathbb K_4$ such that the corresponding list of knot types of $(c_1, \ldots, c_7)$ is $(k_1, \ldots, k_7)$. 
\end{theorem}

We will interested not only in constituent knots of a spatial graphs, but also in  spatial subgraph.  It is clear from Fig.~\ref{fig:K4cycles} that graph $\mathbb K_4$ contains six $\theta$-graphs as subgraphs.   Thus, any spatial $\mathbb K_4$-graph has six constituent spatial $\theta$-graphs. 
We call them \emph{constituent spatial $\theta$-graphs}.
\begin{figure}[h] 
\begin{center}
%\unitlength=.05mm
% 0 
\begin{tikzpicture}[scale=0.03] 
\draw [very thick, black, fill] (50,50) circle[radius=1.2cm];
\draw[black, very thick] (50,50) -- (50,25);
\draw [very thick, black, fill] (50,25) circle[radius=1.2cm];
\draw[black, very thick] (50,50) -- (29,63);
\draw [very thick, black, fill] (29,63) circle[radius=1.2cm];
\draw[black, very thick] (50,50) -- (71,63);
\draw [very thick, black, fill] (71,63) circle[radius=1.2cm];
\draw[black, very thick] (50,50) circle (25);
\node at (40,63) { $\small a_1$};
\node at (60,63) { $\small a_2$};
\node at (42,37) { $\small a_3$};
\node at (82,37) { $\small a_4$};
\node at (20,37) { $\small a_5$};
\node at (50,82) { $\small a_6$};
%\node (K) at (3,1) {Graph $\bf K_4$};
% \node at (-1,2) {$v_1$};
%  \node at (5,2) {$v_2$};
%\draw [thick, ->] (50,50) -- (39.5,56.5); 
%\draw [thick, ->] (50,50) -- (50,37.5); 
%\draw [thick, ->] (50,50) -- (60.5,56.5); 
%\draw [thick, ->] (49,75) -- (51,75); 
%\draw [thick, ->] (71,37) -- (70,36); 
%\draw [thick, ->] (30,36) -- (29,37); 
\end{tikzpicture} 
%1
\begin{tikzpicture}[scale=0.03] 
\draw[black, very thick] (50,50) -- (50,25);
\draw [very thick, black, fill] (50,25) circle[radius=1.2cm];
\draw[black, very thick] (50,50) -- (71,63);
\draw [very thick, black, fill] (71,63) circle[radius=1.2cm];
\draw[black, very thick] (50,50) circle (25);
\end{tikzpicture}
%2
\begin{tikzpicture}[scale=0.03] 
\draw[black, very thick] (50,50) -- (50,25);
\draw [very thick, black, fill] (50,25) circle[radius=1.2cm];
\draw[black, very thick] (50,50) -- (29,63);
\draw [very thick, black, fill] (29,63) circle[radius=1.2cm];
%\draw[black, very thick] (50,50) -- (71,63);
%\draw [very thick, black, fill] (71,63) circle[radius=1.2cm];
\draw[black, very thick] (50,50) circle (25);
\end{tikzpicture}
%3
\begin{tikzpicture}[scale=0.03] 
\draw[black, very thick] (50,50) -- (29,63);
\draw [very thick, black, fill] (29,63) circle[radius=1.2cm];
\draw[black, very thick] (50,50) -- (71,63);
\draw [very thick, black, fill] (71,63) circle[radius=1.2cm];
\draw[black, very thick] (50,50) circle (25);
\end{tikzpicture}
%4
\begin{tikzpicture}[scale=0.03]
\draw[black, very thick] (71,63) arc
	[
		start angle=32,
		end angle=270,
		radius=25cm
	] ;
\draw [very thick, black, fill] (50,50) circle[radius=1.2cm];
\draw[black, very thick] (50,50) -- (50,25);
%\draw [very thick, black, fill] (50,25) circle[radius=1.2cm];
\draw[black, very thick] (50,50) -- (29,63);
\draw [very thick, black, fill] (29,63) circle[radius=1.2cm];
\draw[black, very thick] (50,50) -- (71,63);
%\draw [very thick, black, fill] (71,63) circle[radius=1.2cm];
%\draw[black, very thick] (50,50) circle (25);
\end{tikzpicture}
%5
\begin{tikzpicture}[scale=0.03]
\draw[black, very thick] (50,25) arc
	[
		start angle=-90,
		end angle=150,
		radius=25cm
	] ;
\draw [very thick, black, fill] (50,50) circle[radius=1.2cm];
\draw[black, very thick] (50,50) -- (50,25);
%\draw [very thick, black, fill] (50,25) circle[radius=1.2cm];
\draw[black, very thick] (50,50) -- (29,63);
%\draw [very thick, black, fill] (29,63) circle[radius=1.2cm];
\draw[black, very thick] (50,50) -- (71,63);
\draw [very thick, black, fill] (71,63) circle[radius=1.2cm];
%\draw[black, very thick] (50,50) circle (25);
\end{tikzpicture}
%6
\begin{tikzpicture}[scale=0.03]
\draw[black, very thick] (29,63) arc
	[
		start angle=150,
		end angle=390,
		radius=24.5cm
	] ;
\draw [very thick, black, fill] (50,50) circle[radius=1.2cm];
\draw[black, very thick] (50,50) -- (50,25);
\draw [very thick, black, fill] (50,25) circle[radius=1.2cm];
\draw[black, very thick] (50,50) -- (29,63);
%\draw [very thick, black, fill] (29,63) circle[radius=1.2cm];
\draw[black, very thick] (50,50) -- (71,63);
%\draw [very thick, black, fill] (71,63) circle[radius=1.2cm];
%\draw[black, very thick] (50,50) circle (25);
\end{tikzpicture}
\end{center}
\caption{Graph $\mathbb K_4$ and six $\theta$-graphs.} \label{fig:K4cycles}
\end{figure}
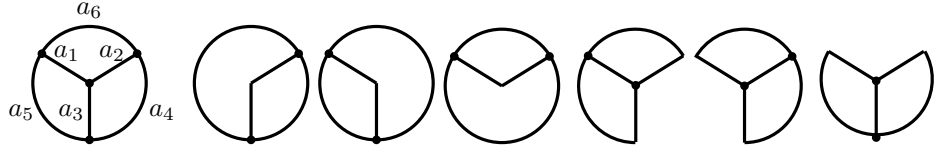

\subsection{Band surfaces}

In~\cite{KSWZ} the authors corresponded to a connected spatial graph an oriented band surface for which the spatial knot is a spine.

\begin{lemma} \cite[Lemma~2.1]{KSWZ} \label{lem2.1}
For any graph $G$ and any spatial embedding $\mathcal G$, there is an orientable surface $S(\mathcal G)$ containing $\mathcal G$ and collapsing to $\mathcal G$. 
\end{lemma} 

We call a surface from the Lemma~\ref{lem2.1} a \emph{band surface}. Such a surface can be described in the following way. Take a regular projection of $\mathcal G$ and isotope $\mathcal G$ such that near each vertex, all edges lie in a small disk parallel to the projection plane. Put such a disk at each vertex; then connect disks with bands, one along each edge. See in Fig.~\ref{fig:Omega7} an orientable surface $S(\Omega_7)$ for the spatial $\mathbb K_4$-graph  $\Omega_7$.    
\begin{figure}[h]
\includegraphics[width=0.45\textwidth]{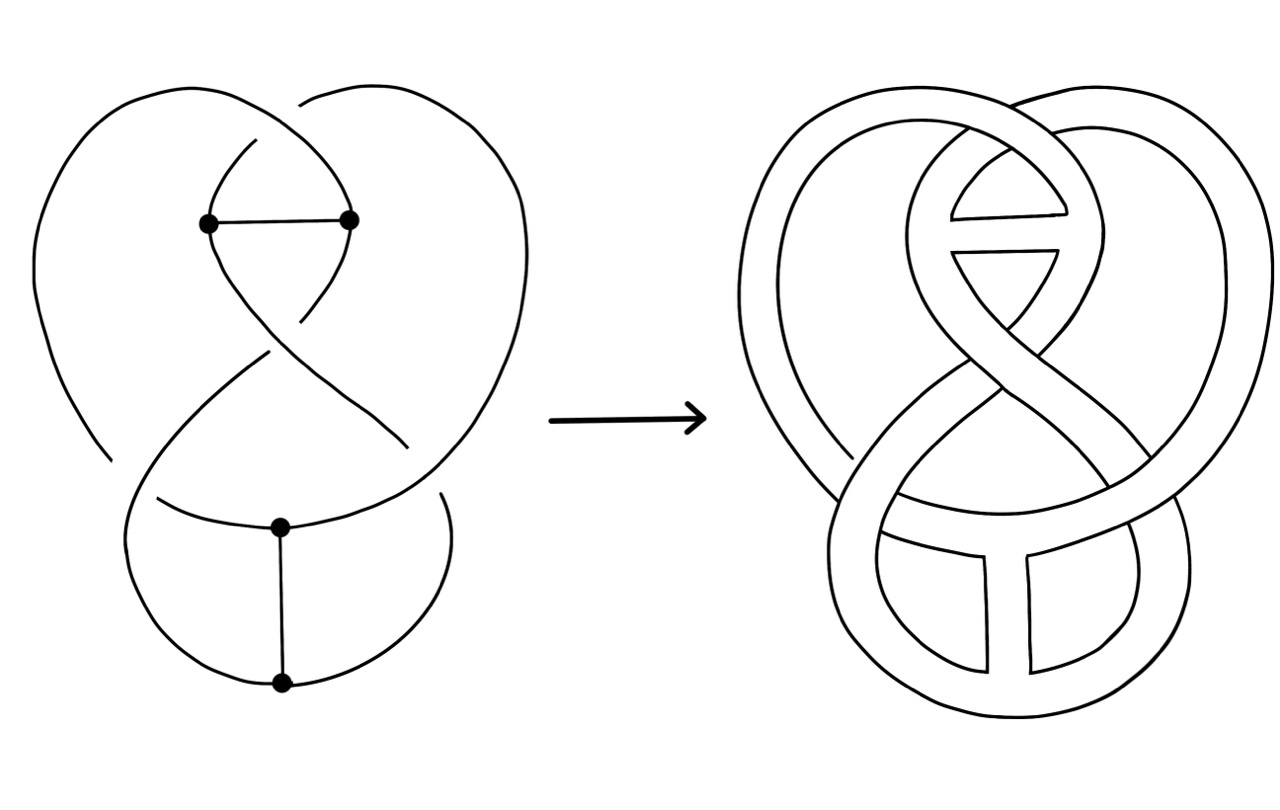}
\caption{Spatial $\mathbb K_4$-graph $\Omega_7$ and  a band surface $S(\Omega_7)$.} \label{fig:Omega7}
\end{figure}

In~\cite{KSWZ} there was studded a question: when surface $S(\mathcal G)$ is an invariant of $\mathcal G$? To formulate the result, we recall some necessary terminology according to~\cite{KSWZ}.  

Two spatial graphs are said to be \emph{equivalent} if they are ambient isotopic. The equivalence class of a spatial graphs $\mathcal G$ is called the \emph{knot type} of $\mathcal G$. A spatial graph $\mathcal G$ is said to be  \emph{planar} if it is equivalent to some spatial graph $\mathcal G_0$ in $\mathbb R^2$ (or $S^2$). A graph $G$ is \emph{planar} if there is a planar embedding $\mathcal G$ of $G$. A graph $G$ is \emph{connected} if some $\mathcal G$ is connected as subspace of $\mathbb R^3$ (or $S^3$).

Let $L = K_1 \cup K_2$ be a 2-component oriented link  in $S^3$, and $D$ be a diagram of $L$ with set $C(D)$ of crossings. Let $C(K_1 \cap K_2)$ be the subset of $C(D)$ consisting of crossings where $K_1$ and $K_2$ meet.  The \emph{linking number} $\operatorname{lk}(K_1, K_2)$ of $K_1$ and $K_2$ is defined as
$$
\operatorname{lk} (K_1, K_2) = \frac{1}{2} \sum_{c \in C(K_1 \cap K_2)} \, \varepsilon(c), 
$$
where $\varepsilon(c)$ is defined according to Fig.~\ref{fig:cross}. 
\begin{figure}[h]
\begin{center}
\begin{tikzpicture} 
\draw[black, very thick] (0,0) -- (0.4,0.4);
\node at (2,0.5) {$\varepsilon (c) = +1$};
\draw[black, very thick, ->] (0,1) -- (1,0);
\draw[black, very thick, ->] (0.6,0.6) -- (1,1);
\draw[black, very thick] (5,1) -- (5.4,0.6);    
\node at (7,0.5) {$\varepsilon (c) = -1$};
\draw[black, very thick, ->] (5.6,0.4) -- (6,0);
\draw[black, very thick, ->] (5,0) -- (6,1);
\end{tikzpicture}
\end{center} \caption{Two types of crossings.} \label{fig:cross}
\end{figure}
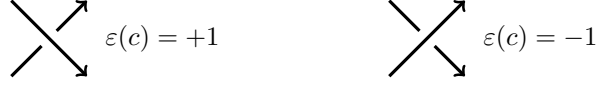

Let $S$ be an oriented surface in $S^3$ (or in $\mathbb R^3$), $x$ and $y$ be cycles on $S$. Let $x^+$ denote the result of pushing $x$ in a very small amount into $S^3 \setminus S$ (or $\mathbb R^3 \setminus S$) along the positive normal direction to $S$. The function $\langle , \rangle : H_1 (S, \mathbb Z) \times H_1 (S, \mathbb Z) \to \mathbb Z$ defined by 
$$
\langle x, y \rangle = \operatorname{lk} (x^+, y), 
$$
is called the \emph{Seifert form} (or \emph{Seifert linking form}) for surface $S$, and $\langle x, y \rangle$ is called the \emph{Seifert pairing}. It is a well-defined, bilinear pairing, an invariant of the ambient isotopy class of the embedding $S \subset S^3$ (or $\mathbb R^3$). %Let $x \cdot y$ be the intersection number of $x$ and $y$, then 
%$$
%\langle x, y \rangle - \langle y, x \rangle = x \cdot y.
%$$
An oriented band surface $S(\mathcal G)$ is called  \emph{good} if its  Seifert form is zero. It is known that to have a good surface the graph must to be planar. 

\begin{lemma} \cite[Lemma~2.2]{KSWZ}
If $G$ is a nonplanar graph, then for any its embedding $\mathcal G$ there are no surfaces of zero Seifert form collapsing to $\mathcal G$. 
\end{lemma}

We say that a spatial graph $\mathcal G$ is \emph{splittable} if there exists a 2-sphere $S$ in $S^3 \setminus \mathcal G$ which splits $S^3$ into 3-balls $B^3_1$ and $B^3_2$ such that $B^3_i \cap \mathcal G$ nonempty for $i=1,2$.  We say that a spatial graph $\mathcal G$ is \emph{prime} if it is nonsplittable and it cannot be decomposed into a connected sum along a point or along two points, see~\cite{KSWZ}.  

The following theorem describes the existence and uniqueness of band surfaces.

\begin{theorem} \cite[Theorem~2.4]{KSWZ} \label{theorem:KSWZ}
Let $\mathcal G_0$ be a planar embedding of a connected trivalent planar graph $G$. Suppose $\mathcal G_0$ is prime. 
\begin{itemize}
\item[(1)] If the number of edges in $G$ is at most six, then for each $\mathcal G$ there exists a unique (up to ambient isotopy) surface $S(\mathcal G)$ with zero Seifert form. 
\item[(2)] If the number of edges in $G$ is more than six, then
\begin{itemize}
\item[(i)] there exists a $\mathcal G$ with no $S(\mathcal G)$ of zero Seifert form; 
\item[(ii)] if there is an $S(\mathcal G)$ of zero Seifert form, it is the unique such surface. 
\end{itemize}
\end{itemize} 
\end{theorem}

If the number of edges in $G$ is at most six, then being trivalent, $G$ either consist of two vertices and three edges, or of four vertices and six edges. The case of the $\theta$-graph was considered in~\cite{H}. In the present paper we consider the case of the graph $\mathbb K_4$. 

\section{Jones polynomial,  Yamada polynomial and Jaeger polynomial} \label{section:polynomials}

In this section we recall definitions of Kauffman bracket and Jones polynomials of knots and links, and Yamada and Jaeger polynomials of spatial graphs. 

\subsection{Jones polynomial}

Jones polynomial was defined by Jones in~\cite{Jo}. Here we recall the definition of the Jones polynomial via the bracket polynomial introduced by Kauffman in~\cite{K87}. Let $D$ be a diagram of a non-oriented link $L$, then the Kauffman bracket polynomial $\langle D \rangle \in \mathbb{Z}[A^{\pm 1}]$ is defined by the following three axioms, where the first is a skein relation, the second describes a disjoint union with an unknot, and the third is a normalization for an unknot:    
\begin{enumerate}
\item $\langle D \rangle = A \, \langle D_{+} \rangle + A^{-1} \, \langle D_{-} \rangle$;
\item $\langle D \sqcup {\bf 0 } \rangle  = (-A^2 - A^{-2}) \, \langle D \rangle$;
\item $\langle {\bf 0} \rangle = 1$.
\end{enumerate}
Here diagrams $D_+$ and $D_-$ are obtained by smoothing at some crossing of a diagram $D$ as pictured below, and ${\bf 0}$ denotes an unknot.
\smallskip 
%\begin{figure}[h]
\begin{center}
\scalebox{0.9}{
\begin{tikzpicture} 
%\draw[step=5.mm, gray, very thin] (0,0) grid (100.mm,10.mm);
\begin{knot}[
  consider self intersections=true,
%  draft mode=crossings,
  flip crossing=2,
%  only when rendering/.style={
%    show curve controls
%  }
  ]
\strand[black, very thick] (0,0) -- (1,1);
\strand[black, very thick] (0,1) -- (1,0);
\node at (0.5,-0.7) {$D$};
\draw[blue, dotted, very thick] (0.5,0.5) circle (0.75);
\draw[black, ultra thick] (3,0) arc (-90:90: 0.3cm and 0.5cm);
\draw[black, ultra thick] (4,1) arc (90:270: 0.3cm and 0.5cm);
\node at (3.5,-0.7) {$D_+$};
\draw[blue, dotted, very thick] (3.5,0.5) circle (0.75);
\draw[black, ultra thick] (7,1) arc (0:-180: 0.5cm and 0.3cm);
\draw[black, ultra thick] (7,0) arc (0:180: 0.5cm and 0.3cm);
\node at (6.5,-0.7) {$D_-$};
\draw[blue, dotted, very thick] (6.5,0.5) circle (0.75);
\end{knot}
\end{tikzpicture}
}
\end{center} 
%\caption{Skein triple $D$, $D_0$ and $D_{\infty}$ for a knot or link diagram.} \label{fig:skein}
%\end{figure}
It is well-known that the bracket polynomial is a regular isotopy invariant, i.e. it is invariant under second and third Reidemeister moves. Assume that $L$ is equipped with an orientation, and denote the obtained oriented diagram by $\overline{D}$. Let $C(\overline{D})$ be the set of all crossings of $\overline{D}$, define a function $\varepsilon : C (\overline{D}) \to \{ +1, -1\}$, which depends on an orientation at $c$, as presented in Fig.~\ref{fig:cross}.

Consider the writhe number $w( \overline{D} ) = \sum \limits_{c \in C(\overline{D})} \varepsilon (c)$. It is known~\cite{K87} that  
\begin{equation}
     V(L)=(-A^3)^{-w( \overline{D})} \, \langle D \rangle 
     \label{eqnJones}
\end{equation}
is an invariant of $L$, and substitution  $A=t^{-\frac{1}{4}}$ gives the Jones polynomial.

\subsection{Yamada polynomial} 

Let us recall a definition of Yamada polynomial of a spatial graph, see~\cite{Ya}. Let $G$ be a graph with  the set of vertices $V(G)$ and the set of edges $E(G)$. Denote by $\omega(G)$ and $\beta(G) = |E(G)| - |V(G)| + \omega(G)$ the number of connected components and the first Betti number of $G$,  respectively. Then consider a 2-variable Laurent polynomial $h(G;x,y)$ defined by 
$$
 h(G;x,y)=\sum\limits_{F\subset E(G)}(-x)^{-|F|}x^{\omega(G-F)}y^{\beta(G-F)}
$$
with $h(\varnothing; x,y)=1$ for an empty graph. 

Let $D$ be a diagram of a spatial graph $\mathcal{G}$. Denote by $C(D)$ the set of all crossings in $D$. Define a \emph{state} $S$ of $D$ as  a function 
$$
S : C(D)\rightarrow \{+1,-1,0\}.
$$
Changing each crossing $z$ in $D$ according to the state $S$ as illustrated in Fig.~\ref{fig:state}  we obtain a diagram without crossings, which we denote $D_S$. Thus, if $D$ has $n$ crossings then there are $3^n$ states and each $D_S$ represents a planar graph itself.

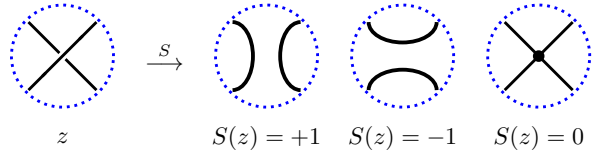
\begin{figure}[h]
\begin{center}
\scalebox{0.9}{
\begin{tikzpicture} 
%\draw[step=5.mm, gray, very thin] (0,0) grid (100.mm,10.mm);
\begin{knot}[
  consider self intersections=true,
%  draft mode=crossings,
  flip crossing=2,
%  only when rendering/.style={
%    show curve controls
%  }
  ]
\strand[black, very thick] (0,0) -- (1,1);
\strand[black, very thick] (0,1) -- (1,0);
\node at (0.5,-0.7) {$z$};
\draw[blue, dotted, very thick] (0.5,0.5) circle (0.75);
\node at (2,0.5) {$\stackrel{S}{\longrightarrow}$};
\draw[black, ultra thick] (3,0) arc (-90:90: 0.3cm and 0.5cm);
\draw[black, ultra thick] (4,1) arc (90:270: 0.3cm and 0.5cm);
\node at (3.5,-0.7) {$S(z) = +1$};
\draw[blue, dotted, very thick] (3.5,0.5) circle (0.75);
\draw[black, ultra thick] (6,1) arc (0:-180: 0.5cm and 0.3cm);
\draw[black, ultra thick] (6,0) arc (0:180: 0.5cm and 0.3cm);
\node at (5.5,-0.7) {$S(z) = -1$};
\draw[blue, dotted, very thick] (5.5,0.5) circle (0.75);
\draw[black, very thick] (7,0) -- (8,1);
\draw[black, very thick] (7,1) -- (8,0);
\node at (7.5,-0.7) {$S (z) = 0$};
\filldraw [line width=2pt, black] (7.5,0.5) circle[radius=0.05cm];
\draw[blue, dotted, very thick] (7.5,0.5) circle (0.75);
\end{knot}
\end{tikzpicture}
}
\end{center} \caption{Changing a crossing $z$ according to the state.} \label{fig:state}
\end{figure}

Let us define $\{D|S\}=A^{m_1-m_2}$ for diagram $D$ and state $S$, where $m_1= \# \{ z \in C(D) : S(z) = +1 \}$ and $m_2 = \# \{ z \in C(D) : S(z) = -1 \}$. Then the  Yamada polynomial is defined as follows
$$
Y(D;A)=\sum\limits_{S \in \mathcal{S}} \{ D | S\} \, h(D_S; -1, -A-2-A^{-1})
$$
with $Y(\varnothing; A)=1$,  where $\mathcal{S}$ is the set of all states of diagram $D$. (For the reader's convenience we point out that $Y(D; A)$ was denoted by $R(g) (A)$ in~\cite{Ya}.)  

Denote by $D_+$, $D_-$, and $D_0$  diagrams obtained from a diagram $D$ by changings according to $S(z) = +1$, $S(z) = -1$, and $S(z) = 0$, respectively, for some crossing $z \in D$ as presented in Fig.~\ref{fig:state}. In \cite{Ya} the following properties were proved.

\begin{lemma}  \label{lemma2.0} Let $\mathcal{G}$ be a spatial graph and $D$ be a diagram of $\mathcal{G}$. Then   
\begin{enumerate}
\item \cite[Prop.~3]{Ya}  The following skein relation holds for diagrams in above notations:   
$$
Y (D; A) = A \, Y(D_+; A) + A^{-1} \, Y(D_-; A) + Y(D_0; A).
$$
\item \cite[Theorem~3]{Ya} $Y(D; A)$ is a flat isotopy invariant up to multiplying  by $(-A)^k$ for some integer $k$.
\item \cite[Theorem~4]{Ya} If $D$ is a diagram of a spatial graph whose maximum vertex degree is at most three, then $Y(D; A)$ is a pliable isotopy invariant up to multiplying by $(-A)^k$ for some integer $k$. 
\end{enumerate}
\end{lemma}

The Table~1 presents the Yamada polynomials for the ten spatial $\mathbb K_4$-graphs from Fig.~\ref{fig:K4} calculated in~\cite{VD}.  
\begin{table}[th]
\begin{center}
\begin{tabular}{|c|c|} %\hline
\multicolumn{2}{r}{Table \, 1}  \\[1mm]  \hline
Graph $G$ & The Yamada polynomial $Y(G; A)$\\ \hline
$\scriptstyle \Omega_1$ &  $ \scriptstyle A^3+2A+2A^{-1}+A^{-3} $ \\ \hline
$\scriptstyle \Omega_2$ & $\scriptstyle A^8+A^6+A^5-A^4+A^3-2A^2+A-1+A^{-1}+A^{-2}+A^{-3}+A^{-4}+A^{-5} $ \\ \hline
$\scriptstyle \Omega_3$ & $\scriptstyle  2A^6+A^4+A^3-2A^2-4-A^{-1}-3A^{-2}-A^{-3}+A^{-7}  $ \\ \hline
$\scriptstyle \Omega_4$ & $\scriptstyle A^8-A^7+A^6-A^4+A^3-2A^2+A-2-A^{-2}-A^{-3}-A^{-4}-A^{-6} $ \\ \hline
$\scriptstyle \Omega_5$ & $\scriptstyle  A^8-A^7+A^6-A^5-A^4-2A^2+A-1+2A^{-1}+A^{-2}+2A^{-3}+A^{-4}+2A^{-5}+A^{-7}$ \\ \hline
$\scriptstyle \Omega_6$ & $\scriptstyle  A^7-A^6+A^4+A^2+3A+3A^{-1}-A^{-2}+A^{-3}-A^{-4}-2A^{-5}+A^{-6}-A^{-7}+A^{-9} $ \\ \hline
$\scriptstyle \Omega_7$ & $\scriptstyle -A^8-A^5+A^4+A^3+3A+3A^{-1}+A^{-3}+A^{-4}-A^{-5}-A^{-8} $ \\ \hline
$\scriptstyle \Omega_8$ & $\scriptstyle A^9-A^8+2A^6-A^5+A^4+2A^3-A^2+2A-2+A^{-1}-A^{-2}-A^{-3}+2A^{-4}+2A^{-7} $ \\ \hline
$\scriptstyle \Omega_9$ & $\scriptstyle  -A^8+A^7-A^5+2A^4+2A-1+2A^{-1}-A^{-2}+A^{-3}+A^{-4}-A^{-5}+A^{-6}+A^{-7}-A^{-8}+A^{-9} $ \\ \hline
$\scriptstyle \Omega_{10}$ & $\scriptstyle A^9-A^8+A^7-A^5+A^4+2A+2A^{-1}+A^{-4}-A^{-5}+A^{-7}-A^{-8}+A^{-9} $ \\ \hline
\end{tabular}
\end{center}
\end{table}  

\subsection{Bar diagram and Jaeger polynomial} 

The Yamada polynomial can be expressed from the Jaeger polynomial in the case when spatial graph is planar, see Lemma~\ref{lemma2.2}. Let us recall a definition of the Jaeger polynomial of a spatial graph from~\cite{Ja}. Let $G$ denote a  graph, $\mathcal{G}$ a spatial $G$-graph, and $D$ a diagram of $\mathcal{G}$. For a diagram $D$ define the \emph{bar diagram} $B$ as a band diagram with labels called \emph{bars}. The bar diagram $B$ is constructed from $D$ as follows. 
\begin{itemize}
\item  For any vertex of $D$ we replace it locally as presented in Fig.~\ref{fig:DtoB}-(a), where number of bands is equal to valency of a vertex.  
\item For any crossing of $D$ we replace it locally as presented in Fig.~\ref{fig:DtoB}-(b) by getting four new crossings instead one. 
\item For any arc of $D$ between two vertices or crossings we replace it by a doubled arc with a grey bar between new arcs as presented in Fig.~\ref{fig:DtoB}-(c). 
\end{itemize}
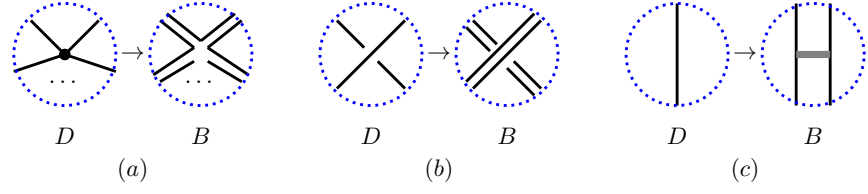
\begin{figure}[h]
\begin{center}
\scalebox{0.9}{
\begin{tikzpicture} 
\draw[black, very thick] (0.5,0.5) -- (0.,1.);
\draw[black, very thick] (0.5,0.5) -- (1.,1.);
\draw[black, very thick] (0.5,0.5) -- (-0.25,0.25);
\draw[black, very thick] (0.5,0.5) -- (1.25,0.25);
\filldraw [line width=2pt, black] (0.5,0.5) circle[radius=0.05cm];
\node at (0.5,0.1) {$\dots$};
\node at (0.5,-0.7) {$D$};
\draw[blue, dotted, very thick] (0.5,0.5) circle (0.75);
\node at (1.5,0.5) {$\to$};
\node at (1.5,-1.2) {$(a)$};
\draw[black, very thick] (2.5,0.7) -- (2.,1.1);
\draw[black, very thick] (2.5,0.7) -- (3.,1.1);
\draw[black, very thick] (2.6,0.6) -- (3.1,1.0);
\draw[black, very thick] (2.6,0.6) -- (3.2,0.2);
\draw[black, very thick] (2.4,0.6) -- (1.9,1.0);
\draw[black, very thick] (2.4,0.6) -- (1.8,0.2);
\draw[black, very thick] (2.6,0.4) -- (3.1,0.1);
\draw[black, very thick] (2.4,0.4) -- (1.9,0.1);
\node at (2.5,0.1) {$\dots$};
\node at (2.5,-0.7) {$B$};
\draw[blue, dotted, very thick] (2.5,0.5) circle (0.75);
\draw[black, very thick] (4.5,0) -- (5.5,1);
\draw[black, very thick] (4.5,1) -- (4.9,0.6);
\draw[black, very thick] (5.1,0.4) -- (5.5,0);
\node at (5.,-0.7) {$D$};
\draw[blue, dotted, very thick] (5.,0.5) circle (0.75);
\node at (6.,0.5) {$\to$};
\node at (6.,-1.2) {$(b)$};
\draw[black, very thick] (6.5,0) -- (7.5,1);
\draw[black, very thick] (6.4,0.1) -- (7.4,1.1);
\draw[black, very thick] (6.5,1) -- (6.85,0.65);
\draw[black, very thick] (7.1,0.4) -- (7.5,0);
\draw[black, very thick] (6.4,0.9) -- (6.75,0.55);
\draw[black, very thick] (7.,0.3) -- (7.4,-0.1);
\node at (7.,-0.7) {$B$};
\draw[blue, dotted, very thick] (7.,0.5) circle (0.75);
\draw[black, very thick] (9.5,-0.25) -- (9.5,1.25);
\node at (9.5,-0.7) {$D$};
\draw[blue, dotted, very thick] (9.5,0.5) circle (0.75);
\node at (10.5,0.5) {$\to$};
\node at (10.5,-1.2) {$(c)$};
\draw[black, very thick] (11.25,-0.25) -- (11.25,1.25);
\draw[black, very thick] (11.75,-0.25) -- (11.75,1.25);
\draw[gray, line width = 3 pt] (11.25,0.5) -- (11.75,0.5);
\node at (11.5,-0.7) {$B$};
\draw[blue, dotted, very thick] (11.5,0.5) circle (0.75);
\end{tikzpicture}
}
\end{center} \caption{The rules for constructing the bar diagram $B$ from $D$.} \label{fig:DtoB}
\end{figure}

In Fig.~\ref{fig:tobar} we illustrate the construction of a bar diagram from a diagram of the spatial $\mathbb K_4$-graph $\Omega_7$. The obtained bar diagram has $14$ bars. 

\begin{figure}[h]
    \includegraphics[width=0.5\textwidth]{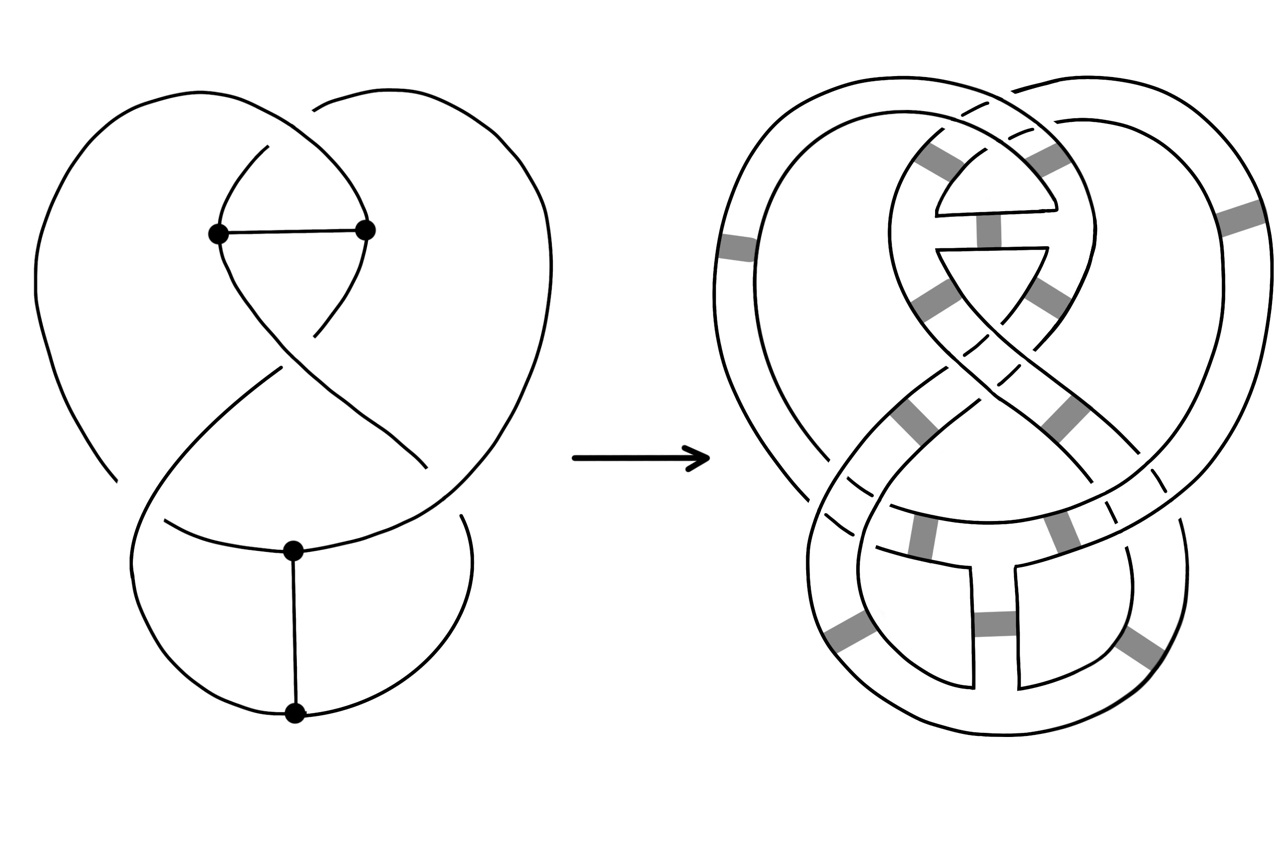}
\caption{A bar diagram for the diagram of the spatial $\mathbb K_4$-graph $\Omega_7$.} \label{fig:tobar}
\end{figure}

Now we define a polynomial for a bar diagram $B$. This polynomial is an extension of some specialization of the Dubrovnik polynomial for links. Let $D$ be a diagram of a non-oriented  link $L$. In~\cite[Section~VII]{K90} Kauffman introduced the Dubrovnik polynomial $\mathfrak{D} (D; a, z) \in \mathbb Z[a^{\pm1}, z^{\pm1}]$. It is an invariant of regular isotopy of links and it satisfies the  following  four axioms 
\begin{enumerate}
\item $\mathfrak{D}(D_{c_+})-\mathfrak{D}(D_{c_-})=z(\mathfrak{D}(D_+)-\mathfrak{D}(D_-))$;
\item $\mathfrak{D}(D_{l_+})=a\mathfrak{D}(D_{l_0)}$;
\item $\mathfrak{D}(D_{l_-})=a^{-1}\mathfrak{D}(D_{l_0})$;
\item $\mathfrak{D}({\bf 0})=1$, 
\end{enumerate}
where diagrams $D_{c_+}$, $D_{c_-}$, $D_+$, $D_-$, $D_{l_+}$, $D_{l_0}$, $D_{l_-}$  are as pictured below and ${\bf 0}$ denotes an unknot.
%\begin{figure}[h]
\medskip 
\begin{center}
\scalebox{0.9}{
\begin{tikzpicture} 
\draw[black, very thick] (3,0) -- (4,1);
\draw[black, very thick] (3,1) -- (3.4,0.6);
\draw[black, very thick] (3.6,0.4) -- (4,0);
\node at (3.5,-0.7) {$D_{c_+}$};
\draw[blue, dotted, very thick] (3.5,0.5) circle (0.75);
\end{tikzpicture} 
\quad
\begin{tikzpicture}
    \draw[blue, dotted, very thick] (3.5,0.5) circle (0.75);
    \draw[black, very thick] (3,1) -- (4,0);
\draw[black, very thick] (3,0) -- (3.4,0.4);
\draw[black, very thick] (3.6,0.6) -- (4,1);
    \node at (3.5,-0.7) {$D_{c_-}$};
\end{tikzpicture}
\quad
\begin{tikzpicture}
    \draw[black, ultra thick] (3,0) arc (-90:90: 0.3cm and 0.5cm);
\draw[black, ultra thick] (4,1) arc (90:270: 0.3cm and 0.5cm);
\node at (3.5,-0.7) {$D_+$};
\draw[blue, dotted, very thick] (3.5,0.5) circle (0.75);
\end{tikzpicture}
\quad
\begin{tikzpicture}
 \draw[black, ultra thick] (10,1) arc (0:-180: 0.5cm and 0.3cm);
\draw[black, ultra thick] (10,0) arc (0:180: 0.5cm and 0.3cm);
\node at (9.5,-0.7) {$D_-$};
\draw[blue, dotted, very thick] (9.5,0.5) circle (0.75);
\end{tikzpicture}
\qquad \quad 
\begin{tikzpicture}
    \draw[black, very thick] (0,1) -- (0.35,0.55); 
\draw[black, very thick] (0.45,0.45) to [out=315, in=180] (0.85,0.2);
\draw[black, very thick] (0.85,0.2) to [out=0, in=270] (1.1,0.5);
\draw[black, very thick] (1.1,0.5) to [out=90, in=0] (0.85,0.8);
\draw[black, very thick] (0.85,0.8) to [out=180, in=45] (0.45,0.55);
\draw[black, very thick] (0,0) -- (0.45,0.55); 
\node at (0.5,-0.7) {$D_{l_+}$};
\draw[blue, dotted, very thick] (0.5,0.5) circle (0.75);
\end{tikzpicture}
\quad
\begin{tikzpicture}
\draw[black, very thick] (0,0) to [out=45, in=180] (0.85,0.2); 
\draw[black, very thick] (0.85,0.2) to [out=0, in=270] (1.1,0.5); 
\draw[black, very thick] (1.1,0.5) to [out=90, in=0] (0.85,0.8); 
\draw[black, very thick] (0.85,0.8) to [out=180, in=315] (0,1); 
\node at (0.5,-0.7) {$D_{l_0}$};
\draw[blue, dotted, very thick] (0.5,0.5) circle (0.75);
\end{tikzpicture}
\quad
\begin{tikzpicture}
\draw[black, very thick] (0,1) -- (0.4,0.5); 
\draw[black, very thick] (0.4,0.5) to [out=315, in=180] (0.85,0.2);
\draw[black, very thick] (0.85,0.2) to [out=0, in=270] (1.1,0.5);
\draw[black, very thick] (1.1,0.5) to [out=90, in=0] (0.85,0.8);
\draw[black, very thick] (0.85,0.8) to [out=180, in=45] (0.5,0.6);
%\draw[black, very thick, rounded corners] (0.,1) -- (0.35,0.55) -- (0.45,0.45) -- (0.6,0.3) --(0.95,0.2) -- (1.1,0.5) -- (0.95,0.8) -- (0.6,0.7) -- (0.45, 0.55);
\draw[black, very thick] (0.35,0.45) -- (0,0);
\draw[blue, dotted, very thick] (0.5,0.5) circle (0.75);
\node at (0.5,-0.7) {$D_{l_-}$};
%\node at (1.6,0.5) {${\bf \longrightarrow}$};
\end{tikzpicture}
}
\end{center} 
%\caption{Seven local differences in diagrams.} \label{fig:dubrovnik}
%\end{figure}

Denote by $R (D; a, t) \in \mathbb Z[a^{\pm1}, t^{\pm 1}]$ a specialization of the Dubrovnik polynomial corresponding to the substitution $z = t - t^{-1}$:   
\begin{equation}
R(D; a, t) =\mathfrak{D} (D; a, t-t^{-1}). \label{eqn1}
\end{equation}

Let $B$ be a bar diagram. For a bar $b \in B$, consider a skein tetrad of bar diagrams  $B$, $B_{c_+}$, $B_+$ and $B_-$ as presented below. 
\medskip 
%\begin{figure}[h]
\begin{center}
\scalebox{0.9}{
\begin{tikzpicture} 
\draw[black, very thick] (0.25,-0.25) -- (0.25,1.25);
\draw[black, very thick] (0.75,-0.25) -- (0.75,1.25);
  \draw[gray, line width = 3 pt] (0.25,0.5) -- (0.75,0.5);
\node at (0.5,-0.7) {$B$};
\draw[blue, dotted, very thick] (0.5,0.5) circle (0.75);
\draw[black, very thick] (3,0) -- (4,1);
\draw[black, very thick] (3,1) -- (3.4,0.6);
\draw[black, very thick] (3.6,0.4) -- (4,0);
\node at (3.5,-0.7) {$B_{c_+}$};
\draw[blue, dotted, very thick] (3.5,0.5) circle (0.75);
%
%\draw[black, ultra thick] (6,0) arc (-90:90: 0.3cm and 0.5cm);
%\draw[black, ultra thick] (7,1) arc (90:270: 0.3cm and 0.5cm);
\draw[black, very thick] (6.25,-0.25) -- (6.25,1.25);
\draw[black, very thick] (6.75,-0.25) -- (6.75,1.25);
\node at (6.5,-0.7) {$B_+$};
\draw[blue, dotted, very thick] (6.5,0.5) circle (0.75);
\draw[black, ultra thick] (10,1) arc (0:-180: 0.5cm and 0.3cm);
\draw[black, ultra thick] (10,0) arc (0:180: 0.5cm and 0.3cm);
\node at (9.5,-0.7) {$B_-$};
\draw[blue, dotted, very thick] (9.5,0.5) circle (0.75);
\end{tikzpicture} 
}
\end{center} 
%\caption{Skein tetrad $B_b$, $B_c$, $B_+$ and $B_-$ for a bar diagram.} \label{fig:tetrad}
%\end{figure} 
Consider the skein the relation for this tetrad: 
\begin{equation}
    R(B; a, t)= \frac{1}{t+t^{-1}}\Big(R(B_{c_+}; a,t)+t^{-1}R(B_+; a,t)+\frac{t-t^{-1}}{1-at} R(B_-; a,t) \Big). \label{eqn2}
\end{equation}
The  relation (\ref{eqn2}) admits to express $R(B; a, t)$ through $R$-polynomials of three bar diagrams $B_{c_+}$, $B_+$ and $B_-$ having one less bars than $B$.  By continuing the process  one can express $R(B; a, t)$ through $R$-polynomials of a finite number of diagrams without bars, i.e. usual diagrams of links.  Thus, for any bar diagram the polynomial $R(B; a, t)$ is defined through the specialization of the Dubrovnik polynomial $\mathfrak D(D; a, z)$ given by (\ref{eqn1}). 

Remark that if $D_L$ is a link diagram, i.~e. a diagram of a spatial graph having only 2-valent vertices, then by~\cite[Theorem~8]{Ya} 
$$
Y (D_L; A) = \mathfrak D (D_L; A^2, A - A^{-1}).
$$

Let $\mathcal{G}$ be a spatial graph and $D(\mathcal G)$ its diagram. The \emph{Jaeger polynomial} for $D(\mathcal G)$ is defined as follows:
\begin{equation}
J(D(\mathcal G); a, t) = R(B; a, t), \label{eqn3}
\end{equation}
where $B$ is the bar diagram corresponding to $D (\mathcal G)$. 

Below we will consider a specialization $\mathfrak{J} (D; A) \in \mathbb Z[A^{\pm 1}]$ of the Jaeger polynomial obtained by the substitution  $a = -A^3$ and $t = A$, i.e. suppose 
\begin{equation}
\mathfrak{J}(D; A) = J (D; -A^3, A). \label{eqn4} 
\end{equation}

%after the same specialization the Dubrovnik polynomial became the Kauffman bracket polynomial:

The following result was obtained in~\cite{Ja}. 

\begin{lemma} \cite{Ja} \label{lemma2.1} Let $D$ be a diagram of a spatial graph. Then 
\begin{enumerate}
\item $\mathfrak{J}(D;A)$ is a  flat isotopy invariant up to multiplying  by $(-A^4)^k$ for some integer $k$. 
\item The following skein relation holds for bar diagrams in above notations   
$$
R (B; -A^3, A) = R(B_+; -A^3; A) + \frac{1}{A^2 + A^{-2}} R(B_-; -A^3, A). 
$$
%where notations of Fig.~\ref{fig:tetrad} are used.
\end{enumerate}
\end{lemma}

Let $G$ be a planar graph, and $D$ a diagram of a spatial $G$-graph.  The following relation between Jaeger polynomial and Yamada polynomial was found in~\cite{H}. 

\begin{lemma} \cite[Prop.~5]{H}  \label{lemma2.2}
Let $G$ be a planar graph, $\mathcal G$ a spatial $G$-graph, and $D$ a diagram of $\mathcal G$. Then 
\begin{equation}
Y(D;A^4) = -(A^2+A^{-2})^{|E(G)|-|V(G)|+1} \mathfrak{J}(D;A), \label{eqn5}
\end{equation} 
where $E(G)$ is set of edges and $V(G)$ is set of vertices of $G$.
\end{lemma}

\section{Invariants of a theta-curve and related links} \label{section:theta} 

Since $\mathbb K_4$ has six subgraphs which are $\theta$-graphs, see Fig.~\ref{fig:K4cycles}, any spatial $\mathbb K_4$-graph contains  six constituent spatial $\theta$-graphs. In this section we recall some properties of spatial $\theta$-graphs. Recall that \emph{$\theta$-graph} is a graph with two vertices connected by three edges. There are three simple cycles in $\theta$-graph, see Fig.~\ref{fig:theta}. These three cycles will give three constituent knots in a spatial $\theta$-graph.  
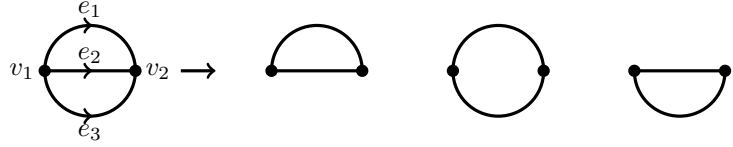
\begin{figure}[h] 
\begin{center}
%\unitlength=.05mm
\begin{tikzpicture}[scale=0.3] 
% 0 
\filldraw [very thick, black] (0,2) circle[radius=0.2cm];
\filldraw [very thick, black] (4,2) circle[radius=0.2cm];
\draw [very thick, black] (0,2)-- (4,2);
\draw [very thick, black] (2,2) circle[radius=2.cm];
 \draw [very thick, ->,shorten >=4pt] (6,2) -- (8,2); 
%1
\filldraw [very thick, black] (10,2) circle[radius=0.2cm];
\filldraw [very thick, black] (14,2) circle[radius=0.2cm];
\draw [very thick, black] (10,2)-- (14,2);;
 \draw [very thick, black]  (14,2) arc (0:180: 2.cm and 2.cm);
%2
\filldraw [very thick, black] (18,2) circle[radius=0.2cm];
\filldraw [very thick, black] (22,2) circle[radius=0.2cm];
\draw [very thick, black] (20,2) circle[radius=2.cm];
%3
\filldraw [very thick, black] (26,2) circle[radius=0.2cm];
\filldraw [very thick, black] (30,2) circle[radius=0.2cm];
\draw [very thick, black] (26,2)-- (30,2);
 \draw [very thick, black]  (30,2) arc (360:180: 2.cm and 2.cm);
 \node at (-1,2) {$v_1$};
  \node at (5,2) {$v_2$};
    \node at (2,4.7) {$e_1$};
        \node at (2,2.7) {$e_2$};
                \node at (2,-0.7) {$e_3$};
                 \draw [very thick, ->] (1.9,2) -- (2.1,2); 
                   \draw [very thick, ->] (1.9,4) -- (2.1,4); 
                    \draw [very thick, ->] (1.9,0) -- (2.1,0); 
\end{tikzpicture}
\end{center}
\caption{$\theta$-graph and its three cycles.} \label{fig:theta}
\end{figure}

The following result was obtained by Wolcott in~\cite{Wo}. 

\begin{theorem} \cite{Wo} \label{theorem:Wolcott}
For any three given knots $k_1$, $k_2$, and $k_3$ there exist a spatial $\theta$-graph such that these knots are realized as images  of pairs of edges. Moreover, knots $k_1$, $k_2$,  and $k_3$ do not determine spatial $\theta$-graph uniquely.
\end{theorem}

Let $D_{\Theta}$ be a diagram of some spatial $\theta$-graph $\Theta$ with vertices  $v_1, v_2$ and edges $e_1, e_2, e_3$. Assume that along the counter-clockwise direction, the edges appear in the order $(e_3, e_2, e_1)$ at $v_1$, and $(e_1, e_2, e_3)$ at $v_2$, see Fig.~\ref{fig:theta}.  Suppose that edges are oriented in the direction from $v_1$ to $v_2$.  The orientation of edges induces an orientations of arcs in the digram $D_{\Theta}$. The sign $\varepsilon (c) \in \{ +1, -1\}$ of any crossing $c \in D_{\Theta}$ is defined by the same rule as for diagrams of knots or links, see Fig.~\ref{fig:cross}. 

Given diagram $D$ of some spatial graph we can define a \emph{band diagram} by doubling edges of $D$ as it pictured in Fig.~\ref{fig:DtoB}-(a) and Fig.~\ref{fig:DtoB}-(b). In particular, given a diagram of a link,  we can define a \emph{2-parallel link diagram} by doubling edges of the diagram as pictured in Fig.~\ref{fig:DtoB}-(b).  

Let us construct band diagram for $D_{\Theta}$ and denote it by $L$. Note that $L$ can be considered as boundary of a closed two-punctured disk $S_{\Theta}$ such that $\Theta$ is a spine of $S_{\Theta}$, and $\partial S_{\Theta} = L$. It is easy to see that $L$ is the 3-component link, $L = \ell_1 \cup \ell_2 \cup \ell_3$. The link $L$ is oriented so that $\ell_1$ is homologous to $e_2 - e_3$, $\ell_2$ to $e_3 - e_1$, and $\ell_3$ to $e_1 - e_2$ on $S_{\Theta}$, see Fig.~\ref{fig:DtoL}-(a).   
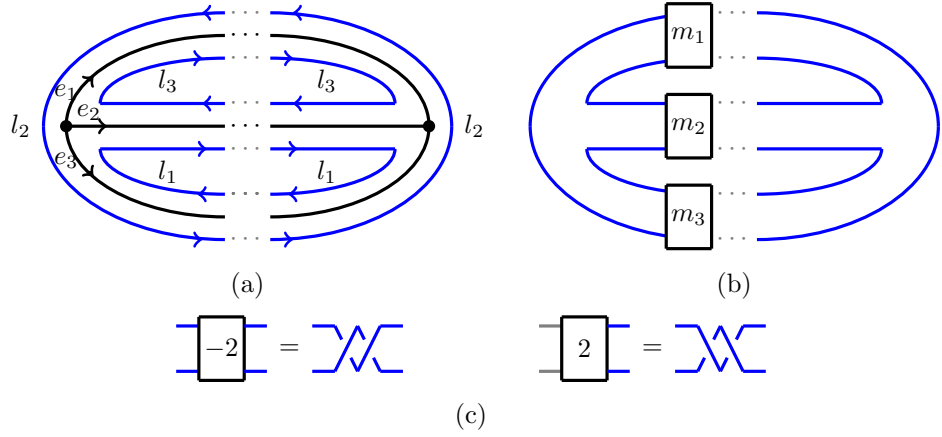
\begin{figure}[h] 
\begin{center}
%\unitlength=.05mm
\begin{tikzpicture}[scale=0.3] 
\filldraw [very thick, black] (0,2)  circle[radius=0.2cm];
\filldraw [very thick, black] (16,2) circle[radius=0.2cm];
\draw [very thick, black] (0,2)-- (7,2);
\node at (8,2) {$\cdots$};
\draw [very thick, black] (9,2)-- (16,2);
\draw [very thick, black]  (0,2) arc (180:270: 7.cm and 4.cm);
\node at (8,-1) {$\cdots$};
\draw [very thick, black]  (16,2) arc (360:270: 7.cm and 4.cm);
\draw [very thick, black]  (0,2) arc (180:90: 7.cm and 4.cm);
\node at (8,6) {$\cdots$};
\draw [very thick, black]  (16,2) arc (0:90: 7.cm and 4.cm);
\node at (0,3.5) {$e_1$};
\node at (1,2.6) {$e_2$};
\node at (0,0.5) {$e_3$};
\draw [very thick, ->] (1.6,2) -- (1.8,2); 
\draw [very thick, ->] (1.,4) -- (1.2,4.2); 
\draw [very thick, ->] (1,0) -- (1.2,-0.2); 
\draw [very thick, blue] (1.5,3)-- (7,3);
\draw [very thick, blue] (9,3)-- (14.5,3);
\draw [very thick, blue]  (1.5,3) arc (180:90: 5.5cm and 2.cm);  
\draw [very thick, blue]  (14.5,3) arc (0:90: 5.5cm and 2.cm);
\draw [very thick, blue] (1.5,1)-- (7,1);
\draw [very thick, blue] (9,1)-- (14.5,1);
\draw [very thick, blue]  (1.5,1) arc (180:270: 5.5cm and 2.cm);  
\draw [very thick, blue]  (14.5,1) arc (360:270: 5.5cm and 2.cm);
\draw [very thick, blue]  (-1,2) arc (180:90: 8.cm and 5.cm);  
\draw [very thick, blue]  (-1,2) arc (180:270: 8.cm and 5.cm);
\draw [very thick, blue]  (17,2) arc (0:90: 8.cm and 5.cm);  
\draw [very thick, blue]  (17,2) arc (360:270: 8.cm and 5.cm);
\draw [very thick, blue, ->] (6,1) -- (6.2,1);  
\draw [very thick, blue, ->] (10,1) -- (10.2,1);  
\draw [very thick, blue, ->] (6,-1) -- (5.8,-1);
\draw [very thick, blue, ->] (10,-1) -- (9.8,-1); 
 \node at (4.5,0) {$l_1$};
  \node at (11.5,0) {$l_1$};
\draw [very thick, blue, <-] (6,3) -- (6.2,3);  
\draw [very thick, blue, <-] (10,3) -- (10.2,3);  
\draw [very thick, blue, <-] (6,5) -- (5.8,5);
\draw [very thick, blue, <-] (10,5) -- (9.8,5); 
 \node at (4.5,4) {$l_3$};
  \node at (11.5,4) {$l_3$}; 
\draw [very thick, blue, <-] (6,7) -- (6.2,7);  
\draw [very thick, blue, <-] (10,7) -- (10.2,7);  
\draw [very thick, blue, <-] (6,-3) -- (5.8,-3);
\draw [very thick, blue, <-] (10,-3) -- (9.8,-3); 
\node at (-2,2) {$l_2$};
\node at (18,2) {$l_2$}; 
\node[gray] at (8,7) {$\cdots$};
\node[gray] at (8,5) {$\cdots$};
\node[gray] at (8,3) {$\cdots$}; 
\node[gray] at (8,1) {$\cdots$};
\node[gray] at (8,-1) {$\cdots$};
\node[gray] at (8,-3) {$\cdots$};
\node at (8,-5) {(a)};
\end{tikzpicture}
%%%
\quad 
%%%
\begin{tikzpicture}[scale=0.3] 
\draw [very thick, blue] (1.5,3)-- (7,3);
\draw [very thick, blue] (9,3)-- (14.5,3);
\draw [very thick, blue]  (1.5,3) arc (180:90: 5.5cm and 2.cm);  
\draw [very thick, blue]  (14.5,3) arc (0:90: 5.5cm and 2.cm);
\draw [very thick, blue] (1.5,1)-- (7,1);
\draw [very thick, blue] (9,1)-- (14.5,1);
\draw [very thick, blue]  (1.5,1) arc (180:270: 5.5cm and 2.cm);  
\draw [very thick, blue]  (14.5,1) arc (360:270: 5.5cm and 2.cm);
\draw [very thick, blue]  (-1,2) arc (180:90: 8.cm and 5.cm);  
\draw [very thick, blue]  (-1,2) arc (180:270: 8.cm and 5.cm);
\draw [very thick, blue]  (17,2) arc (0:90: 8.cm and 5.cm);  
\draw [very thick, blue]  (17,2) arc (360:270: 8.cm and 5.cm);
\filldraw[white] (5,4.6) -- (5,7.4) -- (7,7.4) -- (7,4.6) -- (5,4.6);
\draw [very thick, black] (5,4.6)-- (5,7.4);
\draw [very thick, black] (7,4.6)-- (7,7.4);
\draw [very thick, black] (5,4.6)-- (7,4.6);
\draw [very thick, black] (5,7.4)-- (7,7.4);
\node at (6,6) {$m_1$};
\filldraw[white] (5,0.6) -- (5,3.4) -- (7,3.4) -- (7,0.6) -- (5,0.6);
\draw [very thick, black] (5,0.6)-- (5,3.4);
\draw [very thick, black] (7,0.6)-- (7,3.4);
\draw [very thick, black] (5,0.6)-- (7,0.6);
\draw [very thick, black] (5,3.4)-- (7,3.4);
\node at (6,2) {$m_2$};
\filldraw[white] (5,-3.4) -- (5,-0.6) -- (7,-0.6) -- (7,-3.4) -- (5,-3.4);
\draw [very thick, black] (5,-3.4)-- (5,-0.6);
\draw [very thick, black] (7,-3.4)-- (7,-0.6);
\draw [very thick, black] (5,-3.4)-- (7,-3.4);
\draw [very thick, black] (5,-0.6)-- (7,-0.6);
\node at (6,-2) {$m_3$};
\node[gray] at (8,7) {$\cdots$};
\node[gray] at (8,5) {$\cdots$};
\node[gray] at (8,3) {$\cdots$}; 
\node[gray] at (8,1) {$\cdots$};
\node[gray] at (8,-1) {$\cdots$};
\node[gray] at (8,-3) {$\cdots$};
\node at (8,-5) {(b)};
\end{tikzpicture}

\smallskip 
%%%
\begin{tikzpicture}[scale=0.3] 
\draw [very thick, blue] (-13,5) -- (-12,5);
\draw [very thick, blue] (-13,7) -- (-12,7);
\draw [very thick, black] (-12,4.6)-- (-12,7.4);
\draw [very thick, black] (-10,4.6)-- (-10,7.4);
\draw [very thick, black] (-12,4.6)-- (-10,4.6);
\draw [very thick, black] (-12,7.4)-- (-10,7.4);
\node at (-11,6) {$-2$};
\draw [very thick, blue] (-10,5) -- (-9,5);
\draw [very thick, blue] (-10,7) -- (-9,7);
\node at (-8,6) {$=$};
\draw [very thick, blue] (-7,5) -- (-6,5);
\draw [very thick, blue] (-7,7) -- (-6,7);
\draw [very thick, blue] (-4,5) -- (-3,5);
\draw [very thick, blue] (-4,7) -- (-3,7);
\draw [very thick, blue] (-6,5)-- (-5,7);
\draw [very thick, blue] (-5,5)-- (-4,7);
\draw [very thick, blue] (-6,7)-- (-5.7,6.4);
\draw [very thick, blue] (-5,7)-- (-4.7,6.4);
\draw [very thick, blue] (-5,5)-- (-5.3,5.6);
\draw [very thick, blue] (-4,5)-- (-4.3,5.6);
\draw [very thick, gray] (3,5) -- (4,5);
\draw [very thick, gray] (3,7) -- (4,7);
\draw [very thick, black] (4,4.6)-- (4,7.4);
\draw [very thick, black] (6,4.6)-- (6,7.4);
\draw [very thick, black] (4,4.6)-- (6,4.6);
\draw [very thick, black] (4,7.4)-- (6,7.4);
\node at (5,6) {$2$};
\draw [very thick, blue] (6,5) -- (7,5);
\draw [very thick, blue] (6,7) -- (7,7);
\node at (8,6) {$=$};
\draw [very thick, blue] (9,5) -- (10,5);
\draw [very thick, blue] (9,7) -- (10,7);
\draw [very thick, blue] (10,7)-- (11,5);
\draw [very thick, blue] (11,7)-- (12,5);
\draw [very thick, blue] (12,5) -- (13,5);
\draw [very thick, blue] (12,7) -- (13,7);
\draw [very thick, blue] (10,5) -- (10.3,5.6);
\draw [very thick, blue] (11,7) -- (10.7,6.4);
\draw [very thick, blue] (11,5) -- (11.3,5.6);
\draw [very thick, blue] (12,7) -- (11.7,6.4);
\node at (0,3) {(c)};
\end{tikzpicture}
\end{center}
\caption{Construction of links $L$ and $L(m_1, m_2, m_3)$.} \label{fig:DtoL}
\end{figure}

By Theorem~\ref{theorem:KSWZ} there is a band diagram of $\Theta$ with zero Seifert form, and the diagram with this property is unique up to ambient isotopy. We describe the construction of this band diagram following~\cite{KSWZ}.  

Let $C(e_i \cap e_i)$, $i, j = 1,2,3$,  be the set of crossing in $D_{\Theta}$ where oriented edges $e_i$ and $e_j$ meet; we also admit self-intersections, $j=i$. Denote  
$$
v_{ij} = \sum_{c \in C(e_i \cap e_j)} \, \varepsilon (c),
$$ 
where $i,j = 1,2,3$. Remark that $v_{ij} = v_{ji}$. It was shown in the proof of~\cite[Theorem~2.4]{KSWZ} that there is a unique choice of the number of half twists $m_i$, $i=1, 2, 3$, on bands to make the corresponding band surfaces be of zero Seifert form, namely,  
\begin{equation}
\begin{cases} 
\begin{array}{l} 
m_1 =-2v_{11}+v_{12}+v_{13}-v_{23}, \cr 
m_2 =-2v_{22}+v_{12}+v_{23}-v_{13}, \cr  
m_3 =-2v_{33}+v_{13}+v_{23}-v_{12}. 
\end{array}
\end{cases} \label{eqn:m}
\end{equation} 

\begin{definition}
Integers $m_1, m_2, m_3$ satisfying (\ref{eqn:m}) are said to be \emph{twist parameters} of the diagram~$D_{\Theta}$. 
\end{definition}

Let $L (m_1, m_2, m_3)$ be the link whose diagram is obtained from $D_{\Theta}$ by adding $m_i$ half twists around edge $e_i$, $i = 1,2,3$, see Fig.~\ref{fig:DtoL}. By choosing of $m_1$, $m_2$ and $m_3$ the Seifert form of surface corresponding to $L (m_1, m_2, m_3)$ is zero, which implies a uniqueness of this link, see Theorem~\ref{theorem:KSWZ}. The link $L (m_1, m_2, m_3)$ is the \emph{associated link} of $\Theta$~\cite{KSWZ} and we denote it by $\mathcal L_{\Theta}$. 

%\begin{figure}[h]
%    \includegraphics[width=0.5\textwidth]{fig-11.jpg}
%\caption{Additional twists in the diagram of the spatial $\theta$-graph.} \label{fig:thetalink}
%\end{figure}

The following result gives a formula for the Jones polynomial for the associated link $\mathcal L_{\Theta} = L(m_1, m_2, m_3)$ with \emph{even} integers $m_1$, $m_2$, $m_3$. 

\begin{lemma} \cite[Prop.~6]{H} \label{lemma3.1}
The following equality holds
  $$ 
  \begin{array}{lll} 
    V(L(m_1,m_2,m_3)) & = & A^{4(m_1+m_2+m_3)} \Big[\langle L \rangle+\sum\limits_{i=1}^3 \dfrac{1-(A^{-4})^{m_i}}{\varphi}\langle l^{(2)}_i\rangle \cr  & & +\dfrac{1}{\varphi^2}(2-\sum\limits_{i=1}^3 (A^{-4})^{m_i}+(A^{-4})^{m_1+m_2+m_3})\Big],
    \end{array}
$$
where  $l_i^{(2)}$ is the 2-component link 2-parallel of $l_i$,  $m_1$, $m_2$, $m_3$ are even, and $\varphi = A^2 + A^{-2}$.  
\end{lemma}
Using the proof of~\cite[Prop.~6]{H} it is easy to get a similar formula for the general case:

\begin{lemma} \label{lemma3.1_1}
The following equality holds
  $$ 
  \begin{array}{lll} 
    V(L(m_1,m_2,m_3)) & = & (-A^4)^{m_1+m_2+m_3} \Big[\langle L \rangle+\sum\limits_{i=1}^3 \dfrac{1-(-A^{-4})^{m_i}}{\varphi}\langle l^{(2)}_i\rangle \cr  & & +\dfrac{1}{\varphi^2}(2-\sum\limits_{i=1}^3 (-A^{-4})^{m_i}+(-A^{-4})^{m_1+m_2+m_3})\Big],
    \end{array}
$$
where  $l_i^{(2)}$ is the 2-component link 2-parallel of $l_i$ and $\varphi = A^2 + A^{-2}$.  
\end{lemma}

\begin{definition}
Let $D_{\Theta}$ be a diagram of a spatial $\theta$-graph $\Theta$ and $m_1, m_2, m_3$ be twist parameters for $D_{\Theta}$. A polynomial 
$$
\widetilde{\mathfrak{J}}(D_{\Theta})=(-A^4)^{m_1+m_2+m_3}\mathfrak{J}(D_{\Theta}). 
$$ 
is called the \emph{normalized Jaeger polynomial} for $D_\Theta$.
\end{definition} 

It was shown in~\cite{H, HJ} that $\widetilde{\mathfrak{J}}(D_{\Theta})$ is an ambient isotopy invariant  of the spatial $\theta$-graph $\Theta$. 

\begin{lemma} \cite[Prop.~7]{H} \label{lemma3.2}
The following equality holds
$$
\widetilde{\mathfrak{J}}(D_\Theta) = (-A^4)^{m_1+m_2+m_3} \Big[ \langle L \rangle +\dfrac{1}{\varphi}\sum\limits_{i=1}^3 \langle l_i^{(2)}\rangle +\dfrac{2}{\varphi} \Big].
$$
\end{lemma}

\begin{definition} 
Let $D_K$ be a diagram of a knot $K$ and $w(D_K)$ be the writhe number of $D_K$. The polynomial 
$$
\widetilde{\mathfrak{J}}(D_K; A)=A^{-8w(D_K)}\mathfrak{J}(D_K; A) 
$$ 
is called the  \emph{normalized Jaeger polynomial} for a knot diagram $D_K$.   
\end{definition} 

It was shown in~\cite{H} that  $\widetilde{\mathfrak{J}}(D_K)$ is a pliable isotopic invariant of  the  knot $K$ having diagram $D_K$.
 
A relation between the normalized Jaeger polynomial of a spatial $\theta$-graph and polynomials of constituent knots and the associated link is presented in the following theorem.  

\begin{theorem} \cite[Th.~8]{H} \label{theorem-theta}
Let $\Theta$ be a spatial $\theta$-graph with constituent knots  $ \mathcal K_1, \mathcal K_2, \mathcal K_3 $ and the associated link $\mathcal L_{\Theta}$. Then 
$$
\widetilde{\mathfrak{J}}(\Theta)-V(\mathcal{L}_{\Theta})=\dfrac{1}{\varphi}\sum\limits_{i=1}^3 \widetilde{\mathfrak{J}}(\mathcal{K}_i)-\dfrac{1}{\varphi^2},
$$
where $\varphi = A^2 + A^{-2}$. 
\end{theorem}

\section{Jones polynomial of the associated link of a spatial $\mathbb{K}_4$-graph} \label{section:K4}

Let $D_{\Omega}$ be a diagram of a spatial $\mathbb K_4$-graph $\Omega$. Denote  edges of $\Omega$ by $a_1,a_2,\ldots,a_6$ and suppose they are oriented as in Fig.~\ref{fig:K4cycles}. We construct band diagram $D_{\Omega}^{(2)}$ for $D_{\Omega}$ according to the same rules as for spatial $\theta$-graph, see Fig.~\ref{fig:DtoL4}. Denote this diagram by $L$. Assume that the link $L$ is oriented in such a way that the boundaries of bands are oppositely directed, see Fig.~\ref{fig:DtoL4}.
\begin{figure}[h] 
\begin{center}
%\unitlength=.05mm
\begin{tikzpicture}[scale=0.4] 
% 
%\draw[step=5.mm, gray, very thin] (-10,0) grid (100.mm,80.mm);
\filldraw [very thick, black] (0,1)  circle[radius=0.1cm];
\filldraw [very thick, black] (0,5) circle[radius=0.1cm];
\filldraw [very thick, black] (3.2,7.4) circle[radius=0.1cm];
\filldraw [very thick, black] (-3.2,7.4) circle[radius=0.1cm];
\draw [very thick, black] (0,5)-- (0,1);
\draw [very thick, black] (0,5)-- (-3.2,7.4); 
\draw [very thick, black] (0,5)-- (3.2,7.4);
\draw [very thick, black]  (0,5) circle [radius=4.];
\draw [very thick, blue]  (0,5) circle [radius=4.5];
\draw[very thick, blue] (-2.8,7.5) .. controls (-1.4,8.5) .. (0,8.6).. controls (1.4, 8.5 ) .. (2.8,7.5);
\draw [very thick, blue] (0,5.5)-- (-2.8,7.5);
\draw [very thick, blue] (0,5.5)-- (2.8,7.5);
\draw[very thick, blue] (3.2,7) .. controls (3.8,5) .. (3.3,3.4).. controls (2.2, 2) .. (0.4,1.3);
\draw [very thick, blue] (0.4,4.7)-- (0.4,1.3);
\draw [very thick, blue] (0.4,4.7)-- (3.2,7);
\draw[very thick, blue] (-3.2,7) .. controls (-3.8,5) .. (-3.3,3.4).. controls (-2.2, 2) .. (-0.4,1.3);
\draw [very thick, blue] (-0.4,4.7)-- (-0.4,1.3);
\draw [very thick, blue] (-0.4,4.7)-- (-3.2,7);
\draw [very thick, black, ->] (0,5) -- (0,4); 
\draw [very thick, black, ->] (0,5) -- (1,5.7); 
\draw [very thick, black, ->] (0,5) -- (-1,5.7); 
\draw [very thick, black, ->] (3.2,7.4) -- (3.8,6.5); 
\draw [very thick, black, ->] (0,1) -- (-0.7,1.1); 
\draw [very thick, black, ->] (-3.2,7.4) -- (-2.7,8);
\filldraw [very thick, white] (-2,6.5)  circle[radius=0.5cm];
\node at (-1.5,7.2) {$a_1$};  \node at (-2,6.5) {$\cdot$};  \node at (-2.2,6.7) {$\cdot$};   \node at (-1.8,6.3) {$\cdot$}; 
\filldraw [very thick, white] (2,6.5)  circle[radius=0.5cm];
\node at (1.5,7.2) {$a_2$};  \node at (2,6.5) {$\cdot$}; \node at (2.2,6.7) {$\cdot$};   \node at (1.8,6.3) {$\cdot$}; 
\filldraw [very thick, white] (0,3)  circle[radius=0.5cm];
\node at (1,3) {$a_3$}; \node at (0,3.2) {$\cdot$}; \node at (0,2.8) {$\cdot$}; \node at (0,3) {$\cdot$};
\filldraw [very thick, white] (3.5,3)  circle[radius=0.5cm];
\node at (4.5,2.8) {$a_4$}; \node at (3.5,3) {$\cdot$}; \node at (3.6,3.2) {$\cdot$}; \node at (3.4,2.8) {$\cdot$};
\filldraw [very thick, white] (-3.5,3)  circle[radius=0.5cm];
\node at (-2.5,3.2) {$a_5$}; \node at (-3.5,3) {$\cdot$}; \node at (-3.6,3.2) {$\cdot$}; \node at (-3.4,2.8) {$\cdot$};
\filldraw [very thick, white] (0,9.05)  circle[radius=0.54cm];
\node at (0,8) {$a_6$};  \node at (0,9) {$\cdot$};   \node at (-0.2,9) {$\cdot$};   \node at (0.2,9) {$\cdot$}; 
\draw [very thick, blue, ->] (3.7,5) -- (3.6,4.5); 
\draw [very thick, blue, ->] (-3.6,4.5) -- (-3.6,5); 
\draw [very thick, blue, ->] (1,8.5) -- (1.5,8.4); 
\draw [very thick, blue, ->] (2,9.05) -- (1.5,9.25); 
\node at (0,-1) {(a)};
\end{tikzpicture}
%%%
\qquad \qquad 
%%%
\begin{tikzpicture}[scale=0.4] 
\draw [very thick, blue]  (0,5) circle [radius=4.5];
\draw[very thick, blue] (-2.8,7.5) .. controls (-1.4,8.5) .. (0,8.6).. controls (1.4, 8.5 ) .. (2.8,7.5);
\draw [very thick, blue] (0,5.5)-- (-2.8,7.5);
\draw [very thick, blue] (0,5.5)-- (2.8,7.5);
\draw[very thick, blue] (3.2,7) .. controls (3.8,5) .. (3.3,3.4).. controls (2.2, 2) .. (0.4,1.3);
\draw [very thick, blue] (0.4,4.7)-- (0.4,1.3);
\draw [very thick, blue] (0.4,4.7)-- (3.2,7);
\draw[very thick, blue] (-3.2,7) .. controls (-3.8,5) .. (-3.3,3.4).. controls (-2.2, 2) .. (-0.4,1.3);
\draw [very thick, blue] (-0.4,4.7)-- (-0.4,1.3);
\draw [very thick, blue] (-0.4,4.7)-- (-3.2,7);
\filldraw [very thick, white] (-2,6.5)  circle[radius=0.6cm];
 \node at (-2,6.5) {$\cdot$};  \node at (-2.2,6.7) {$\cdot$};   \node at (-1.8,6.3) {$\cdot$}; 
 \filldraw [very thick, white] (-0.8,5.5)  circle[radius=0.6cm];
\draw [thick, black] (-0.8,5.5)  circle[radius=0.6cm]; \node at (-0.8,5.5) {$n_1$}; 
\filldraw [very thick, white] (2,6.5)  circle[radius=0.6cm];
 \node at (2,6.5) {$\cdot$}; \node at (2.2,6.7) {$\cdot$};   \node at (1.8,6.3) {$\cdot$}; 
\filldraw [very thick, white] (0.8,5.5)  circle[radius=0.6cm];
\draw [thick, black] (0.8,5.5)  circle[radius=0.6cm]; \node at (0.8,5.5) {$n_2$};
\filldraw [very thick, white] (0,2.8)  circle[radius=0.6cm];
\node at (0,2.8) {$\cdot$}; \node at (0,2.6) {$\cdot$}; \node at (0,3) {$\cdot$};
\filldraw [very thick, white] (0,4)  circle[radius=0.6cm];
\draw [thick, black] (0,4)  circle[radius=0.6cm]; \node at (0,4) {$n_3$};
\filldraw [very thick, white] (3.5,3)  circle[radius=0.6cm];
\node at (3.5,3) {$\cdot$}; \node at (3.6,3.2) {$\cdot$}; \node at (3.4,2.8) {$\cdot$};
\filldraw [very thick, white] (3.8,6.5)  circle[radius=0.6cm];
\draw [thick, black] (3.8,6.5)  circle[radius=0.6cm]; \node at (3.8,6.5) {$n_4$};
\filldraw [very thick, white] (-3.5,3)  circle[radius=0.6cm];
\node at (-3.5,3) {$\cdot$}; \node at (-3.6,3.2) {$\cdot$}; \node at (-3.4,2.8) {$\cdot$};
\filldraw [very thick, white] (-1.2,1.1)  circle[radius=0.6cm]; 
\draw [thick, black] (-1.2,1.1)  circle[radius=0.6cm]; \node at (-1.2,1.1) {$n_5$};
\filldraw [very thick, white] (0,9.05)  circle[radius=0.6cm];
 \node at (0,9) {$\cdot$};   \node at (-0.2,9) {$\cdot$};   \node at (0.2,9) {$\cdot$}; 
 \filldraw [very thick, white] (-2.4,8.3)  circle[radius=0.6cm];
\draw [thick, black] (-2.4,8.3)  circle[radius=0.6cm];  \node at (-2.4,8.3) {$n_6$};
\node at (0,-1) {(b)};
\end{tikzpicture}
\end{center}
\caption{Construction of links $L$ and $L(n_1, n_2, n_3, n_4, n_5, n_6)$.}  \label{fig:DtoL4}
\end{figure}
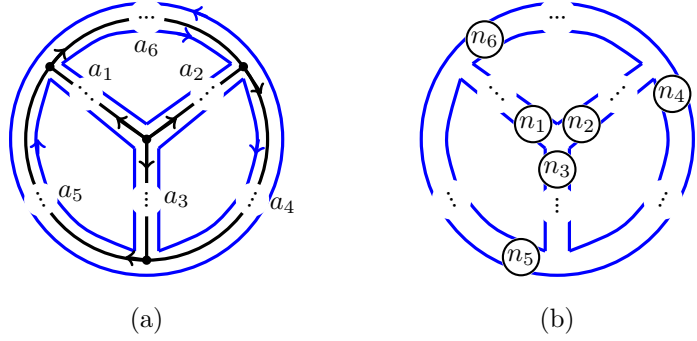

Let $C(a_i \cap a_j)$, $i, j = 1,\dots, 6$,  be the set of crossing in $D_{\Omega}$ where oriented edges $a_i$ and $a_j$ meet. Denote  
$$
w_{ij} = \sum_{c \in C(a_i \cap a_j)} \, \varepsilon (c),
$$ 
where $i,j = 1, \ldots, 6$. 
It was shown in the proof of~\cite[Theorem~2.4]{KSWZ} that there is the unique choice of the number of half twists $n_i$, $i=1, \ldots,  6$, on bands to make the corresponding band surfaces be of zero Seifert form, namely, 
\begin{equation}
    \begin{cases}
   n_1=-2w_{11}-w_{23}-w_{25}+w_{21}+w_{13}+w_{15}+w_{36}-w_{16}+w_{56}; \\
   n_2=-2w_{22}-w_{24}+w_{14}+w_{46}+w_{23}-w_{13}-w_{36}+w_{12}+w_{26};\\
   n_3=-2w_{33}+w_{34}-w_{14}+w_{45}+w_{23}+w_{25}-w_{12}+w_{13}-w_{35};\\
   n_4=-2w_{44}-w_{24}+w_{34}-w_{46}+w_{36}-w_{26}-w_{45}-w_{25}+w_{35};\\
   n_5=-2w_{55}-w_{35}+w_{15}-w_{36}+w_{16}-w_{56}-w_{34}+w_{14}-w_{45};\\
   n_6=-2w_{66}-w_{16}+w_{26}+w_{25}-w_{15}-w_{56}+w_{24}-w_{14}-w_{46}.
\end{cases}
\label{eqnni}
\end{equation}

\begin{definition}
Integers $n_1, \ldots, n_6$ satisfying (\ref{eqnni}) are said to be \emph{twist parameters} for the diagram~$D_{\Omega}$. 
\end{definition}

Let $L(n_1,n_2,n_3,n_4,n_5,n_6)$ be the link whose diagram is obtained from $D_{\Omega}$ by adding  $n_i$ half twists around each edge $a_i$, $i=1,\dots,6$,  see Fig.~\ref{fig:DtoL4}. Let us choose such numbers $n_1, \ldots, n_6$ that satisfy the system~(\ref{eqnni}). Then the Seifert form of surface corresponding to  $L(n_1,n_2,n_3,n_4,n_5,n_6)$ is zero, which implies a uniqueness of this link,  see Theorem~\ref{theorem:KSWZ}. 

The link $L(n_1,n_2,n_3,n_4,n_5,n_6)$ is the \emph{associated link} of $\Omega$~\cite{KSWZ} and we denote it by $\mathcal L_{\Omega}$. 

Now we will calculate the Jones polynomial of $\mathcal L_{\Omega} = L(n_1,\dots,n_6)$.  Recall that $V(D;A)=(-A^3)^{-w(D)}\langle D\rangle$, where $w(D)$ is writhe of $D$. Since boundaries of the band surface corresponding to  $\mathcal L_{\Theta}$ are oppositely directed, the writhe number $w$ of its diagram depends only on half twists added to $L_{\Omega}$. Thus, $w=-\sum\limits_{i=1}^6 n_i$.
 
Let us use the following notations for bracket polynomials $b_n$ and $b_{\infty}$ of link diagrams which have difference only in the number of half twists as pictired below.
\medskip  
\begin{center}
\begin{tikzpicture}
\begin{knot}[
  consider self intersections=true,
%  draft mode=crossings,
  flip crossing=2,
  only when rendering/.style={
%    show curve controls
  }
  ]
  \node at (1.1,0.35) {$b_n=\langle$};
\strand (1.7,0.5) -- (1.9,0.5);
\strand (1.7,0.2)--(1.9,0.2);
\strand (1.9,0.7) -- (2.6,0.7);
\strand (1.9,0) -- (2.6,0);
\strand (1.9,0) -- (1.9,0.7);
\strand (2.6,0) -- (2.6,0.7);
\strand (2.6,0.5) -- (2.8,0.5);
\strand (2.6,0.2) -- (2.8,0.2);
\node at (2.25,0.35) {$n$};
\node at (3.9,0.35) {$\rangle$,\qquad $b_{\infty}=\langle$};
\strand (5,0.6) .. controls (5.3,0.45) .. (5.3,0.35) .. controls (5.3,0.25) .. (5,0.1);
\strand (5.7,0.6) .. controls (5.4,0.45) .. (5.4,0.35) .. controls (5.4,0.25) .. (5.7,0.1);
\node at (5.9,0.35) {$\rangle$.};
\end{knot}
\end{tikzpicture}
\end{center}
It is shown in the proof of~\cite[Prop.~6]{H} that for any \emph{even} integer $n$ the following relation holds 
\begin{equation}
 b_n=A^{n} b_0+f_n b_{\infty},   \label{eqn6}
\end{equation}
where 
\begin{equation}
f_n=\dfrac{A^{n-2}(1-(-A^{-4})^{n})}{1+A^{-4}}.  \label{eqn:7}
\end{equation}
Using proof of~~\cite[Prop.~6]{H} it is easy to prove that this relation holds for any integer.

Using the relation~(\ref{eqn6}) step by step  we obtain the formula for bracket polynomial: 
\begin{equation}
\begin{array}{rl}
 \langle L(n_1, n_2, n_3, n_4, n_5, n_6) \rangle  = &  A^{2n_1}\langle L(0,n_2,n_3,n_4,n_5,n_6)\rangle \cr 
 &  +f_{n_1}\langle L(\infty,n_2,n_3,n_4,n_5,n_6)\rangle \cr  
 = & A^{2n_1} \left[ A^{2n_2}\langle L(0,0,n_3,n_4,n_5,n_6)\rangle \right. \cr 
 & \left. + f_{n_2}\langle L(0,\infty,n_3,n_4,n_5,n_6)\rangle \right] \cr 
 & + f_{n_1}\langle L(\infty,n_2,n_3,n_4,n_5,n_6)\rangle \cr
=  & A^{2(n_1+n_2)} \left[ A^{2n_3}\langle L(0,0,0,n_4,n_5,n_6)\rangle \right. \cr 
& \left. +f_{n_3}\langle L(0,0,\infty,n_4,n_5,n_6)\rangle \right]  \cr 
& + A^{2n_1} f_{n_2}\langle L(0,\infty,n_3,n_4,n_5,n_6)\rangle \cr 
& + f_{n_1}\langle L(\infty,n_2,n_3,n_4,n_5,n_6)\rangle \cr 
  \cdots & \cr 
=  & A^{2(n_1+n_2+\cdots+n_6)}\langle L(0,0,0,0,0,0)\rangle \cr 
& + A^{2(n_1+n_2+\cdots+n_5)}f_{n_6}\langle L(0,0,0,0,0,\infty)\rangle \cr  
& + A^{2(n_1+n_2+n_3+n_4)}f_{n_5}\langle L(0,0,0,0,\infty,n_6)\rangle \cr 
& +A^{2(n_1+n_2+n_3)}f_{n_4}\langle L(0,0,0,\infty,n_5,n_6)\rangle \cr  
& + A^{2(n_1+n_2)}f_{n_3}\langle L(0,0,\infty,n_4,n_5,n_6)\rangle \cr 
& +A^{2n_1}f_{n_2}\langle L(0,\infty,n_3,n_4,n_5,n_6)\rangle \cr 
& + f_{n_1}\langle L(\infty,n_2,n_3,n_4,n_5,n_6)\rangle.
\end{array}
\label{eqn:brL}
\end{equation}

Let $\Theta_i$ be a spatial $\theta$-graph obtained from the spatial $\mathbb K_4$-graph $\Omega$ by deleting the edge $a_i$. Denote by $\Theta_i^{(2)}$ the band diagram of $\Theta_i$ with the orientation induced by the orientation of $L$. Let $\Theta_i^{(2)}(k_{i1},k_{i2},k_{i3})$ be the diagram obtained from $\Theta_i^{(2)}$ by adding $k_{ij}$ half twists around the  $j$-th edge. By the construction, any $\Theta_i^{(2)}(k_{i1},k_{i2},k_{i3})$ is a link with three components. Note that one can rewrite diagrams from relation~(\ref{eqn:brL}) as follows. 
\begin{equation}
    \begin{array}{l} 
L(\infty,n_2,n_3,n_4,n_5,n_6)  =  \Theta_1^{(2)}(n_5+n_6,n_2+n_3,n_4), \cr
L(0,\infty,n_3,n_4,n_5,n_6)  =  \Theta_2^{(2)}(n_4+n_6,n_3,n_5), \cr
L(0,0,\infty,n_4,n_5,n_6)  =  \Theta_3^{(2)}(n_4+n_5,0,n_6), \cr 
L(0,0,0,\infty,n_5,n_6)  =  \Theta_4^{(2)}(n_6,0,n_5), \cr
L(0,0,0,0,\infty,n_6)  =  \Theta_5^{(2)}(0,0,n_6), \cr
L(0,0,0,0,0,\infty)  =  \Theta^{(2)}_6. 
\end{array}
\label{eqn:LiTh}
\end{equation}
%Let us substitute these expressions into~\ref{eqn:brL}:
%\begin{equation}
%\begin{array}{rl}
 %\langle L(n_1,n_2,n_3,n_4,n_5,n_6)\rangle  = & A^{2(n_1+n_2+\cdots+n_6)}\langle L\rangle \cr 
%& + A^{2(n_1+n_2+\cdots+n_5)}f_{n_6}\langle \Theta^{(2)}_6\rangle \cr  
%& + A^{2(n_1+n_2+n_3+n_4)}f_{n_5}\langle \Theta_5^{(2)}(0,0,n_6)\rangle \cr 
%& +A^{2(n_1+n_2+n_3)}f_{n_4}\langle \Theta_4^{(2)}(n_6,0,n_5)\rangle \cr  
%& + A^{2(n_1+n_2)}f_{n_3}\langle \Theta_3^{(2)}(n_4+n_5,0,n_6)\rangle \cr 
%& +A^{2n_1}f_{n_2}\langle \Theta_2^{(2)}(n_4+n_6,n_3,n_5)\rangle \cr 
%& + f_{n_1}\langle \Theta_1^{(2)}(n_5+n_6,n_2+n_3,n_4)\rangle.
%\end{array}
%\label{eqn:brLTh}
%\end{equation}
Note that the Jones polynomial of $\Theta_i^{(2)}(k_{i1},k_{i2},k_{i3})$ can be expressed by~(\ref{eqnJones}) as following:
\begin{equation}
V(\Theta_i^{(2)}(k_{i1},k_{i2},k_{i3}))=(-A^3)^{k_{i1}+k_{i2}+k_{i3}}\langle \Theta_i^{(2)}(k_{i1},k_{i2},k_{i3})\rangle.
\label{eqn:VTheta}
\end{equation}
From~(\ref{eqn:VTheta}) for the particular cases we get relations 
\begin{equation}
\begin{array}{l}
\langle\Theta_1^{(2)}(n_5+n_6,n_2+n_3,n_4)\rangle  =  (-A^3)^{-(n_2+n_3+\cdots+n_6)}V(\Theta_1^{(2)}(n_5+n_6,n_2+n_3,n_4)), \cr
\langle\Theta_2^{(2)}(n_4+n_6,n_3,n_5)\rangle  =  (-A^3)^{-(n_3+n_4+n_5+n_6)}V(\Theta_2^{(2)}(n_4+n_6,n_3,n_5)), \cr
\langle\Theta_3^{(2)}(n_4+n_5,0,n_6)\rangle  =  (-A^3)^{-(n_4+n_5+n_6)}V(\Theta_3^{(2)}(n_4+n_5,0,n_6)), \cr
\langle\Theta_4^{(2)}(n_6,0,n_5)\rangle  =  (-A^3)^{-(n_5+n_6)}V(\Theta_4^{(2)}(n_6,0,n_5)), \cr 
\langle\Theta_5^{(2)}(0,0,n_6)\rangle  =  (-A^3)^{-n_6}V(\Theta_5^{(2)}(0,0,n_6)), \cr
 \langle\Theta^{(2)}_6\rangle  =  V(\Theta^{(2)}_6).
\end{array}
\label{eqn:brTheta}
\end{equation}

Let  $l_{ijk}$ be a cycle of $\mathbb K_4$ formed by three edges  $\{ a_i, a_j, a_k\}$, and $l_{ijkl}$ a cycle formed by four edges  $\{ a_i, a_j, a_k, a_l\}$. Denote by $l_{ijk}^{(2)}$ and $l_{ijkl}^{(2)}$ the corresponding 2-parallel diagrams with the orientation induced by the orientation of $L$. 

Let us compute the Jones polynomials from~(\ref{eqn:brTheta}) by Lemma~\ref{lemma3.1_1}:
\begin{equation}
\begin{array} {l}
V(\Theta_1^{(2)}(n_5+n_6,n_2+n_3,n_4))   =   (-A^4)^{n_2+n_3+\cdots+n_6} \Big[\langle \Theta_1^{(2)}\rangle  \cr 
 \qquad + \frac{1}{\varphi} \Big( (1 -(-A^{-4})^{n_5+n_6})\langle l_{234}^{(2)}\rangle  +(1  -(-A^{-4})^{n_2+n_3})\langle l_{456}^{(2)}\rangle  
 +(1-(-A^{-4})^{n_4})\langle l_{2356}^{(2)}\rangle \Big) \cr 
 \qquad + \frac{1}{\varphi^2} \Big( 2-(-A^{-4})^{n_5+n_6}-(-A^{-4})^{n_2+n_3}-(-A^{-4})^{n_4}+(-A^{-4})^{n_2+n_3+\cdots+n_6} \Big) \Big], \cr
V(\Theta_2^{(2)}(n_4+n_6,n_3,n_5))   =   (-A^4)^{n_3+n_4+n_5+n_6} \Big[ \langle\Theta_2^{(2)}\rangle \cr 
\qquad  +\frac{1}{\varphi} \Big((1- (-A^{-4})^{n_4+n_6})\langle l_{135}^{(2)}\rangle+(1-(-A^{-4})^{n_3})\langle l_{456}^{(2)}\rangle+(1-(-A^{-4})^{n_5})\langle l_{1346}^{(2)}\rangle \Big) \cr 
\qquad  +\frac{1}{\varphi^2} \Big(2-(-A^{-4})^{n_4+n_6}-(-A^{-4})^{n_3}-(-A^{-4})^{n_5}+(-A^{-4})^{n_3+n_4+n_5} \Big) \Big], \cr 
V(\Theta_3^{(2)}(n_4+n_5,0,n_6))   =   (-A^4)^{n_4+n_5+n_6} \Big[ \langle\Theta_3^{(2)}\rangle   \cr 
\qquad  + \frac{1}{\varphi} \Big( (1-(-A^{-4})^{n_4+n_5})\langle l_{126}^{(2)}\rangle+ (1-(-A^{-4})^{n_6})\langle l^{(2)}_{1245}\rangle \Big) \cr 
 \qquad +\frac{1}{\varphi^2} \Big(1-(-A^{-4})^{n_4+n_5}-(-A^{-4})^{n_6}+(-A^{-4})^{n_4+n_5+n_6} \Big) \Big], \cr 
V(\Theta_4^{(2)}(n_6,0,n_5))   =   (-A^4)^{n_5+n_6} \Big[ \langle\Theta_4^{(2)}\rangle + \frac{1}{\varphi} \Big( (1-(-A^{-4})^{n_6})\langle l_{135}^{(2)}\rangle+ (1-(-A^{-4})^{n_5})\langle l_{126}^{(2)}\rangle \Big) \cr 
\qquad   + \frac{1}{\varphi^2} \Big( 1-(-A^{-4})^{n_5}-(-A^{-4})^{n_6}+(-A^{-4})^{n_5+n_6} \Big) \Big],  \cr 
V(\Theta_5^{(2)}(0,0,n_6))   =   (-A^4)^{n_6} \Big[ \langle\Theta_5^{(2)}\rangle + \frac{1}{\varphi} \Big(1-(-A^{-4})^{n_6}\Big) \langle l_{234}^{(2)}\rangle\Big], \cr 
V(\Theta^{(2)}_6)   =   \langle\Theta^{(2)}_6\rangle.
\end{array}
 \label{eqn:JonesTheta} 
\end{equation}
By expressing $\langle \Theta_i^{(2)}(k_{i1},k_{i2},k_{i3})\rangle$ through $V(\Theta_i^{(2)}(k_{i1},k_{i2},k_{i3}))$ from~(\ref{eqn:brTheta}) and using~(\ref{eqn:LiTh}) and~(\ref{eqn:JonesTheta}), we get:
$$
\begin{array}{l}
\displaystyle V(L(n_1,n_2,n_3,n_4,n_5,n_6)) = (-A^4)^{n_1+n_2+\cdots+n_6} \Big\{ \langle L\rangle + \sum\limits_{i=1}^6\dfrac{f_{n_i}}{A^{n_i}}\langle \Theta_i^{(2)}\rangle \cr 
\displaystyle
\qquad + \frac{1}{\varphi} \Big[ \langle l_{126}^{(2)}\rangle \Big( \frac{f_{n_4}}{A^{n_4}}(1-(-A^{-4})^{n_5})  + \frac{f_{n_3}}{A^{n_3}}(1-(-A^{-4})^{n_4+n_5}) \Big)  \cr 
\displaystyle
\qquad + \langle l_{234}^{(2)}\rangle \Big( \frac{f_{n_5}}{A^{n_5}}(1-(-A^{-4})^{n_6})  + \frac{f_{n_1}}{A^{n_1}}(1-(-A^{-4})^{n_5+n_6}) \Big) \cr 
\displaystyle
\qquad + \langle l_{135}^{(2)}\rangle \Big( \frac{f_{n_4}}{A^{n_4}}(1-(-A^{-4})^{n_6}) + \frac{f_{n_2}}{A^{n_2}}(1-(-A^{-4})^{n_4+n_6}) \Big)  \cr 
\displaystyle
\qquad + \langle l_{456}^{(2)}\rangle \Big( \frac{f_{n_2}}{A^{n_2}}(1-(-A^{-4})^{n_3}) +\frac{f_{n_1}}{A^{n_1}}(1-(-A^{-4})^{n_2+n_3}) \Big) \cr 
\displaystyle
\qquad + \langle l_{1245}^{(2)}\rangle \frac{f_{n_3}}{A^{n_3}}(1-(-A^{-4})^{n_6}) + \langle l_{1346}^{(2)}\rangle\frac{f_{n_2}}{A^{n_2}}(1-(-A^{-4})^{n_5}) \cr 
\end{array}
$$
$$
\begin{array}{l}
\displaystyle
\qquad + \langle l_{2356}^{(2)}\rangle\frac{f_{n_1}}{A^{n_1}}(1-(-A^{-4})^{n_4}) \Big] \cr 

\displaystyle
\qquad + \frac{1}{\varphi^2} \Big[ \frac{f_{n_4}}{A^{n_4}}(1-(-A^{-4})^{n_5} - (-A^{-4})^{n_6}+(-A^{-4})^{n_5+n_6}) \cr 
\displaystyle
\qquad + \frac{f_{n_3}}{A^{n_3}}(1-(-A^{-4})^{n_4+n_5}-(-A^{-4})^{n_6}+(-A^{-4})^{n_4+n_5+n_6}) \cr 
\displaystyle
\qquad +\frac{f_{n_2}}{A^{n_2}}(2-(-A^{-4})^{n_4+n_6}-(-A^{-4})^{n_3}-(-A^{-4})^{n_5}+(-A^{-4})^{n_3+n_4+n_5+n_6}) \cr 
 \displaystyle
\qquad + \frac{f_{n_1}}{A^{n_1}}(2-(-A^{-4})^{n_5+n_6}-(-A^{-4})^{n_2+n_3} -(-A^{-4})^{n_4}+(-A^{-4})^{n_2+n_3+\cdots+n_6}) \Big] \Big\}.
\end{array} 
$$
From (\ref{eqn:7}) we get   
\begin{equation}
\frac{f_n}{A^{n}} = \frac{1-(-A^{-4})^{n}}{\varphi}, \label{eqn:8}
\end{equation}
where $\varphi = A^2 + A^{-2}$. 

Using the relation (\ref{eqn:8}) we obtain:  
\begin{equation}
\begin{array}{l} 
V(L(n_1,n_2,n_3,n_4,n_5,n_6)) = (-A^4)^{n_1+n_2+\cdots+n_6} \Big\{ \langle L\rangle + \sum\limits_{i=1}^6 \frac{1-(-A^{-4})^{n_i}}{\varphi} \langle\Theta_i^{(2)}\rangle \cr 
\qquad + \frac{1}{\varphi^2} \Big[ \langle l_{126}^{(2)}\rangle  (2-(-A^{-4})^{n_3}-(-A^{-4})^{n_4}-(-A^{-4})^{n_5}+(-A^{-4})^{n_3+n_4+n_5}) \cr 
\qquad +  \langle l_{234}^{(2)}\rangle (2-(-A^{-4})^{n_1} - (-A^{-4})^{n_5}-(-A^{-4})^{n_6}+(-A^{-4})^{n_1+n_5+n_6}) \cr 
\qquad + \langle l_{135}^{(2)}\rangle (2-(-A^{-4})^{n_2}-(-A^{-4})^{n_4}-(-A^{-4})^{n_6}+(-A^{-4})^{n_2+n_4+n_6})  \cr  
\qquad + \langle l_{456}^{(2)}\rangle (2-(-A^{-4})^{n_1}-(-A^{-4})^{n_2}-(-A^{-4})^{n_3}+(-A^{-4})^{n_1+n_2+n_3})  \cr 
\qquad + \langle l_{1245}^{(2)}\rangle (1-(-A^{-4})^{n_3}-(-A^{-4})^{n_6}+(-A^{-4})^{n_3+n_6}) \cr 
\qquad + \langle l_{1346}^{(2)}\rangle (1-(-A^{-4})^{n_2}-(-A^{-4})^{n_5}+(-A^{-4})^{n_2+n_5})  \cr 
\qquad + \langle l_{2356}^{(2)}\rangle (1-(-A^{-4})^{n_1}-(-A^{-4})^{n_4}+(-A^{-4})^{n_1+n_4}) \Big] \cr 
\qquad + \frac{1}{\varphi^3} \Big[6-2\sum\limits_{i=1}^6 (-A^{-4})^{n_i}+\sum\limits_{i=1}^3 (-A^{-4})^{n_i+n_{i+3}}+(-A^{-4})^{n_3+n_4+n_5}+(-A^{-4})^{n_2+n_4+n_6} \cr 
\qquad +(-A^{-4})^{n_1+n_5+n_6}   + (-A^{-4})^{n_1+n_2+n_3}-(-A^{-4})^{n_1+n_2+\cdots+n_6} \Big] \Big\}.
\end{array}
\label{eqn:Jonesmain}
\end{equation}
Thus, we obtained the formula for the Jones polynomial of the associated link $\mathcal L_{\Theta} = L (n_1, n_2, n_3, n_4, n_5, n_6)$ of the spatial $\mathbb K_4$-graph~$\Omega$ in terms of the bracket polynomials of the related links. 

\section{Normalized Jaeger polynomial of a spatial $\mathbb K_4$-graph} \label{section:Jaeger}

\begin{definition}
Let $D$ be a diagram of a spatial $\mathbb K_4$-graph $\Omega$ and $n_1, \ldots, n_6$ be twist parameters for $D$. Then 
$$
\widetilde{\mathfrak{J}}(D; A)=(-A^4)^{n_1+n_2+\cdots+n_6}\mathfrak{J}(D; A) 
$$
is said to be the \emph{normalized Jaeger polynomial} of $D$.   And 
$$
\widetilde{Y}(D;A)=(-A)^{n_1+n_2+\cdots+n_6}Y(D;A)
$$
is said to be the \emph{normalized Yamada polynomial} of $D$. 
\end{definition}

\begin{theorem} \label{theorem4.1}
Let $D$ be a diagram of a spatial $\mathbb K_4$-graph $\Omega$. Then the 
normalized Jaeger polynomial  $\widetilde{\mathfrak{J}}(D; A)$ is a pliable isotopic invariant of $\Omega$.
\end{theorem}

\begin{proof}
Let $D$ be a diagram of a $\mathbb K_4$-graph $\Omega$ and $n_1, \dots, n_6$ be twist parameters for~$D$. Denote $\nu = -\sum\limits_{i=1}^6 n_i$. By Lemma~\ref{lemma2.2}, 
$$
Y(D;A^4) = -(A^2+A^{-2})^{|E(\Omega)|-|V(\Omega)|+1} \mathfrak{J}(D;A) = - \varphi^3  \mathfrak{J}(D;A),
$$
since $|V(\Omega)| = 4$ and $|E(\Omega)| = 6$, where $\varphi = A^2 + A^{-2}$. 
Hence  
$$
\widetilde{Y}(D;A^4) = - \varphi^3 \widetilde{\mathfrak{J}}(D;A). 
$$
Thus, the polynomial $\widetilde{\mathfrak{J}}(D; A)$ is a pliable isotopic invariant if and only if the same is true for the polynomial $\widetilde{Y}(D; A)$. 
Note that 
\begin{eqnarray*}
w(l_{162}^{(2)}) & = & w_{11}+w_{22}+w_{66}+w_{16}-w_{26}-w_{12}, \cr 
w(l_{243}^{(2)}) & = & w_{22}+w_{44}+w_{33}+w_{24}-w_{34}-w_{23}, \cr 
w(l_{135}^{(2)}) & = & w_{11}+w_{33}+w_{55}+w_{35}-w_{13}-w_{15}, \cr 
w(l_{456}^{(2)}) & = & w_{44}+w_{55}+w_{66}+w_{46}+w_{45}+w_{56},
\end{eqnarray*}    
where $w(l)$ denotes the writhe of a link $l$. Then we obtain: 
$$
\begin{gathered} 
\sum\limits_{i=1}^6 n_i= -2\sum\limits_{i=1}^6 w_{ii}-(w_{16}+w_{24}+w_{35}+w_{46}+ 
   +w_{45}+w_{56} -w_{26}-w_{12} \cr -w_{34}-w_{23}- w_{13}-w_{15})  
   =-(w(l_{126}^{(2)})+w(l_{234}^{(2)})+w(l_{135}^{(2)})+w(l_{456}^{(2)})).
\end{gathered}    
$$   

Let us check the  invariance of $\widetilde{Y}(D; A)  = (-A)^{-\nu}Y(D; A)$ under the generalized Reidemeister moves (I)-(VI). By~\cite[Theorem~2]{Ya}, the polynomial $Y(D)$ doesn't change under the moves (II)-(IV). Since the writhe numbers of the above 2-component links $l^{(2)}_{ijk}$ do not change under the moves (II)-(IV) the number $\nu$ doesn't change too. 

Let us discuss the moves (I), (V) and (VI). Consider three diagrams $\xi_+$, $\xi$ and $\xi_-$ as pictured below. Then $\nu(\xi_+)=\nu(\xi)+3$. 
\begin{center}
\scalebox{0.9}{
    \begin{tikzpicture} [scale=0.45]
 \node at (0.2,0.8) {$\xi_+=$};
\draw[blue, dotted, very thick] (2.9,0.8) circle (1.9);
\draw[black, very thick] (1.3,1.6) -- (2,1.6);
\draw[black, very thick] (1.3,0.8) -- (2,0.8);
\draw[black, very thick] (1.3,0) -- (2,0);
\draw[black, very thick] (2,0) -- (3.6,1.6);
\draw[black, very thick] (2,0.8) -- (2.3,0.5);
\draw[black, very thick] (2.5, 0.3) -- (2.8,0); 
\draw[black, very thick] (2,1.6) -- (2.7,0.9);
\draw[black, very thick] (2.8,0) -- (3.6,0.8);
\draw[black, very thick] (2.9,0.7) -- (3.1,0.5);
\draw[black, very thick] (3.3,0.3) -- (3.6,0);
\draw[black, very thick] (3.6,1.6) -- (4,1.6);
\draw[black, very thick] (3.6,0.8) -- (4,0.8);
\draw[black, very thick] (3.6,0) -- (4,0);
\draw[black, very thick] (4,1.6) -- (4.5,0.8);
\draw[black, very thick] (4,0.8) -- (4.5,0.8);
\draw[black, very thick] (4,0) -- (4.5,0.8);
\draw[fill=black, black] (4.5,0.8) circle (0.1);
\end{tikzpicture}
\qquad
\begin{tikzpicture} [scale=0.45]
    \node at (5.3,0.8) {$\xi=$};
    \draw[blue, dotted, very thick] (7.8,0.8) circle (1.9);
\draw[black, very thick] (6.2,1.6) -- (8.4,1.6);
\draw[black, very thick] (6.2,0.8) -- (8.4,0.8);
\draw[black, very thick] (6.2,0) -- (8.4,0);
\draw[black, very thick] (8.4,1.6) -- (9,0.8);
\draw[black, very thick] (8.4,0.8) -- (9,0.8);
\draw[black, very thick] (8.4,0) -- (9,0.8);
\draw[fill=black, black] (9,0.8) circle (0.1);
\end{tikzpicture}
\qquad
\begin{tikzpicture} [scale=0.45]
  \node at (-5.8,0.8) {$\xi_-=$};
  \draw[blue, dotted, very thick] (-3.1,0.8) circle (1.9);
\draw[black, very thick] (-4.7,1.6) -- (-4,1.6);
\draw[black, very thick] (-4.7,0.8) -- (-4,0.8);
\draw[black, very thick] (-4.7,0) -- (-4,0);
\draw[black, very thick] (-4,0.8) -- (-3.7,1.1);
\draw[black, very thick] (-3.5,1.3) -- (-3.2,1.6);
\draw[black, very thick] (-4, 0) -- (-3.3,0.7);
\draw[black, very thick] (-3.1,0.9) -- (-2.9,1.1);
\draw[black, very thick] (-2.7,1.3) -- (-2.4,1.6);
\draw[black, very thick] (-4,1.6) -- (-2.4,0);
\draw[black, very thick] (-3.2,1.6) -- (-2.4,0.8) ;
\draw[black, very thick] (-2.4,0) -- (-2,0) -- (-1.5,0.8);
\draw[black, very thick] (-2.4,1.6) -- (-2,1.6) -- (-1.5,0.8);
\draw[black, very thick] (-2.4,0.8) -- (-1.5,0.8);
\draw[fill=black, black] (-1.5,0.8) circle (0.1);
\end{tikzpicture}
}
\end{center}
   
\noindent 
By~\cite[Prop.~5]{Ya} we have $Y(\xi_+; A)=(-A)^3 Y(\xi; A)$. Hence
$$
\widetilde{Y}(\xi_+; A)=(-A)^{-(\nu(\xi)+3)}(-A)^3 Y(\xi; A)=(-A)^{-\nu(\xi)} Y(\xi; A)=\widetilde{Y}(\xi; A).
$$
Analogously,  $\widetilde{Y}(\xi_-; A)=\widetilde{Y}(\xi; A)$. Hence $\widetilde{Y}(D; A)$ is invariant under the move (V). 

Consider three diagrams $c_+$, $c$ and $c_-$ as pictured below. Then $\nu(c_+) = \nu (c) + 2$ and $\nu(c_-) = \nu(c) - 2$. 
\begin{center}
\scalebox{0.9}{
\begin{tikzpicture} [scale=1.1]
\node at (-0.7,0.5) {$c_+$};
    \draw[black, very thick] (0,1) -- (0.35,0.55); 
\draw[black, very thick] (0.45,0.45) to [out=315, in=180] (0.85,0.2);
\draw[black, very thick] (0.85,0.2) to [out=0, in=270] (1.1,0.5);
\draw[black, very thick] (1.1,0.5) to [out=90, in=0] (0.85,0.8);
\draw[black, very thick] (0.85,0.8) to [out=180, in=45] (0.45,0.55);
\draw[black, very thick] (0,0) -- (0.45,0.55); 
\draw[blue, dotted, very thick] (0.5,0.5) circle (0.75);
\end{tikzpicture}
\qquad
\begin{tikzpicture} [scale=1.1]
\node at (-0.7,0.5) {$c=$};
\draw[black, very thick] (0,0) to [out=45, in=180] (0.85,0.2); 
\draw[black, very thick] (0.85,0.2) to [out=0, in=270] (1.1,0.5); 
\draw[black, very thick] (1.1,0.5) to [out=90, in=0] (0.85,0.8); 
\draw[black, very thick] (0.85,0.8) to [out=180, in=315] (0,1); 
\draw[blue, dotted, very thick] (0.5,0.5) circle (0.75);
\end{tikzpicture}
\qquad
\begin{tikzpicture} [scale=1.1]
\node at (-0.7,0.5) {$c_-=$};
\draw[black, very thick] (0,1) -- (0.4,0.5); 
\draw[black, very thick] (0.4,0.5) to [out=315, in=180] (0.85,0.2);
\draw[black, very thick] (0.85,0.2) to [out=0, in=270] (1.1,0.5);
\draw[black, very thick] (1.1,0.5) to [out=90, in=0] (0.85,0.8);
\draw[black, very thick] (0.85,0.8) to [out=180, in=45] (0.5,0.6);
%\draw[black, very thick, rounded corners] (0.,1) -- (0.35,0.55) -- (0.45,0.45) -- (0.6,0.3) --(0.95,0.2) -- (1.1,0.5) -- (0.95,0.8) -- (0.6,0.7) -- (0.45, 0.55);
\draw[black, very thick] (0.35,0.45) -- (0,0);
\draw[blue, dotted, very thick] (0.5,0.5) circle (0.75);
%\node at (1.6,0.5) {${\bf \longrightarrow}$};
\end{tikzpicture}
}
\end{center}
By~\cite[Prop.~4]{Ya},  $Y(c_+; A)=A^2Y(c; A)$. Hence 
$$
 \widetilde{Y}(c_+; A)=(-A)^{-(\nu(c)+2)}(-A)^2Y(c; A)=\widetilde{Y}(c; A).
 $$
Analogously, $\widetilde{Y} (c_-; A) = \widetilde{Y} (c;A)$. Hence $\widetilde{Y}(D;A)$ is invariant under the move (I). 

Consider three diagrams $\eta_+$, $\eta$ and $\eta_-$ as pictured below. Then $\nu(\eta_+) = \nu (\eta) + 1$ and $\nu(\eta_-) = \nu(\eta) - 1$. 
\begin{center}
\scalebox{0.9}{
    \begin{tikzpicture}[scale=1.1]
\node at (-0.7,0.5) {$\eta_+=$};
\filldraw [line width=2pt, black] (0.3,0.5) circle[radius=0.05cm];
\draw[black, very thick] (0.3,0.5) to [out=270, in=200] (0.65,0.25);
\draw[black, very thick] (0.65,0.25) to [out=45, in=200] (1.1,0.9);
\draw[black, very thick] (0.3,0.5) to [out=90, in=180] (0.6,0.75);
\draw[black, very thick] (0.6,0.75) to [out=0, in=135] (0.75,0.6);
\draw[black, very thick] (0.85,0.4) to [out=315, in=160] (1.1,0.1);
\draw[blue, dotted, very thick] (0.5,0.5) circle (0.75);
\draw[black, very thick] (0.3,0.5) -- (-0.2,0.5);
\end{tikzpicture}
\qquad
\begin{tikzpicture}[scale=1.1]
\node at (-0.7,0.5) {$\eta=$};
\filldraw [line width=2pt, black] (0.3,0.5) circle[radius=0.05cm];
\draw[black, very thick] (0.3,0.5) to [out=90, in=225] (1.1,0.9);
\draw[black, very thick] (0.3,0.5) to [out=270, in=135] (1.1,0.1);
%???????µ?·????
\draw[black, very thick] (0.3,0.5) -- (-0.2,0.5);
\draw[blue, dotted, very thick] (0.5,0.5) circle (0.75);
\end{tikzpicture}
\qquad
\begin{tikzpicture}[scale=1.1]
\node at (-0.7,0.5) {$\eta_-=$};
\filldraw [line width=2pt, black] (0.3,0.5) circle[radius=0.05cm];
%???µ?????°?? ???????°
\draw[black, very thick] (0.3,0.5) to [out=90, in=160] (0.65,0.75);
\draw[black, very thick] (0.65,0.75) to [out=315, in=160] (1.1,0.1);
%?????????°?? ???????°
\draw[black, very thick] (0.3,0.5) to [out=270, in=180] (0.6,0.25);
\draw[black, very thick] (0.6,0.25) to [out=0, in=45] (0.75,0.4);
\draw[black, very thick] (0.85,0.6) to [out=45, in=200] (1.1,0.9);
%???????µ?·????
\draw[black, very thick] (0.3,0.5) -- (-0.2,0.5);
\draw[blue, dotted, very thick] (0.5,0.5) circle (0.75);
\end{tikzpicture}
}
\end{center}
Again by~\cite[Prop.~4]{Ya} we obtain $Y(\eta_+; A)=-A Y(\eta; A)$. Hence
$$
\widetilde{Y}(\eta_+; A)=(-A)^{-(\nu(\eta)+1)}(-A) Y(\eta; A)=\widetilde{Y}(\eta;A).
$$
Analogously, $\widetilde{Y} (\eta_-;A) = \widetilde{Y} (\eta;A)$. Hence $\widetilde{Y}(D;A)$ is invariant under the move (VI). 
Thus, $\widetilde{Y}(D;A)$ is a pliable isotopic invariant of the spatial $\mathbb K_4$-graph~$\Omega$.
 \end{proof}

Now we prove an analogue of the Lemma~\ref{lemma3.2}. 

\begin{lemma}  \label{lemma4.1}
Let $D$ be a diagram of a spatial $\mathbb{K}_4$-graph $\Omega$. Let  $\{ \Theta_i^{(2)},  \mid i=1, \ldots, 6\}$ be the set of band diagrams of constituent spatial $\theta$-graphs. Let $\{ t_j^{(2)} \mid j=1, 2, 3, 4\}$ and $\{ q_k^{(2)} \mid   k=1,2,3\}$, be  sets of 2-parallel link diagrams of constituent knots corresponding to simple cycles of $\mathbb K_4$ of length $3$ and $4$, respectively. Let $L=D^{(2)}$. Then  
$$
 \widetilde{\mathfrak{J}}(D) = (-A^4)^{n_1+\cdots+n_6} \Big[ \langle L \rangle + \frac{1}{\varphi} \sum\limits_{i=1}^6 \langle \Theta_i^{(2)}\rangle + \frac{1}{\varphi^2} \Big( 2\sum\limits_{j=1}^4 \langle t_j^{(2)}\rangle + \sum\limits_{k=1}^3 \langle q_k^{(2)}\rangle \Big) + \frac{6}{\varphi^3} \Big],
 $$
 where $\varphi=A^2+A^{-2}$. 
\end{lemma}

\begin{proof}
Let us calculate polynomial $\mathfrak{J}(D)$ using the corresponding bar diagram  $B_D$ and  calculating $R(B_D)$. We will use the following properties proved in~\cite{Ja}:

\scalebox{1.1}{
\begin{tikzpicture}
\begin{knot}[
  consider self intersections=true,
%  draft mode=crossings,
  flip crossing=2,
  only when rendering/.style={
%    show curve controls
  }
  ]
  \node at (0,0.35) {$(R1)$ $R($};
  \draw[gray, line width = 3 pt] (0.98,0.35) -- (1.38,0.35);
\strand (0.93,0.7) .. controls (1.03,0) .. (1.18,0) .. controls (1.33,0) .. (1.43,0.7);
\node at (2.13,0.35) {$)=0;$};
\end{knot}
\end{tikzpicture}
}

\scalebox{1.1}{
\begin{tikzpicture}
\begin{knot}[
  consider self intersections=true,
%  draft mode=crossings,
  flip crossing=2,
  only when rendering/.style={
%    show curve controls
  }
  ]
  \node at (0,0.35) {$(R2)$ $R($};
  \draw[gray, line width = 3 pt] (0.95,0.2) -- (1.35,0.2);
  \draw[gray, line width = 3 pt] (0.95,0.45) -- (1.35,0.45);
\strand (0.95,0) -- (0.95,0.7);
\strand (1.35,0) -- (1.35,0.7);
\node at (2.15,0.35) {$)=R($};
  \draw[gray, line width = 3 pt] (2.95,0.35) -- (3.35,0.35);
\strand (2.95,0) -- (2.95,0.7);
\strand (3.35,0) -- (3.35,0.7);
\node at (3.6,0.35) {$);$};
\end{knot}
\end{tikzpicture}
}

\scalebox{1.1}{
\begin{tikzpicture}
\begin{knot}[
  consider self intersections=true,
%  draft mode=crossings,
  flip crossing=2,
  only when rendering/.style={
%    show curve controls
  }
  ]
  \node at (0,0.6) {$(R3)$ $R($};
  \draw[gray, line width = 2 pt] (1.35,1.1) -- (1.5,0.8);
  \strand (0.9, 0.2) -- (1.4,1.2);
  \strand (1.1,0) -- (1.6,1);
  \strand (0.9,1) -- (1.05,0.7);
  \strand (1.15,0.5) -- (1.2,0.4);
  \strand (1.3,0.2) -- (1.4,0);
  %?????®?°? ?? ?«???­????
  \strand (1.1,1.2) -- (1.2,1);
  \strand (1.3,0.8) -- (1.35,0.7);
  \strand (1.45,0.5) -- (1.6,0.2);
\node at (2.3,0.6) {$)=R($};
\draw[gray, line width = 2 pt] (3.15,0.4) -- (3.3,0.1);
  \strand (3.05, 0.2) -- (3.55,1.2);
  \strand (3.25,0) -- (3.75,1);
  \strand (3.05,1) -- (3.2,0.7);
  \strand (3.3,0.5) -- (3.35,0.4);
  \strand (3.45,0.2) -- (3.55,0);
  %?????®?°? ?? ?«???­????
  \strand (3.25,1.2) -- (3.35,1);
  \strand (3.45,0.8) -- (3.5,0.7);
  \strand (3.6,0.5) -- (3.75,0.2); 
\node at (3.85,0.6) {$)$.};
\end{knot}
\end{tikzpicture}
}

Since moves presented in properties $(R2)$ and $(R3)$ do not change $R(B_D)$, we apply them to $B_D$. Due to $(R3)$ we will move the bars along edges and place them near vertices. Then due to  $(R2)$ we leave at most one bar at each band, as illustrated in Fig.~\ref{fig7}.
%\begin{figure}[h]
 %   \includegraphics[width=0.45\textwidth]{fig-14.jpg}
 %   \caption{Bar diagram for $K_4$.} \label{fig7} 
%\end{figure}
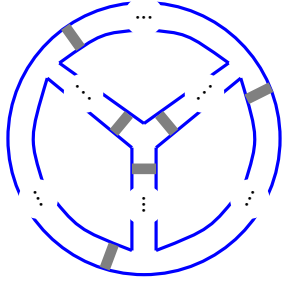
\begin{figure}[h] 
\begin{center}
%%%
\begin{tikzpicture}[scale=0.4] 
%\draw[step=5.mm, gray, very thin] (-10,0) grid (100.mm,80.mm);
% 
\draw [very thick, blue]  (0,5) circle [radius=4.5];
\draw[very thick, blue] (-2.8,7.5) .. controls (-1.4,8.5) .. (0,8.6).. controls (1.4, 8.5 ) .. (2.8,7.5);
\draw [very thick, blue] (0,5.5)-- (-2.8,7.5);
\draw [very thick, blue] (0,5.5)-- (2.8,7.5);
\draw[very thick, blue] (3.2,7) .. controls (3.8,5) .. (3.3,3.4).. controls (2.2, 2) .. (0.4,1.3);
\draw [very thick, blue] (0.4,4.7)-- (0.4,1.3);
\draw [very thick, blue] (0.4,4.7)-- (3.2,7);
\draw[very thick, blue] (-3.2,7) .. controls (-3.8,5) .. (-3.3,3.4).. controls (-2.2, 2) .. (-0.4,1.3);
\draw [very thick, blue] (-0.4,4.7)-- (-0.4,1.3);
\draw [very thick, blue] (-0.4,4.7)-- (-3.2,7);
\filldraw [very thick, white] (-2,6.5)  circle[radius=0.6cm];
 \node at (-2,6.5) {$\cdot$};  \node at (-2.2,6.7) {$\cdot$};   \node at (-1.8,6.3) {$\cdot$}; 
 \draw [ line width = 4pt, gray] (-1,5.2) -- (-0.5,5.8);
\filldraw [very thick, white] (2,6.5)  circle[radius=0.6cm];
 \node at (2,6.5) {$\cdot$}; \node at (2.2,6.7) {$\cdot$};   \node at (1.8,6.3) {$\cdot$}; 
  \draw [ line width = 4pt, gray] (1,5.2) -- (0.5,5.8);
\filldraw [very thick, white] (0,2.8)  circle[radius=0.6cm];
\node at (0,2.8) {$\cdot$}; \node at (0,2.6) {$\cdot$}; \node at (0,3) {$\cdot$};
  \draw [ line width = 4pt, gray] (-0.4,4) -- (0.4,4);
\filldraw [very thick, white] (3.5,3)  circle[radius=0.6cm];
\node at (3.5,3) {$\cdot$}; \node at (3.6,3.2) {$\cdot$}; \node at (3.4,2.8) {$\cdot$};
  \draw [ line width = 4pt, gray] (3.4,6.3) -- (4.2,6.7);
\filldraw [very thick, white] (-3.5,3)  circle[radius=0.6cm];
\node at (-3.5,3) {$\cdot$}; \node at (-3.6,3.2) {$\cdot$}; \node at (-3.4,2.8) {$\cdot$};
  \draw [ line width = 4pt, gray] (-1.3,0.7) -- (-1,1.5);
\filldraw [very thick, white] (0,9.05)  circle[radius=0.6cm];
 \node at (0,9) {$\cdot$};   \node at (-0.2,9) {$\cdot$};   \node at (0.2,9) {$\cdot$}; 
   \draw [ line width = 4pt, gray] (-2.6,8.7) -- (-2.1,8);
\end{tikzpicture}
\end{center}
\caption{Bar diagram $L[\varepsilon_1, \varepsilon_2, \varepsilon_3, \varepsilon_4, \varepsilon_5, \varepsilon_6]$.}  \label{fig7}
\end{figure}
Denote the obtained diagram by  $L[\varepsilon_1,\varepsilon_2,\varepsilon_3,\varepsilon_4,\varepsilon_5,\varepsilon_6]$, where $\varepsilon_i=1$ if there is a bar at $a_i$, and $\varepsilon = 0$ otherwise. Similarly, we denote by $\Theta_i^{(2)}[\varepsilon_1, \varepsilon_2, \varepsilon_3]$ the bar diagrams obtained from the diagram of $\Theta_i^{(2)}$ with bars added to edges $e_j$. 

Using the property (2) from Lemma~\ref{lemma2.1} we obtain
\begin{eqnarray*}
R(L[1,1,1,1,1,1]) & = & R(L[0,1,1,1,1,1])+ \frac{1}{\varphi} R(\Theta_1^{(2)}[2,2,1]) \cr 
& = & R(L[0,0,1,1,1,1]) + \frac{1}{\varphi} \left[ R(\Theta_2^{(2)}[2,1,1]) +  R(\Theta_1^{(2)}[1,1,1]) \right] \cr 
& = & R(L[0,0,0,1,1,1]) + \frac{1}{\varphi} \left[ R(\Theta_3^{(2)}[1,0,1]) +  R(\Theta_2^{(2)}[1,1,1]) \right.  \cr 
& & \left. +  R(\Theta_1^{(2)}[1,1,1])  \right] \cr
& \vdots & \cr  
& = & R(L[0,0,0,0,0,0])+ \frac{1}{\varphi} \left[ (R(\Theta_6^{(2)}[0,0,0])+ R(\Theta_5^{(2)}[0,0,1]) \right. \cr 
& &  + R(\Theta_4^{(2)}[1,0,1])+R(\Theta_3^{(2)}[1,0,1]) +R(\Theta_2^{(2)}[1,1,1])\cr 
& & \left. + R(\Theta_1^{(2)}[1,1,1]) \right] .
\end{eqnarray*}
Applying the property (2) from Lemma~\ref{lemma2.1} to $R(\Theta_i^{(2)}[\varepsilon_{i1},\varepsilon_{i2},\varepsilon_{i3}])$ we obtain 
\begin{eqnarray*}
R(\Theta_5^{(2)}[0,0,1]) & = & R(\Theta_4^{(2)})+ \frac{1}{\varphi} R(l^{(2)}_{234}), \cr 
R(\Theta_4^{(2)}[1,0,1]) & = & R(\Theta_4^{(2)}[0,0,1])+ \frac{1}{\varphi} R(l^{(2)}_{135}[1]) \cr 
&  = & R(\Theta_4^{(2)})+ \frac{1}{\varphi} (R(l^{(2)}_{126})+R(l^{(2)}_{135}))+ \frac{1}{\varphi^2} R({\bf 0}), \cr 
R(\Theta_3^{(2)}[1,0,1]) & = & R(\Theta_3^{(2)})+ \frac{1}{\varphi} (R(l^{(2)}_{1245})+R(l^{(2)}_{126}))+ \frac{1}{\varphi^2} R({\bf 0}), \cr 
R(\Theta_2^{(2)}[1,1,1]) & = & R(\Theta_2^{(2)})+ \frac{1}{\varphi} (R(l^{(2)}_{1346})+R(l^{(2)}_{456})+R(l^{(2)}_{135}))+ \frac{2}{\varphi^2}R({\bf 0}), \cr 
R(\Theta_1^{(2)}[1,1,1]) & = & R(\Theta_1^{(2)})+ \frac{1}{\varphi} (R(l^{(2)}_{2356})+R(l^{(2)}_{456})+R(l^{(2)}_{234}))+ \frac{2}{\varphi^2}R({\bf 0}), 
\end{eqnarray*}
where ${\bf 0}$ denotes an unknot. Recall that  
$$
R(l^{(2)}_i)=\langle l_i^{(2)}\rangle, \qquad R(L[0,0,0,0,0,0])=R(L)=\langle L\rangle, \qquad R({\bf 0})=\langle {\bf 0} \rangle=1.
$$  
Therefore we have:

$$
\begin{gathered} 
R(B_D)  =  \langle L\rangle + \frac{1}{\varphi} \Big[ \sum\limits^{6}_{i=1} \langle\Theta_i^{(2)}\rangle + \frac{1}{\varphi} (2\langle l^{(2)}_{234}\rangle+2\langle l^{(2)}_{126}\rangle+2\langle l^{(2)}_{135}\rangle+2\langle l^{(2)}_{456}\rangle \cr 
  +\langle l^{(2)}_{1245}\rangle+\langle l^{(2)}_{1346}\rangle+\langle l^{(2)}_{2356}\rangle)+ \frac{6}{\varphi^2} \Big].
\end{gathered}
$$
By normalizing this polynomial, we obtain the formula for $\widetilde{\mathfrak{J}}(D)$.  
\end{proof} 

\section{Invariants of a $\mathbb K_4$-graph and related links} \label{section:results}

The following theorem gives a relation between Jaeger and Jones polynomials of spatial $\mathbb K_4$-graph and constituent spatial $\theta$-graphs and constituent  knots. 

\begin{theorem}  \label{theorem:K4}
Let $\Omega$ be a spatial $\mathbb K_4$-graph in $S^3$. Let  $\{ \Theta_i, \, i=1,\ldots,6 \}$ be the set of constituent spatial $\theta$-graphs,  $\{ l_j, \,  j=1,\ldots, 7\}$ be the set of constituent knots, and $\mathcal{L}$ be the associated link of $\Omega$. Then 
\begin{equation}
\widetilde{\mathfrak{J}}(\Omega) = V(\mathcal{L}) + \dfrac{1}{\varphi}\sum\limits_{i=1}^6 \widetilde{\mathfrak{J}}(\Theta_i)-\dfrac{1}{\varphi ^2}\sum\limits_{j=1}^7 \widetilde{\mathfrak{J}}(l_j)+\dfrac{1}{\varphi^3},  \label{eqn:main}
\end{equation}
where $\varphi=A^2+A^{-2}$.
\end{theorem}

\begin{proof}
Consider the difference $\widetilde{\mathfrak{J}}(\Omega) - V(\mathcal{L})$, where these polynomials are calculated by  using Lemma~\ref{lemma4.1} and formula~(\ref{eqn:Jonesmain}):   
$$
\begin{array}{rl}
\widetilde{\mathfrak{J}}(\Omega)-V(\mathcal{L}) =  & (-A^4)^{n_1+n_2+\cdots+n_6} \Big\{ \sum\limits_{i=1}^6 \frac{(-A^{-4})^{n_i}}{\varphi} \langle\Theta_i^{(2)}\rangle + \frac{1}{\varphi^2} \Big[ \cr 
& \langle l_{126}^{(2)}\rangle (-(-A^{-4})^{n_3} - (-A^{-4})^{n_4} - (-A^{-4})^{n_5} + (-A^{-4})^{n_3+n_4+n_5}) \cr
& + \langle l_{234}^{(2)}\rangle (-(-A^{-4})^{n_1}-(-A^{-4})^{n_5}-(-A^{-4})^{n_6}+(-A^{-4})^{n_1+n_5+n_6})  \cr
& + \langle l_{135}^{(2)}\rangle (-(-A^{-4})^{n_2}-(-A^{-4})^{n_4}-(-A^{-4})^{n_6}+(-A^{-4})^{n_2+n_4+n_6}) \cr
& + \langle l_{456}^{(2)}\rangle (-(-A^{-4})^{n_1}-(-A^{-4})^{n_2}-(-A^{-4})^{n_3}+(-A^{-4})^{n_1+n_2+n_3}) \cr 
& + \langle l_{1245}^{(2)}\rangle (-(-A^{-4})^{n_3}-(-A^{-4})^{n_6}+(-A^{-4})^{n_3+n_6}) \cr 
& + \langle l_{1346}^{(2)}\rangle (-(-A^{-4})^{n_2}-(-A^{-4})^{n_5}+(-A^{-4})^{n_2+n_5}) \cr 
& + \langle l_{2356}^{(2)}\rangle (-(-A^{-4})^{n_1}-(-A^{-4})^{n_4}+(-A^{-4})^{n_1+n_4}) \Big] \cr 
& -\frac{1}{\varphi^3} \Big[- 2\sum\limits_{i=1}^6 (-A^{-4})^{n_i} + \sum\limits_{i=1}^3 (-A^{-4})^{n_i+n_{i+3}}  + (-A^{-4})^{n_3+n_4+n_5} \cr 
& +(-A^{-4})^{n_2+n_4+n_6}+(-A^{-4})^{n_1+n_5+n_6}+(-A^{-4})^{n_1+n_2+n_3} \cr 
& -(-A^{-4})^{n_1+n_2+\cdots+n_6}\Big]  \Big\}.
\end{array}
$$
This formula can be rewritten in the following form: 
$$
\begin{gathered} 
\widetilde{\mathfrak{J}}(\Omega)-V(\mathcal{L})=\frac{(-A^{-4})^{n_2+\cdots+n_6}}{\varphi} \Big[\langle \Theta_1^{(2)}\rangle +\frac{1}{\varphi}\Big(\langle l^{(2)}_{234}\rangle+\langle l^{(2)}_{456}\rangle+\langle l^{(2)}_{2356}\rangle \Big) +\frac{2}{\varphi^2} \Big] \cr
+ \frac{(-A^{-4})^{n_1+n_3+\cdots+n_6}}{\varphi} \Big[\langle \Theta_2^{(2)}\rangle+\dfrac{1}{\varphi} \big(\langle l_{135}^{(2)}\rangle + \langle l_{456}^{(2)}\rangle + \langle l_{1346}^{(2)}\rangle +\dfrac{2}{\varphi^2}\big) \Big] %\cr 
\end{gathered}
$$
$$
\begin{gathered}
+ \cdots  
+ \frac{(-A^{-4})^{n_1+\cdots+n_5}}{\varphi} \Big[ \langle \Theta_6^{(2)}\rangle +\dfrac{1}{\varphi}\big(\langle l^{(2)}_{234}\rangle+\langle l^{(2)}_{135}\rangle+\langle l^{(2)}_{1245}\rangle\big)+\dfrac{2}{\varphi^2} \Big] \cr 
- \frac{(-A^{-4})^{n_1+n_2+n_6}}{\varphi^2} \Big[ \langle l^{(2)}_{126}\rangle + \frac{1}{\varphi} 
 \Big]-\cdots - \frac{(-A^{-4})^{n_1+n_2+n_4+n_5}}{\varphi^2} \Big[ \langle l^{(2)}_{1245}\rangle+\dfrac{1}{\varphi} 
 \Big] + \frac{1}{\varphi^3}.
 \end{gathered}
$$
Let us use Lemma~\ref{lemma3.2} and the following equality proved in~\cite[Theorem 8]{H}:
\begin{equation}
    \widetilde{\mathfrak{J}}(\mathcal{K})=A^{-8 w(K)}\mathfrak{J}(K)=A^{-8 w(K)} \Big(\langle K^{(2)}\rangle +\dfrac{1}{\varphi} \Big),
    \label{eq18}
\end{equation}
where  $K$ is a diagram of a knot  $\mathcal{K}$ and $w(K)$ is the writhe number of $K$. Recall that  
$$
2w(\Theta_i)=-(n_1+\cdots+n_{i-1}+n_{i+1}+\cdots+n_6)
$$ 
and the following relations for writhe number of subknots: 
\begin{eqnarray*}    
2w(l_{162}) & = & 2(w_{11}+w_{22}+w_{66}+w_{16}-w_{26}-w_{12})=-(n_1+n_6+n_2), \cr 
2w(l_{243}) & = & 2(w_{22}+w_{44}+w_{33}+w_{24}-w_{34}-w_{23})=-(n_2+n_4+n_3), \cr 
2w(l_{135}) & = & 2(w_{11}+w_{33}+w_{55}+w_{35}-w_{13}-w_{15})=-(n_1+n_3+n_5), \cr 
2w(l_{456}) & = & 2(w_{44}+w_{55}+w_{66}+w_{46}+w_{45}+w_{56})=(n_4+n_5+n_6), \cr 
2w(l_{4316}) & = & 2(w_{44}+w_{33}+w_{11}+w_{66} +w_{46}+w_{16}-w_{36}+w_{14}-w_{34}-w_{13}) \cr 
& = & -(n_4+n_3+n_1+n_6), \cr 
2w(l_{6235}) & = & 2(w_{66}+w_{22}+w_{33}+w_{55}-w_{26}+w_{36}+w_{56}-w_{23}-w_{25}+w_{35}) \cr 
& = & -(n_6+n_2+n_3+n_5), \cr 
2w(l_{1245}) & = & 2(w_{44}+w_{55}+w_{11}+w_{22}+w_{24}+w_{25}-w_{12}+w_{45}-w_{14}-w_{15}) \cr 
& = & -(n_1+n_2+n_4+n_5),
\end{eqnarray*}   
we get 
$$
 \widetilde{\mathfrak{J}}(\Omega)-V(\mathcal{L})=\dfrac{1}{\varphi}\sum\limits_{i=1}^6 \widetilde{\mathfrak{J}}(\Theta_i)-\dfrac{1}{\varphi ^2}\sum\limits_{j=1}^7 \widetilde{\mathfrak{J}}(l_j)+\dfrac{1}{\varphi^3},
$$
that completes the proof of the theorem. 
\end{proof}

Recall the relation between the normalized Yamada polynomial and the normalized Jaeger polynomial: 
$$
\widetilde{Y}(D;A^4) = -(A^2+A^{-2})^{|E(G)|-|V(G)|+1} \widetilde{\mathfrak{J}}(D;A). 
$$ 

\begin{corollary} \label{cor6.1} 
There is the following relation between the normalized Yamada polynomials and the Jones polynomial:
    \begin{equation}
\widetilde{Y}(\Omega;A^4)+\varphi^3 V(\mathcal{L}; A)=\sum\limits_{i=1}^6 \widetilde{Y}(\Theta_i;A^4)-\sum\limits_{j=1}^7 \widetilde{Y}(l_j;A^4)-1.  \label{eqn:main_yamada}
\end{equation}
\end{corollary}

Let $l_{ijk}$ be a substituent knot of $\Omega$ formed by a cycle of three edges, and  $l_{ijk}^{(2)}$ be a 2-parallel  diagram for $l_{ijk}$. Denote by $\mathcal{L}_{ijk}$ the link obtained from  $l_{ijk}^{(2)}$ by adding $n_i+n_j+n_k$ half twists on its two components, where $n_1,\dots,n_6$ are twist parameters from~(\ref{eqnni}). Similarly, let $l_{ijkl}$ be a substituent knot of $\Omega$ formed by a cyclic of four edges, and $l_{ijkl}^{(2)}$ be a 2-parallel diagram for $l_{ijkl}$. Denote by $\mathcal{L}_{ijkl}$ the link obtained from $l_{ijkl}^{(2)}$ by adding $n_i+n_j+n_k+n_l$ half twists on its two components. For $i=1,\ldots, 6$ denote by  $\mathcal{L}_i$ the associated link of constituent spatial $\theta$-graph $\Theta_i$ of $\Omega$.
 
\begin{corollary} \label{cor6.2}
The following relation holds
\begin{eqnarray*}
\widetilde{\mathfrak{J}}(\Omega) & = & V(\mathcal{L})+\dfrac{1}{\varphi}\sum\limits_{i=1}^6 V(\mathcal{L}_i)   + \dfrac{1}{\varphi^2}\bigg( 2\big( V(\mathcal{L}_{234})+V(\mathcal{L}_{456})+V(\mathcal{L}_{135})+V(\mathcal{L}_{126}) \big ) \cr 
& & \qquad \qquad \qquad \qquad \qquad +V(\mathcal{L}_{2356})+V(\mathcal{L}_{1346})+V(\mathcal{L}_{1245})\bigg)+\dfrac{6}{\varphi^3}.
\end{eqnarray*}    
\end{corollary}

\begin{proof}
By Theorem~\ref{theorem-theta} we have 
\begin{eqnarray*} 
\widetilde{\mathfrak{J}}(\Theta_1) & = & V(\mathcal{L}_1)+\dfrac{1}{\varphi}(\widetilde{\mathfrak{J}}(l_{2356})+\widetilde{\mathfrak{J}}(l_{234})+\widetilde{\mathfrak{J}}(l_{456}))-\dfrac{1}{\varphi^2}, \cr 
 \widetilde{\mathfrak{J}}(\Theta_2) & = & V(\mathcal{L}_2)+\dfrac{1}{\varphi}(\widetilde{\mathfrak{J}}(l_{1346})+\widetilde{\mathfrak{J}}(l_{135})+\widetilde{\mathfrak{J}}(l_{456}))-\dfrac{1}{\varphi^2}, \cr 
\widetilde{\mathfrak{J}}(\Theta_3) & = & V(\mathcal{L}_3)+\dfrac{1}{\varphi}(\widetilde{\mathfrak{J}}(l_{1245})+\widetilde{\mathfrak{J}}(l_{126})+\widetilde{\mathfrak{J}}(l_{456}))-\dfrac{1}{\varphi^2}, \cr 
\widetilde{\mathfrak{J}}(\Theta_4) & = & V(\mathcal{L}_4)+\dfrac{1}{\varphi}(\widetilde{\mathfrak{J}}(l_{2356})+\widetilde{\mathfrak{J}}(l_{126})+\widetilde{\mathfrak{J}}(l_{135}))-\dfrac{1}{\varphi^2}, \cr 
\widetilde{\mathfrak{J}}(\Theta_5) & = & V(\mathcal{L}_5)+\dfrac{1}{\varphi}(\widetilde{\mathfrak{J}}(l_{1346})+\widetilde{\mathfrak{J}}(l_{126})+\widetilde{\mathfrak{J}}(l_{234}))-\dfrac{1}{\varphi^2}, \cr 
\widetilde{\mathfrak{J}}(\Theta_6) & = & V(\mathcal{L}_6)+\dfrac{1}{\varphi}(\widetilde{\mathfrak{J}}(l_{1245})+\widetilde{\mathfrak{J}}(l_{234})+\widetilde{\mathfrak{J}}(l_{135}))-\dfrac{1}{\varphi^2}.
\end{eqnarray*}
By substituting these expressions in formula~(\ref{eqn:main}) we conclude 
$$
\begin{array}{rl}
\widetilde{\mathfrak{J}}(\Omega) = & V(\mathcal{L})+\dfrac{1}{\varphi} \sum\limits_{i=1}^6\widetilde{\mathfrak{J}}(\Theta_i) - \dfrac{1}{\varphi^2} \sum\limits_{j=1}^7 \widetilde{\mathfrak{J}}(l_j) + \dfrac{1}{\varphi^3} = V(\mathcal{L}) + \dfrac{1}{\varphi} \Big[ \sum\limits_{i=1}^6 V(\mathcal{L}_i)  \cr 
& + \dfrac{1}{\varphi} \Big( 3 \big(\widetilde{\mathfrak{J}}(l_{234})  
  + \widetilde{\mathfrak{J}}(l_{456}) + \widetilde{\mathfrak{J}}(l_{135}) + \widetilde{\mathfrak{J}}(l_{126}) \big) + 2\big(\widetilde{\mathfrak{J}}(l_{2356}) + \widetilde{\mathfrak{J}}(l_{1346}) \cr 
  & + \widetilde{\mathfrak{J}}(l_{1245})\big) \Big) - \dfrac{6}{\varphi^2} \Big] - \dfrac{1}{\varphi^2}\bigg[\widetilde{\mathfrak{J}}(l_{234})+\widetilde{\mathfrak{J}}(l_{456})+\widetilde{\mathfrak{J}}(l_{135})+\widetilde{\mathfrak{J}}(l_{126}) \cr 
  & +\widetilde{\mathfrak{J}}(l_{2356})+\widetilde{\mathfrak{J}}(l_{1346})+\widetilde{\mathfrak{J}}(l_{1245}) \Big]  
 + \dfrac{1}{\varphi^3} \cr 
= &
V(\mathcal{L})+\dfrac{1}{\varphi} \sum\limits_{i=1}^6 V(\mathcal{L}_i) + \dfrac{1}{\varphi^2} \Big( 2\big(\widetilde{\mathfrak{J}}(l_{234}) + \widetilde{\mathfrak{J}}(l_{456}) + \widetilde{\mathfrak{J}}(l_{135}) + \widetilde{\mathfrak{J}}(l_{126})\big) \cr 
& + \widetilde{\mathfrak{J}}(l_{2356}) +\widetilde{\mathfrak{J}}(l_{1346}) + \widetilde{\mathfrak{J}}(l_{1245}) \Big) - \dfrac{5}{\varphi^3}.
\end{array}
$$

Let  $\mathcal{K}$ be a knot, and $K$ be its diagram. Denote by  $\mathcal{L}_K$ the associated link of $K$ that is defined as the band diagram $K^{(2)}$ with added $n$ half twists such that the Seifert form of corresponding band surface is zero. Then
$$
\langle \mathcal{L}_K \rangle = A^{n} \langle K^{(2)} \rangle + f_n \langle 0 \rangle = A^{n} \Big( \langle K^{(2)}\rangle + \frac{f_n}{A^{n}} \Big) = A^{n}\Big( \langle K^{(2)}\rangle + \frac{1-(-A^{-4})^{n}}{\varphi} \Big).
$$
Therefore, we have
$$
V(\mathcal{L}_K)=(-A^3)^{n}\langle\mathcal{L}_K\rangle=(-A^{4})^{n} \Big( \langle K^{(2)}\rangle+\dfrac{1-(-A^{-4})^{n}}{\varphi} \Big).
$$
By using the formula~(\ref{eq18}) for $\widetilde{\mathfrak{J}}(\mathcal{K})$ we obtain 
$$
\widetilde{\mathfrak{J}}(\mathcal{K}) - V(\mathcal{L}_K)=\dfrac{1}{\varphi}.
$$
By expressing $\widetilde{\mathfrak{J}}(\mathcal{K})$ from this relation we get:
\begin{eqnarray*}
\widetilde{\mathfrak{J}}(\Omega) & = & V(\mathcal{L})+\dfrac{1}{\varphi}\sum\limits_{i=1}^6 V(\mathcal{L}_i)  +\dfrac{1}{\varphi^2}\bigg( 2\big( V(\mathcal{L}_{234})+V(\mathcal{L}_{456})+V(\mathcal{L}_{135})+V(\mathcal{L}_{126}) \big ) \cr 
&& +V(\mathcal{L}_{2356})+V(\mathcal{L}_{1346})+V(\mathcal{L}_{1245})\bigg)+\dfrac{6}{\varphi^3}.
\end{eqnarray*}  
Thus, the corollary is proved. 
\end{proof}

\section{Examples} \label{section:examples}
In this section we give two examples of calculations by above presented formulae for the spatial $\mathbb K_4$-graphs $\Omega_1$ and $\Omega_7$ with diagrams presented in Fig.~\ref{fig:example}. We will follow the above introduces notations. 
\begin{figure}[h] 
\begin{center}
\scalebox{0.8}{
\begin{tikzpicture}[scale=0.05] 
\draw [very thick, black] (50,50) circle[radius=25.0cm];
\draw [very thick, black, fill] (50,50) circle[radius=1.0cm];
\draw [very thick, black, fill] (50,25) circle[radius=1.0cm];
\draw [very thick, black, fill] (30,65) circle[radius=1.0cm];
\draw [very thick, black, fill] (70,65) circle[radius=1.0cm];
\draw [very thick, black] (50,50) -- (50,25);
\draw [very thick, black] (50,50) -- (30,65);
\draw [very thick, black] (50,50) -- (70,65);
\draw [very thick, black,->] (50,50) -- (50,35);
\draw [very thick, black, ->] (50,50) -- (40,57.5);
\draw [very thick, black, ->] (50,50) -- (60,57.5);
\draw [very thick, black, ->] (30,35)--(29,36);
\draw [very thick, black, ->] (71,36)--(70,35);
\draw [very thick, black, ->] (49,75)--(50,75);
%\node at (59,60) {$a_2$}; 
%\node at (50,77) {$a_6$};
%\node at (41,60) {$a_1$};
%\node at (55,35) {$a_3$};
%\node at (25,36) {$a_5$};
%\node at (75,36) {$a_4$};
\node at (50,15) {$\Omega_1$};
\end{tikzpicture}
\qquad \qquad 
\begin{tikzpicture}[scale=0.05] 
\draw[black, very thick, rounded corners] (52,80) -- (57,75) -- (57,70) --(52,62); 
\draw[black, very thick, rounded corners] (48,58) -- (44,55) -- (40,53) -- (37,50) -- (36,48) -- (35,44); 
\draw[black, very thick, rounded corners] (48,80) -- (43,75) -- (43,70) -- (48,62) -- (50,60) -- (52,58) -- (56,55) -- (60,53) -- (63,50) -- (64,48); 
\draw[black, very thick, rounded corners] (52,80) -- (48,83); 
\draw[black, very thick, rounded corners] (52,83) -- (55,85) --(60,85) --(67,83) --(70,79) -- (72,76) -- (74,70) -- (72,60) -- (70,55) -- (68,50) -- (66,47) -- (62,45) -- (57,43) --(53,42) -- (50,42) -- (47,42) -- (43,43) -- (38,45); 
\draw[black, very thick, rounded corners] (48,83) -- (45,85) -- (40,85) -- (33,83) -- (30,79) -- (28,76) -- (26,70) -- (28,60) -- (30,55) -- (32, 50) -- (34,48); 
\draw[black, very thick, rounded corners] (65,44) -- (65,40) -- (64,35) -- (60,29) -- (53,27) -- (50,27) -- (47,27) --(40,29) -- (36,35) -- (35,40) --(35,44); 
\draw[black, very thick, rounded corners] (43,72) -- (57,72); 
\draw[black, very thick, rounded corners] (50,27) -- (50,42); 
\draw [very thick, black, fill] (43,72) circle[radius=1.0cm];
\draw [very thick, black, fill] (57,72) circle[radius=1.0cm];
\draw [very thick, black, fill] (50,27) circle[radius=1.0cm];
\draw [very thick, black, fill] (50,42) circle[radius=1.0cm];
\node at (21,70) {$a_2$}; 
\node at (78,70) {$a_6$};
\node at (50,76) {$a_1$};
\node at (40,56) {$a_3$};
\node at (60,56) {$a_5$};
\node at (45,36) {$a_4$};
\draw[very thick, ->] (50,36) -- (50,30);
\draw[very thick, ->] (63.8,35) -- (64.5,37);
\draw[very thick, ->] (35.6,37) --  (36.3,35);
\draw[very thick, ->] (57,72) -- (45,72);
\draw[very thick, ->] (73.7,68.5) -- (73.2,67);
\draw[very thick, ->] (26.3,68.5) -- (26.8,67);
\node at (50,15) {$\Omega_7$};
\end{tikzpicture}
}
\end{center}
\caption{Spatial $\mathbb K_4$-graphs $\Omega_1$ and $\Omega_7$.} \label{fig:example}
\end{figure}
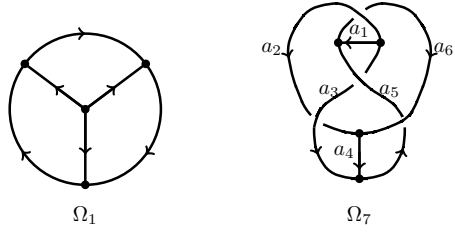 

\begin{example} {\rm 
Let us consider the spatial graph $\Omega_1$ presented by a diagram in Fig.~\ref{fig:example}. 
The Yamada polynomial of $\Theta_1$ was calculated in~\cite{VD}:
$$
 Y(\Omega_1)=A^3+2A+2A^{-1}+A^{-3}, 
$$
hence by Lemma~\ref{lemma2.2} we get 
$$
 \widetilde{\mathfrak{J}}(\Omega_1)=-\dfrac{A^{12}+2A^4+2A^{-4}+A^{-12}}{\varphi^3}.
$$
The spatial graph $\Omega_1$ contains six $\theta$-graphs $\Theta_i$, $i=1, \ldots, 6$, each of which has a diagram without crossings, and seven knots $l_j$, $j=1, \ldots, 7$, each of which is a trivial knot. Then by~\cite{Ya} we have 
$$
Y(\Theta_i;A)=-A^2-A-2-A^{-1}-A^{-2},\qquad i=1,\dots,6,
$$
and by Lemma~\ref{lemma2.2} we get 
$$
    \widetilde{\mathfrak{J}}(\Theta_i)=\mathfrak{J}(\Theta_i) = -\dfrac{1}{\varphi^{|E(\Theta_i)|-|V(\Theta_i)|+1}}Y(\Theta_i;A^4) 
    = \dfrac{A^8+A^4+2+A^{-4}+A^{-8}}{\varphi^2}
$$    
for $=1,\dots,6$.  For each of trivial knots we have
$$
\widetilde{\mathfrak{J}}(l_j)=-A^2-A^{-2}+\dfrac{1}{\varphi},\qquad j=1,\dots,7.
$$
Since the associated $L$ is the four-component link with unlinked trivial components we get 
$$
 V(\mathcal{L})=(-A-A^{-1})^3=-(A^6+3A^2+3A^{-2}+A^{-6}).
$$
It is easy to see directly that 
$$
\widetilde{\mathfrak{J}}(\Omega_1)-V(\mathcal{L}) = \dfrac{6A^8+13A^4+20+13A^{-4}+6A^{-8}}{\varphi^3} 
     = \dfrac{1}{\varphi}\sum\limits_{i=1}^6 \widetilde{\mathfrak{J}}(\Theta_i)-\dfrac{1}{\varphi ^2}\sum\limits_{j=1}^7 \widetilde{\mathfrak{J}}(l_j)+\dfrac{1}{\varphi^3}. 
$$
}
\end{example}

\begin{example}
{\rm 
Consider the spatial $\mathbb K_4$-graph $\Omega_7$ presented by a diagram in Fig.~\ref{fig:example}. 
The Yamada polynomial of $\Omega_7$ is presented in Table~1, see also~\cite{VD}:  
$$
 Y(\Omega_7;A)=-A^8-A^5+A^4+A^3+3A+3A^{-1}+A^{-3}+A^{-4}-A^{-5}-A^{-8}, 
$$ 
hence by Lemma~\ref{lemma2.2} we get 
$$
 \widetilde{\mathfrak{J}}(\Omega_7)=\mathfrak{J}(\Omega_7)=-\dfrac{-A^{32}-A^{20}+A^{16}+A^{12}+3A^4+3A^{-4}+A^{-12}-A^{-20}-A^{-32}}{\varphi^3}.
$$

Let us denote by $\Theta_i$ a spatial $\theta$-graph obtained from $\Omega_7$ by removing edge $a_i$ for $i=1, \ldots, 6$ in notations of Fig.~\ref{fig:example}. Then each of  $\Theta_2$, $\Theta_3$, $\Theta_5$ and $\Theta_6$ has a diagram without crossings.  Thus,  by~\cite{Ya} we have  
$$
Y(\Theta_i;A)=-A^2-A-2-A^{-1}-A^{-2}
$$
and by Lemma~\ref{lemma2.2} we obtain 
$$
\widetilde{\mathfrak{J}}(\Theta_i,A)=\mathfrak{J}(\Theta_i,A)=-\dfrac{1}{\varphi^2}Y(\Theta_i;A^4)=\dfrac{A^8+A^4+2+A^{-4}+A^{-8}}{\varphi^2} 
$$
for $i \in \{ 2, 3, 5, 6\}$. Let us consider spatial $\theta$-graphs $\Theta_1$ and $\Theta_4$ presented by their diagrams in Fig.~\ref{fig:example2}.  
\begin{figure}[ht]
\begin{center}
\scalebox{0.8}{
\begin{tikzpicture}[scale=0.05] 
\draw[black, very thick, rounded corners] (52,80) -- (57,75) -- (57,70) --(52,62); 
\draw[black, very thick, rounded corners] (48,58) -- (44,55) -- (40,53) -- (37,50) -- (36,48) -- (35,44); 
\draw[black, very thick, rounded corners] (48,80) -- (43,75) -- (43,70) -- (48,62) -- (50,60) -- (52,58) -- (56,55) -- (60,53) -- (63,50) -- (64,48); 
\draw[black, very thick, rounded corners] (52,80) -- (48,83); 
\draw[black, very thick, rounded corners] (52,83) -- (55,85) --(60,85) --(67,83) --(70,79) -- (72,76) -- (74,70) -- (72,60) -- (70,55) -- (68,50) -- (66,47) -- (62,45) -- (57,43) --(53,42) -- (50,42) -- (47,42) -- (43,43) -- (38,45); 
\draw[black, very thick, rounded corners] (48,83) -- (45,85) -- (40,85) -- (33,83) -- (30,79) -- (28,76) -- (26,70) -- (28,60) -- (30,55) -- (32, 50) -- (34,48); 
\draw[black, very thick, rounded corners] (65,44) -- (65,40) -- (64,35) -- (60,29) -- (53,27) -- (50,27) -- (47,27) --(40,29) -- (36,35) -- (35,40) --(35,44); 
%\draw[black, very thick, rounded corners] (43,72) -- (57,72); 
\draw[black, very thick, rounded corners] (50,27) -- (50,42); 
%\draw [very thick, black, fill] (43,72) circle[radius=1.0cm];
%\draw [very thick, black, fill] (57,72) circle[radius=1.0cm];
\draw [very thick, black, fill] (50,27) circle[radius=1.0cm];
\draw [very thick, black, fill] (50,42) circle[radius=1.0cm];
%\node at (23,70) {$a_2$}; 
%\node at (77,70) {$a_6$};
%\node at (50,74) {$a_1$};
%\node at (40,56) {$a_3$};
%\node at (60,56) {$a_5$};
%\node at (47,36) {$a_4$};
%\draw[very thick, ->] (50,36) -- (50,30);
%\draw[very thick, ->] (63.8,35) -- (64.5,37);
%\draw[very thick, ->] (35.6,37) --  (36.3,35);
%\draw[very thick, ->] (57,72) -- (45,72);
%\draw[very thick, ->] (73.7,68.5) -- (73.2,67);
%\draw[very thick, ->] (26.3,68.5) -- (26.8,67);
\node at (50,20) {$\bf{\Theta_1}$};
\end{tikzpicture}
\qquad
\begin{tikzpicture}[scale=0.05] 
\draw[black, very thick, rounded corners] (52,80) -- (57,75) -- (57,70) --(52,62); 
\draw[black, very thick, rounded corners] (48,58) -- (44,55) -- (40,53) -- (37,50) -- (36,48) -- (35,44); 
\draw[black, very thick, rounded corners] (48,80) -- (43,75) -- (43,70) -- (48,62) -- (50,60) -- (52,58) -- (56,55) -- (60,53) -- (63,50) -- (64,48); 
\draw[black, very thick, rounded corners] (52,80) -- (48,83); 
\draw[black, very thick, rounded corners] (52,83) -- (55,85) --(60,85) --(67,83) --(70,79) -- (72,76) -- (74,70) -- (72,60) -- (70,55) -- (68,50) -- (66,47) -- (62,45) -- (57,43) --(53,42) -- (50,42) -- (47,42) -- (43,43) -- (38,45); 
\draw[black, very thick, rounded corners] (48,83) -- (45,85) -- (40,85) -- (33,83) -- (30,79) -- (28,76) -- (26,70) -- (28,60) -- (30,55) -- (32, 50) -- (34,48); 
\draw[black, very thick, rounded corners] (65,44) -- (65,40) -- (64,35) -- (60,29) -- (53,27) -- (50,27) -- (47,27) --(40,29) -- (36,35) -- (35,40) --(35,44); 
\draw[black, very thick, rounded corners] (43,72) -- (57,72); 
%\draw[black, very thick, rounded corners] (50,27) -- (50,42); 
\draw [very thick, black, fill] (43,72) circle[radius=1.0cm];
\draw [very thick, black, fill] (57,72) circle[radius=1.0cm];
%\draw [very thick, black, fill] (50,27) circle[radius=1.0cm];
%\draw [very thick, black, fill] (50,42) circle[radius=1.0cm];
%\node at (23,70) {$a_2$}; 
%\node at (77,70) {$a_6$};
%\node at (50,74) {$a_1$};
%\node at (40,56) {$a_3$};
%\node at (60,56) {$a_5$};
%\node at (47,36) {$a_4$};
%\draw[very thick, ->] (50,36) -- (50,30);
%\draw[very thick, ->] (63.8,35) -- (64.5,37);
%\draw[very thick, ->] (35.6,37) --  (36.3,35);
%\draw[very thick, ->] (57,72) -- (45,72);
%\draw[very thick, ->] (73.7,68.5) -- (73.2,67);
%\draw[very thick, ->] (26.3,68.5) -- (26.8,67);
\node at (50,20) {$\bf{\Theta_4}$};
\end{tikzpicture}
\qquad 
    \begin{tikzpicture}[scale=0.034]
\filldraw [line width=2pt, black] (50,50) circle[radius=1cm];
\filldraw [line width=2pt, black] (50,34) circle[radius=1cm];
\draw[black, very thick] (40,42) to [out=60, in=180] (50,50);
\draw[black, very thick] (50,50) to [out=0, in=120] (60,42);
\draw[black, very thick] (60,42) to [out=300, in=40] (52,17);
\draw[black, very thick] (48,14) to [out=210, in=0] (30,10);
\draw[black, very thick] (30,10) to [out=180, in=270] (10,42);
\draw[black, very thick] (10,42) to [out=90, in=180] (30,70);
\draw[black, very thick] (30,70) to [out=0, in=105] (61,46);
\draw[black, very thick] (59,39) to [out=230, in=0] (50,34);
\draw[black, very thick] (50,34) to [out=180, in=300] (39,40);
\draw[black, very thick] (39,40) to [out=120, in=210] (48,62);
\draw[black, very thick] (52,64) to [out=30, in=180] (70,70);
\draw[black, very thick] (70,70) to [out=0, in=90] (90,42);
\draw[black, very thick] (90,42) to [out=270, in=0] (70,10);
\draw[black, very thick] (70,10) to [out=180, in=260] (38,38);
\draw[black, very thick] (50,50) -- (50,34);
\node at (50,0) {$\bf{\widetilde{\Theta}}$};
\end{tikzpicture}
}
\end{center}
\caption{Spatial $\Theta$-graphs $\Theta_1$, $\Theta_4$ and $\widetilde{\Theta}$.} \label{fig:example2}
\end{figure}
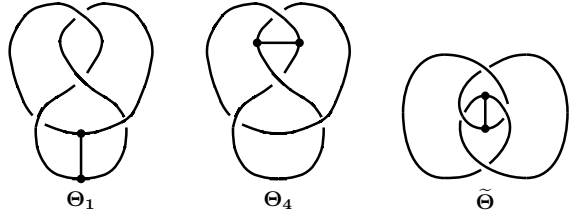

\noindent 
Recall that the Yamada polynomial of the spatial $\theta$-graph $\widetilde{\Theta}$, see Fig.~\ref{fig:example2}, was calculated in~\cite{VD}:  
$$
 Y(\widetilde{\Theta};A)=A^7-A^5-A^3-A^2-1-A^{-2}-A^{-5}-A^{-8}.
$$
It is easy to see that $\Theta_1$ is equivalent to $\widetilde{\Theta}$ and  $\Theta_4$ is equivalent to its mirror image. Hence $Y(\Theta_1; A) = Y(\widetilde{\Theta};A)$ and $Y(\Theta_4; A) = Y (\widetilde{\Theta}; A^{-1})$. Therefore, 
$$
\begin{gathered} 
     \widetilde{\mathfrak{J}}(\Theta_1)=A^8\mathfrak{J}(\widetilde{\Theta})=\dfrac{A^8}{\varphi^2}(-A^{28}+A^{20}+A^{12}+A^{8}+1+A^{-8}+A^{-12}+A^{-20}+A^{-32})\cr
     =\dfrac{1}{\varphi^2}(-A^{36}+A^{28}+A^{20}+A^{16}+A^8+1+A^{-12}+A^{-24}) 
\end{gathered}      
$$
and 
$$
\begin{gathered} 
\widetilde{\mathfrak{J}}(\Theta_4) = \dfrac{A^{-8}}{\varphi^2}(A^{32}+A^{20}+A^{8}+1+A^{-8}+A^{-12}+A^{-20}-A^{-28}) \cr
\qquad    =\dfrac{1}{\varphi^2}(A^{24}+A^{12}+1+A^{-8}+A^{-16}+A^{-20}+A^{-28}-A^{-36}).
\end{gathered}     
$$
Summarizing we get 
$$
\begin{gathered} 
    \sum\limits_{i=1}^6 \widetilde{\mathfrak{J}}(\Theta_i) = \dfrac{1}{\varphi^2}(-A^{36}+A^{28}+A^{24}+A^{20}+A^{16}+A^{12}+5A^8+4A^4+10 \cr
\qquad  + 4A^{-4}+5A^{-8}+A^{-12}+A^{-16}+A^{-20}+A^{-24}+A^{-28}-A^{-36}).
\end{gathered}
$$

Now let us discuss subknots of $\Omega_7$ and the associated link $\mathcal L$. 
By calculating $w_{ij}$, $i,j=1,\dots,6$, we obtain $w_{23}=w_{35}=1$, $w_{26}=w_{56}=-1$, and $w_{ij} = 0$ for other indices. Therefore  
$$
n_1=-2,  \qquad n_4=2, \qquad n_2=n_3=n_5=n_6=0.
$$
Recall that for a knot $\mathcal K$ we have 
$$
 \widetilde{\mathfrak{J}}(\mathcal{K})=A^{-8\omega(K)}\mathfrak{J}(K)=A^{-8\omega(K)}(\langle K^{(2)}\rangle +\dfrac{1}{\varphi}),
$$
where $K$ is a diagram of $\mathcal K$ and $K^{(2)}$  is its band diagram. It can be seen from Fig.~\ref{fig:example} that cycles $l_{162}$, $l_{243}$, $l_{135}$, $l_{456}$, $l_{4316}$, $l_{4512}$ are trivial knots, so  
$$
 \widetilde{\mathfrak{J}}(l_{162})=\widetilde{\mathfrak{J}}(l_{243})=\widetilde{\mathfrak{J}}(l_{135})=\widetilde{\mathfrak{J}}(l_{456})=\widetilde{\mathfrak{J}}(l_{4316})=\widetilde{\mathfrak{J}}(l_{4512})=-A^2-A^{-2}+\dfrac{1}{\varphi}.
$$ 
Since $l_{6235}$ is a figure-eight knot, the bracket polynomial for its $(2,0)$-cable is equal to 
$$
\langle l_{6235}^{(2)}\rangle=-A^{26}+A^{22}-A^2-A^{-2}+A^{-22}-A^{-26}, 
$$
hence
$$
\widetilde{\mathfrak{J}}(l_{6235})=-A^{26}+A^{22}-A^2-A^{-2}+A^{-22}-A^{-26}+\dfrac{1}{\varphi}.
$$
Summarizing we get 
$$
\begin{gathered} 
    \sum\limits_{j=1}^7 \widetilde{\mathfrak{J}}(l_j) = A^{26}+A^{22}-7A^2-7A^{-2}+A^{-22}-A^{-26}+\dfrac{7}{\varphi} \cr
    = \dfrac{-A^{28}+A^{20}-7A^4-7-7A^{-4}+A^{-20}-A^{-28}}{\varphi}. 
    \end{gathered} 
$$
Since $n_1 = -2$ and $n_4 = 2$, we will obtain the associated link $\mathcal L$ by adding  one positive full twist to $a_4^{(2)}$ and one negative full twist to $a_1^{(2)}$, see Fig.~\ref{fig15}. 
\begin{figure}[t!]
    \includegraphics[width=0.3\textwidth]{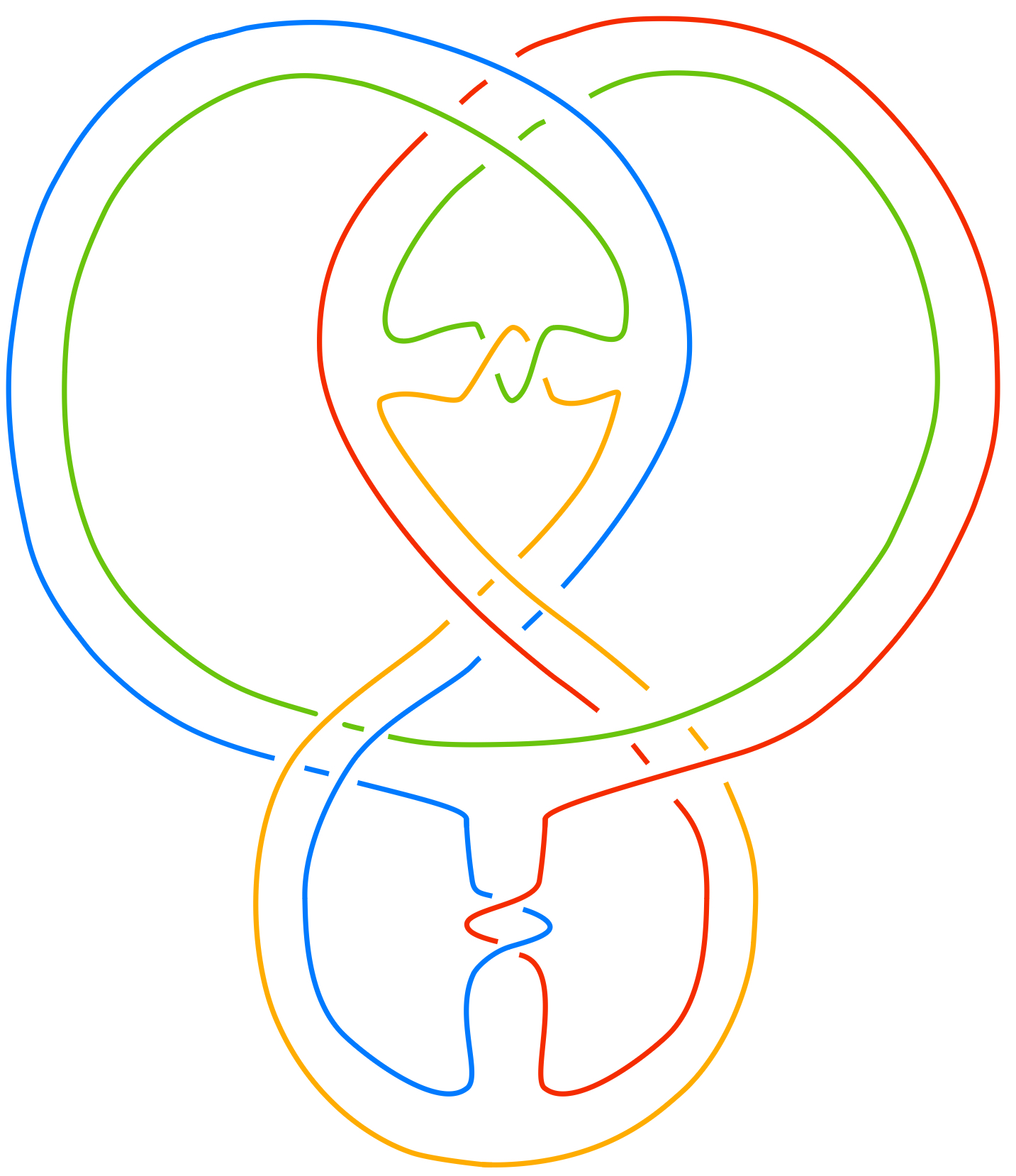}
\caption{Link $\mathcal L.$} \label{fig15}
\end{figure}
Therefore, 
$$
\begin{gathered} 
    V(\mathcal{L})=V(L) = A^{30}-2A^{26}+A^{22}+A^{18}-3A^{14}+3A^{10}-3A^6-2A^2 \cr
\qquad  \qquad    -2A^{-2}-3A^{-6}+3A^{-10}-3A^{-14}+A^{-18}+A^{-22}-2A^{-26}+A^{-30}.
\end{gathered}
$$
Now one can see that the following relation holds:   
$$
\widetilde{\mathfrak{J}}(\Omega)-V(\mathcal{L})=\dfrac{1}{\varphi}\sum\limits_{i=1}^6 \widetilde{\mathfrak{J}}(\Theta_i)-\dfrac{1}{\varphi ^2}\sum\limits_{j=1}^7 \widetilde{\mathfrak{J}}(l_j)+\dfrac{1}{\varphi^3}. 
$$ 
}
\end{example}

%%%

\end{document}